%% file: resolution.tex
\documentclass[10pt,a4paper]{amsart}
\usepackage{a4wide}
\usepackage{hyperref}
\usepackage{enumitem}
\usepackage{amsmath}
\usepackage{amsfonts}
\usepackage{microtype}
\usepackage{multicol}
\usepackage{caption}
\usepackage{subcaption}
\usepackage{xcolor}
\usepackage{stackengine,scalerel}
\usepackage{tikz}
\usepackage{tikz-cd}
\usepackage{amsthm}
\usepackage[capitalise]{cleveref}

\usepackage[utf8]{inputenc}
\usepackage[english]{babel}
\usepackage[oldstyle]{libertine}
\makeatletter
\renewcommand*\libertine@figurestyle{LF}
\makeatother
\usepackage[libertine,cmintegrals,cmbraces]{newtxmath}
\makeatletter
\renewcommand*\libertine@figurestyle{OsF}
\makeatother

\DeclareMathAlphabet{\mathpzc}{OT1}{pzc}{m}{it}

\usepackage[style=alphabetic,backend=biber,maxnames=5,maxalphanames=5,language=english,doi=false,isbn=false,url=false,firstinits=true]{biblatex}
\bibliography{literature.bib}
  
\renewbibmacro{in:}{}
\renewbibmacro*{issue+date}{%
  \ifboolexpr{not test {\iffieldundef{year}} or not test {\iffieldundef{issue}}}
    {\printtext[parens]{%
       \iffieldundef{issue}
         {\usebibmacro{date}}
         {\printfield{issue}%
          \setunit*{\addspace}%
          \usebibmacro{date}}}}
    {}%
  \newunit}

\let\oldtocsection=\tocsection
\let\oldtocsubsection=\tocsubsection
\let\oldtocsubsubsection=\tocsubsubsection
\renewcommand{\tocsection}[2]{\hspace{0em}\oldtocsection{#1}{#2}}
\renewcommand{\tocsubsection}[2]{\hspace{1em}\oldtocsubsection{#1}{#2}}
\renewcommand{\tocsubsubsection}[2]{\hspace{2em}\oldtocsubsubsection{#1}{#2}}

\newtheorem{bigthm}{Theorem}
\newtheorem{bigcor}[bigthm]{Corollary}
\newtheorem{thm}{Theorem}[section]
\newtheorem{lem}[thm]{Lemma}
\newtheorem{cor}[thm]{Corollary}

\theoremstyle{definition}
\newtheorem{dfn}[thm]{Definition}

\theoremstyle{remark}
\newtheorem{ex}[thm]{Example}

\newtheorem{rem}[thm]{Remark}
\newtheorem*{nrem}{Remark}

\setenumerate[1]{leftmargin=*,labelindent=0pt,label=(\roman*),}
\setitemize[1]{label=\raisebox{0.25ex}{\tiny$\bullet$}}

\newcommand\arcsymbol{\mathrel{\ooalign{$-$\cr
  \raise0.15ex\hbox{$\scriptstyle\bullet$}\cr}}}
  
\newcommand\stripsymbol{%
  \mathbin{\scalerel*{\stackinset{r}{-1.5pt}{c}{.3pt}{%
    \def\stackalignment{l}
    \stackon[-.8pt]{\kern0.4pt\rule{2.4pt}{1.2pt}}{\textcolor{white}{\rule{2.4pt}{0.4pt}}}%
  }{$\circ$}\kern1.0pt}{$D$}}%
}

\newcommand\smallsquare{\mathbin{\text{\raise0.17ex\hbox{\scalebox{.7}{$\blacksquare$}}}}}

\newcommand\dslash{/\mkern-6mu/}

\newcommand{\ocolour}[1]{\mathfrak{#1}}

\newcommand{\plus}{\ensuremath{{\scaleobj{0.8}{+}}}}

\newcommand{\aut}{\ensuremath{\operatorname{Aut}}}
\newcommand{\id}{\ensuremath{\operatorname{id}}}
\newcommand{\im}{\ensuremath{\operatorname{im}}}
\newcommand{\braiding}{\ensuremath{b}}
\newcommand{\ob}{\ensuremath{\operatorname{ob}}}
\newcommand{\mor}{\ensuremath{\operatorname{mor}}}
\newcommand{\op}{\ensuremath{\operatorname{op}}}
\newcommand{\oriented}{\ensuremath{\operatorname{or}}}
\newcommand{\braidgroup}{\ensuremath{B}}

\newcommand{\kernel}{\ensuremath{\operatorname{ker}}}
\newcommand{\cokernel}{\ensuremath{\operatorname{coker}}}
\newcommand{\Rel}{\ensuremath{\operatorname{Rel}}}

\newcommand{\fib}{\ensuremath{\operatorname{fib}}}

\newcommand{\RW}{\ensuremath{\operatorname{RW}}}
\newcommand{\mcg}{\ensuremath{\operatorname{mcg}}}

\newcommand{\minimum}{\ensuremath{\operatorname{min}}}

\newcommand{\usr}{\ensuremath{\operatorname{usr}}}

\newcommand{\Path}{\ensuremath{\operatorname{Path}}}

\newcommand{\const}{\ensuremath{\operatorname{const}}}
\newcommand{\PDiff}{\ensuremath{\operatorname{PDiff}}}
\newcommand{\Diff}{\ensuremath{\operatorname{Diff}}}

\newcommand{\catsingle}[1]{\ensuremath{\mathcal{#1}}}
\newcommand{\cat}[1]{\ensuremath{\mathpzc{#1}}}

\newcommand{\oC}{\ensuremath{\operatorname{C}}}

\newcommand{\oH}{\ensuremath{\operatorname{H}}}

\newcommand{\bfH}{\ensuremath{\mathbf{H}}}
\newcommand{\bfN}{\ensuremath{\mathbf{N}}}
\newcommand{\bfR}{\ensuremath{\mathbf{R}}}
\newcommand{\bfZ}{\ensuremath{\mathbf{Z}}}
\newcommand{\bfQ}{\ensuremath{\mathbf{Q}}}

\newcommand{\tE}{\ensuremath{E}}
\newcommand{\tB}{\ensuremath{B}}

\newcommand{\cA}{\ensuremath{\catsingle{A}}}
\newcommand{\cB}{\ensuremath{\catsingle{B}}}
\newcommand{\cC}{\ensuremath{\catsingle{C}}}
\newcommand{\cD}{\ensuremath{\catsingle{D}}}
\newcommand{\cE}{\ensuremath{\catsingle{E}}}
\newcommand{\cF}{\ensuremath{\catsingle{F}}}
\newcommand{\cG}{\ensuremath{\catsingle{G}}}

\newcommand{\cM}{\ensuremath{\catsingle{M}}}
\newcommand{\cN}{\ensuremath{\catsingle{N}}}
\newcommand{\cO}{\ensuremath{\catsingle{O}}}

\newcommand{\UB}{\ensuremath{U\catsingle{B}}}
\newcommand{\CoB}{\ensuremath{\cat{CoB}}}
\newcommand{\CoBM}{\ensuremath{\cat{CoBM}}}
\newcommand{\SC}{\ensuremath{\cat{SC}}}

\newcommand{\TT}{\ensuremath{\cat{Top}}}
\newcommand{\Sets}{\ensuremath{\cat{Sets}}}

\newcommand{\FI}{\ensuremath{\cat{FI}}}
\newcommand{\Ab}{\ensuremath{\cat{Ab}}}
\newcommand{\bfNext}{\ensuremath{\bar{\mathbf{N}}}}

\newcommand{\Conf}{\ensuremath{C}}
\newcommand{\pConf}{\ensuremath{\widetilde{C}}}
\newcommand{\FConf}{\ensuremath{F}}

\newcommand{\Emb}{\ensuremath{\operatorname{Emb}}}
\newcommand{\Maps}{\ensuremath{\operatorname{Maps}}}

\newcommand{\hofib}{\ensuremath{\operatorname{hofib}}}
\newcommand{\semisimpop}{\ensuremath{\Delta^{\operatorname{op}}_{\operatorname{inj}}}}
\newcommand{\semisimp}{\ensuremath{\Delta_{\operatorname{inj}}}}
\newcommand{\semisimpopthick}{\ensuremath{\widetilde{\Delta}^{\operatorname{op}}_{\operatorname{inj}}}}
\newcommand{\semisimpthick}{\ensuremath{\widetilde{\Delta}_{\operatorname{inj}}}}

\newcommand{\colim}{\ensuremath{\operatorname{colim}}}

\newcommand{\hocolim}{\ensuremath{\operatorname{hocolim}}}

\newcommand{\circled}[1]{\raisebox{.5pt}{\textcircled{\raisebox{-.9pt} {#1}}}}

\newcommand{\maptwo}[2]{\ensuremath{#1\longrightarrow#2}}
\newcommand{\maptwoshort}[2]{\ensuremath{#1\rightarrow#2}}
\newcommand{\mapwo}[3]{\ensuremath{#1\colon#2\longrightarrow#3}}

\newcommand{\map}[5]{\ensuremath{#1\colon\begin{array}{rcl} 
      #2 & \longrightarrow & #3 \\[0.3em] 
      #4 & \longmapsto & #5
    \end{array}}}   
\newcommand{\mapwoshort}[3]{\ensuremath{#1\colon#2\rightarrow#3}}
\newcommand{\mapnoname}[4]{\ensuremath{\begin{array}{rcl} 
      #1 & \longrightarrow & #2 \\[0.3em] 
      #3 & \longmapsto & #4
    \end{array}}}
    \newcommand{\mapnonameshort}[4]{\ensuremath{\begin{array}{rcl} 
      #1 & \rightarrow & #2 \\[0.3em] 
      #3 & \mapsto & #4
    \end{array}}}
    
\newcommand{\xlongrightarrow}[1]{\overset{#1}{\longrightarrow}}
\newcommand{\xlongleftarrow}[1]{\overset{#1}{\longleftarrow}}

\begin{document}

\title{Homological stability of topological moduli spaces}
\begin{abstract}
Given a graded $E_1$-module over an $E_2$-algebra in spaces, we construct an augmented semi-simplicial space up to higher coherent homotopy over it, called its canonical resolution, whose graded connectivity yields homological stability for the graded pieces of the module with respect to constant and abelian coefficients. We furthermore introduce a notion of coefficient systems of finite degree in this context and show that, without further assumptions, the corresponding twisted homology groups stabilise as well. This generalises a framework of Randal-Williams and Wahl for families of discrete groups.

In many examples, the canonical resolution recovers geometric resolutions with known connectivity bounds. As a consequence, we derive new twisted homological stability results for e.g.~moduli spaces of high-dimensional manifolds, unordered configuration spaces of manifolds with labels in a fibration, and moduli spaces of manifolds equipped with unordered embedded discs. This in turn implies representation stability for the ordered variants of the latter examples.
\end{abstract}
\author{Manuel Krannich}
\email{\href{mailto:krannich@dpmms.cam.ac.uk}{krannich@dpmms.cam.ac.uk}}
\address{Centre for Mathematical Sciences, University of Cambridge, Wilberforce Road, Cambridge CB3 0WB, United Kingdom}
\subjclass[2010]{55P48, 55R40, 55R80, 57R19}
\maketitle

\addtocontents{toc}{\protect\setcounter{tocdepth}{0}}

 A sequence of spaces  \[\ldots\longrightarrow \cM_{n-1}\longrightarrow \cM_n\longrightarrow \cM_{n+1}\longrightarrow\ldots\] is said to satisfy \emph{homological stability} if the induced maps in homology are isomorphisms in degrees that are small relative to $n$. There is a well-established strategy for proving homological stability that traces back to an argument by Quillen for the classifying spaces of a sequence of inclusions of groups $G_n$. Given simplicial complexes whose connectivity increases with $n$ and on which the groups $G_n$ act simplicially, transitively on simplices, and with stabilisers isomorphic to groups $G_{n-k}$ prior in the sequence, stability can often be derived by employing a spectral sequence relating the different stabilisers. In \cite{RWW}, Randal-Williams and Wahl axiomatised this strategy of proof, resulting in a convenient categorical framework for proving homological stability for families of discrete groups that form a braided monoidal groupoid. Their work unifies and improves many classical stability results and has led to a number of applications since its introduction \cite{Friedrich, GW, PatztWu, RWtwisted,SW}.
 
However, homological stability phenomena have been proved to occur not only in the context of discrete groups, but also in numerous non-aspherical situations, many of them of a moduli space flavour, such as unordered configuration spaces of manifolds \cite{McDuffConf,SegalConf,SegalConfII}, the most classical example, or moduli spaces of high-dimensional manifolds \cite{GRWII,GRWI} to emphasise a more recent one. The majority of the stability proofs in this context resemble the original line of argument for discrete groups, and one of the objectives of the present work is to provide a conceptualisation of this pattern.

Instead of considering the single spaces $\cM_n$ and the maps $\maptwoshort{\cM_n}{\cM_{n+1}}$ between them one at a time, it is beneficial to treat them as a single space $\cM=\coprod_{n\ge0}\cM_n$ together with a \emph{grading} $\mapwoshort{g_\cM}{\cM}{\bfN_0}$ to the nonnegative integers, capturing the decomposition of $\cM$ into the pieces $\cM_n$, and a \emph{stabilisation map} $\mapwoshort{s}{\cM}{\cM}$ that restricts to the maps $\maptwoshort{\cM_n}{\cM_{n+1}}$, so it increases the degree by one. From the perspective of homotopy theory, such $\cM$ that result from families $\cM_n$ that are known to stabilise homologically usually share the characteristic of forming a \emph{(graded) $\tE_1$-module over an $\tE_2$-algebra}---the homotopy theoretical analogue of a module over a braided monoidal category. This observation is the driving force behind the present work.

Referring to \cref{section:modules} for a precise definition, we encourage the reader to think of a graded $\tE_1$-module $\cM$ over an $\tE_2$-algebra $\cA$ as a pair of spaces $(\cM,\cA)$ together with gradings $\mapwoshort{g_\cM}{\cM}{\bfN_0}$ and $\mapwoshort{g_\cA}{\cA}{\bfN_0}$, a homotopy-commutative multiplication $\mapwoshort{\oplus}{\cA\times\cA}{\cA}$, and a homotopy-associative action-map $\mapwoshort{\oplus}{\cM\times \cA}{\cM}$. These are required to satisfy various axioms, among them additivity with respect to $g_\cM$ and $g_\cA$ (see \cref{definition:E1module}). Given such $\cM$ and $\cA$, the choice of a \emph{stabilising object} $X\in\cA$, meaning an element of degree $1$, results in a \emph{stabilisation map} \[\mapwoshort{s\coloneq(-\oplus X)}{\cM}{\cM}\] that increases the degree by $1$ and hence gives rise to a sequence
\[\ldots\xlongrightarrow{} \cM_{n-1}\xlongrightarrow{s} \cM_n\xlongrightarrow{s} \cM_{n+1}\xlongrightarrow{}\ldots\] of the subspaces $\cM_n=g_\cM^{-1}(n)$ of a fixed degree. The sequences of spaces arising in this fashion are the ones whose homological stability behaviour the present work is concerned with. 

The key construction of this work is introduced in \cref{section:complexes}. We assign to $\cM$ its \emph{canonical resolution} \begin{equation}\label{equation:canonicalresolution}\maptwo{R_\bullet(\cM)}{\cM},\end{equation} which is an augmented semi-simplicial space up to higher coherent homotopy---a notion made precise in \cref{section:semisimplicialthick}, but which can be thought of as an augmented semi-simplicial space in the usual sense. The fibre $W_\bullet(A)$ of the canonical resolution at a point $A\in\cM$ is an analogue of the simplicial complex in Quillen's argument; it is a semi-simplicial space up to higher coherent homotopy whose space of $p$-simplices $W_p(A)$ is the homotopy fibre at $A$ of the $(p+1)$st iterated stabilisation map $\mapwoshort{s^{p+1}}{\cM}{\cM}$. Thus $W_\bullet(A)$ should be thought of as the \emph{space of destabilisations of $A$}---a terminology that suggests that the canonical resolution controls the stability behaviour of $\cM$, justified by Theorems~\ref{theorem:constant} and~\ref{theorem:twisted}.

To state our main theorems, we call the canonical resolution of $\cM$ \emph{graded $\varphi(g_\cM)$-connected in degrees $\ge m$} for a function $\mapwoshort{\varphi}{\bfN_0}{\bfQ}$ if the restriction $\maptwoshort{|R_\bullet(\cM)|_n}{\cM_n}$ of the geometric realisation of \eqref{equation:canonicalresolution} to the preimage of $\cM_n$ is $\lfloor \varphi(n)\rfloor$-connected in the usual sense for all $n\ge m$. The first theorem, proved in \cref{section:proofconstant}, treats homological stability with constant and abelian coefficients, the latter being local systems on which the commutator subgroups of the fundamental groups at all basepoints act trivially.

\begin{bigthm}\label{theorem:constant}
Let $\cM$ be a graded $\tE_1$-module over an $\tE_2$-algebra with stabilising object $X$ and $L$ a local system on $\cM$. If the canonical resolution of $\cM$ is graded $(\frac{g_\cM-2+k}{k})$-connected in degrees $\ge1$ for some $k\ge 2$, then \[\mapwo{s_*}{\oH_i(\cM_{n};s^*L)}{\oH_i(\cM_{n+1};L)}\] 
\begin{enumerate}
\item is an isomorphism for $i\le\frac{n-1}{k}$ and an epimorphism for $i\le\frac{n-2+k}{k}$, if $L$ is constant, and 
\item is an isomorphism for $i\le\frac{n+1-k}{k}$ and an epimorphism for $i\le\frac{n}{k}$, if $L$ is abelian and $k\ge3$.
\end{enumerate}
\end{bigthm}

\begin{nrem}
In certain cases, discussed in \cref{remark:improvement}, the ranges of \cref{theorem:constant} can be improved marginally. \end{nrem}

Restricting to homological degree $0$, the theorem has the following cancellation result as a consequence.

\begin{bigcor}\label{corollary:cancellation}
Let $\cM$ be a graded $\tE_1$-module over an $\tE_2$-algebra with stabilising object $X$. If the connectivity assumption of \cref{theorem:constant} is satisfied, then the fundamental groupoid of $\cM$ is $X$-cancellative for objects of positive degree, i.e.~for objects $A$ and $A'$ of $\cM$ of positive degree, $A\oplus X\cong A'\oplus X$ in $\Pi(\cM)$ implies $A\cong A'$.
\end{bigcor}

To cover more general coefficients, we note that the fundamental groupoid of an $\tE_2$-algebra $\cA$ naturally carries the structure of a braided monoidal category $(\Pi(\cA),\oplus,\braiding,0)$ and the fundamental groupoid of an $\tE_1$-module $\cM$ over $\cA$ becomes a right-module $(\Pi(\cM),\oplus)$ over it (see \cref{section:modules}). In terms of this, we define in \cref{section:coefficientsystems} a \emph{coefficient system $F$ for $\cM$} with stabilising object $X$ as an abelian group-valued functor $F$ on $\Pi(\cM)$, together with a natural transformation $\mapwoshort{\sigma^F}{F}{F(-\oplus X)}$ for which the image of the canonical morphism $\maptwoshort{B_m}{\aut_\cA(X^{\oplus m})}$ from the braid group on $m$ strands acts trivially on the image of $\mapwoshort{(\sigma^F)^m}{F}{F(-\oplus X^{\oplus m})}$ for all $n$ and $m$. Such a coefficient system enhances the stabilisation map to a map of spaces with local systems
\[\mapwoshort{(s;\sigma^F)}{(\cM_n;F)}{(\cM_{n+1};F)}\] by restricting $F$ to subspaces of homogenous degree. A coefficient system $F$ induces a new one $\Sigma F=F(-\oplus X)$, called its \emph{suspension}, which comes with a morphism $\maptwoshort{F}{\Sigma F}$, named the \emph{suspension map} (see \cref{definition:suspension}). The coefficient system $F$ is inductively said to be \emph{of degree $r$} if the kernel of the suspension map vanishes and the cokernel has degree $(r-1)$; the zero coefficient system having degree $-1$. In fact, we define a more general notion of being of \emph{(split) degree $r$ at $N$} such that $F$ is of degree $r$ in the sense just described if it is of degree $r$ at $0$ (see \cref{definition:coefficientsystems}). This notion of a coefficient system of finite (split) degree generalises the one introduced by Randal-Williams and Wahl \cite{RWW} for braided monoidal groupoids (see Remarks~\ref{remark:RWWcoefficients} and~\ref{remark:quillencoefficients}), which was itself inspired by work of Dwyer \cite{Dwyer} and van der Kallen \cite{vanderKallen} on general linear groups, and work of Ivanov \cite{Ivanov} on mapping class groups of surfaces.

\begin{nrem}There is an alternative point of view on coefficient systems for $\cM$, namely as abelian-group valued functors on a category $\langle\cM,\cB\rangle$ constructed from the action of $\Pi(\cA)$ on $\Pi(\cM)$ (see \cref{remark:quillencoefficients}).
\end{nrem}

Our second main theorem, demonstrated in \cref{section:twistedstability}, addresses homological stability of $\cM$ with coefficients in a coefficient system of finite degree.

\begin{bigthm}\label{theorem:twisted}
Let $\cM$ be a graded $\tE_1$-module over an $\tE_2$-algebra with stabilising object $X$ and $F$ a coefficient system for $\cM$ of degree $r$ at $N\ge0$. If the canonical resolution of $\cM$ is graded $(\frac{g_\cM-2+k}{k})$-connected in degrees $\ge1$ for some $k\ge 2$, then the map induced by stabilisation\[\mapwo{(s;\sigma^F)_*}{\oH_i(\cM_n;F)}{\oH_i(\cM_{n+1};F)}\] is an isomorphism for  $i\le\frac{n-rk-k}{k}$ and an epimorphism for $i\le\frac{n-rk}{k}$, when $n>N$. If $F$ is of split degree $r$ at $N\ge0$ then $(s;\sigma^F)_*$ is an isomorphism for  $i\le\frac{n-r-k}{k}$ and an epimorphism for $i\le\frac{n-r}{k}$, when $n>N$.
\end{bigthm}

As a proof of concept, we apply the developed theory to three main classes of examples to which we devote the remainder of this introduction.

\subsection*{Configuration spaces}
The \emph{unordered configuration space} $\Conf_n^\pi(W)$ of a manifold with boundary $W$ with labels in a Serre fibration $\mapwoshort{\pi}{E}{W}$ is the quotient of the \emph{ordered configuration space}
\[\FConf^{\pi}_n(W)=\{(e_1,\ldots,e_n)\in E^n\mid\pi(e_i)\neq\pi(e_j)\text{ for }i\neq j\text{ and }\pi(e_i)\in W\backslash\partial W\}\] by the apparent action of the symmetric group $\Sigma_n$. If $W$ is of dimension $d\ge2$ and has nonempty boundary, then the union of its configuration spaces $\cM=\coprod_{n\ge0}\Conf_n^\pi(W)$ admits the structure of an $\tE_1$-module over the $\tE_2$-algebra $\cA=\coprod_{n\ge0}\Conf_n(D^d)$ of configurations in a $d$-disc, graded by the number of points (see \cref{lemma:moduleconfigurations}). In \cref{section:arcresolution}, we identify its canonical resolution with the \emph{resolution by arcs}---an augmented semi-simplicial space of geometric nature that has already been considered in the context of homological stability (see e.g.~\cite{MillerWilson,KM}) and is known to be sufficiently connected to apply Theorems~\ref{theorem:constant} and~\ref{theorem:twisted}.

\begin{bigthm}\label{theorem:configurationspaces}Let $W$ be a connected manifold of dimension at least $2$ with nonempty boundary and let $\mapwoshort{\pi}{E}{W}$ be a Serre fibration with path-connected fibres.\begin{enumerate}
\item For a local system $L$ on $\Conf^\pi_{n+1}(W)$, the stabilisation map \[\mapwo{s_*}{\oH_i(\Conf^\pi_n(W);s^*L)}{\oH_i(\Conf^\pi_{n+1}(W);L)}\] is an isomorphism for $i\le\frac{n-1}{2}$ and an epimorphism for $i\le\frac{n}{2}$, if $L$ is constant. It is an isomorphism for $i\le\frac{n-2}{3}$ and an epimorphism for $i\le\frac{n}{3}$, if $L$ is abelian.
\item If $F$ is a coefficient system of degree $r$ at $N\ge0$, then the stabilisation map \[\mapwo{(s;\sigma^F)_*}{\oH_i(\Conf^\pi_n(W);F)}{\oH_i(\Conf^\pi_{n+1}(W);F)}\] is an isomorphism for $i\le\frac{n-2r-2}{2}$ and an epimorphism for $i\le\frac{n-2r}{2}$, when $n>N$. If $F$ is of split degree $r$ at $N\ge0$, then it is an isomorphism for $i\le\frac{n-r-2}{2}$ and an epimorphism for $i\le\frac{n-r}{2}$, when $n>N$.
\end{enumerate} 
\end{bigthm}

\begin{nrem}Employing the improvement of \cref{remark:improvement}, one obtains for constant coefficients a slightly better isomorphism range of $i\le\frac{n}{2}$ than the one stated in \cref{theorem:configurationspaces}.
\end{nrem}

Configuration spaces have a longstanding history in the context of homological stability, starting with work of Arnold \cite{Arnold}, who established stability for $\Conf_n(D^2)$ with constant coefficients. McDuff and Segal \cite{McDuffConf, SegalConf, SegalConfII} observed that this behaviour is not restricted to the $2$-disc and proved stability for more general $\Conf^{\pi}_n(W)$ with constant coefficients and $\pi=\id_W$, which can be extended to general $\pi$, e.g.~by adapting the proof for a trivial fibration presented in \cite{RWConf} (see \cite{CanteroPalmer,KM} for alternative proofs). 

As proved for example in \cite{RWConf}, the stabilisation map for configuration spaces is in fact split injective in homology with constant coefficients in all degrees---a phenomenon special to configuration spaces, not captured by our general approach.

For a trivial fibration, stability of $\Conf^{\pi}_n(W)$ with respect to a nontrivial coefficient system $F$ was studied by Palmer \cite{Palmer}, building on work of Betley \cite{Betley} on symmetric groups. The second part of \cref{theorem:configurationspaces} extends his result to nontrivial fibrations and a significantly larger class of coefficient systems, partly conjectured by Palmer \cite[Rem.\,1.5]{Palmer} (see \cref{remark:comparisonwithpalmer} for a more detailed comparison to his work). In the case of surfaces and a trivial fibration, a result similar to \cref{theorem:configurationspaces}, but with respect to a slightly smaller class of coefficient systems, is contained in work by Randal-Williams and Wahl \cite[Thm\,D]{RWW}.

In \cref{section:coefficientsystemsconfigurationspaces}, we provide a discussion of coefficient systems for configuration spaces by relating them, for instance, to the theory of $\FI$-modules as introduced by Church, Ellenberg, and Farb \cite{ChurchEllenbergFarb} or to coefficient systems studied in \cite{RWW}. These considerations provide numerous nontrivial coefficient systems $F$ with respect to which the homology of $\Conf^{\pi}_n(W)$ stabilises.

To our knowledge, stability with abelian coefficients for configuration spaces of manifolds of dimensions greater than two has not been considered so far. We next discuss a direct consequence of stability with respect to this class of coefficients as the first item in a series of applications exploiting \cref{theorem:configurationspaces}.

\subsubsection*{Oriented configuration spaces}
The \emph{oriented configuration space} $\Conf_n^{\pi,\oriented}(W)$ with labels in a Serre fibration $\pi$ over $W$ is the double cover of $\Conf_n^{\pi}(W)$ given as the quotient of the ordered configuration space $\FConf_n^{\pi}(W)$ by the action of the alternating group $A_n$, or equivalently, the space of labelled configurations ordered up to even permutations. By the space version of Shapiro's lemma, the homology of $\Conf_{n}^{\pi,\oriented}(W)$ is isomorphic to $\oH_*(\Conf_{n}^{\pi}(W);\bfZ[\bfZ/2\bfZ]),$ with the action of $\pi_1(\Conf_{n}^{\pi}(W))$ on the group ring $\bfZ[\bfZ/2\bfZ]$ being induced by the composition of the sign homomorphism with the morphism $\pi_1(\Conf_{n}^{\pi}(W))\rightarrow\Sigma_{n}$, obtained by choosing an ordering of a basepoint. These local systems are abelian and are preserved by pulling back along the stabilisation map, hence homological stability for $\Conf_n^{\pi,\oriented}(W)$ follows as a by-product of \cref{theorem:configurationspaces}.

\begin{bigcor}\label{corollary:orientedconfigurationspaces}Let $W$ and $\pi$ be as in \cref{theorem:configurationspaces}. The map induced by stabilisation \[\mapwo{s_*}{\oH_i(\Conf^{\pi,\oriented}_{n}(W);\bfZ)}{\oH_i(\Conf^{\pi,\oriented}_{n+1}(W);\bfZ)}\] is an isomorphism for $i\le\frac{n-2}{3}$ and an epimorphism for $i\le\frac{n}{3}$.
\end{bigcor}

Stability for oriented configuration spaces of connected orientable surfaces with nonempty boundary and without labels was proved by Guest, Kozlowsky, and Yamaguchi \cite{GuestKozlowskyYamaguchi} using computations due to Bödigheimer, Cohen, Taylor, and Milgram \cite{BoedigheimerCohenMilgram,BoedigheimerCohenTaylor}. Palmer \cite{PalmerOriented} extended this to manifolds of higher dimensions with nonempty boundary and labels in a trivial fibration. \cref{corollary:orientedconfigurationspaces} gives an alternative proof of his result and enhances it by means of general labels and an improved stability range.
 
\subsubsection*{Configuration spaces of embedded discs}The \emph{configuration space $\Conf^k_n(W)$ of unordered $k$-discs} in a connected $d$-manifold $W$ is the quotient by the action of $\Sigma_n$ on the \emph{configuration space of ordered $k$-discs} \[\textstyle{\FConf^k_n(W)=\Emb(\coprod^n D^k,W\backslash\partial W)},\] equipped with the $\cC^\infty$-topology. For $k=d$ and oriented $W$, there are variants $\FConf_n^{d^{\plus}}(W)$ and $\Conf^{d^{\plus}}_n(W)$ by restricting to orientation preserving embeddings. Mapping an embedding of a $k$-disc to its centre point, labelled with the $k$-frame induced by standard framing of $D^k$ at the origin, results in a map $\maptwoshort{\Conf^k_n(W)}{\Conf^{\pi_k}_n(W)}$, where $\pi_k$ is the bundle of $k$-frames in $M$. This map can be seen to be a weak equivalence by choosing a metric and exponentiating frames. For $k<d$, the fibre of $\pi_k$ is path-connected, so the homological stability results of \cref{theorem:configurationspaces} carry over to $\Conf^k_n(W)$, comprising part of \cref{corollary:configurationspacesofdiscs} below. Using the bundle $\pi_d^{\plus}$ of oriented $d$-framings, the argument for $\Conf^{d^{\plus}}_n(W)$ is analogous, since the orientability condition ensures that the fibres of $\pi_k^{\plus}$ are path-connected.

The topological group of diffeomorphisms $\Diff_\partial(W)$ fixing a neighbourhood of the boundary in the $\cC^{\infty}$-topology naturally acts on the configuration spaces $\FConf^k_n(W)$ and $\Conf^k_n(W)$, and the resulting homotopy quotients $\FConf^k_n(W)\dslash\Diff_\partial(W)$ and $\Conf^k_n(W)\dslash\Diff_\partial(W)$ model the classifying spaces of the subgroups \[\PDiff^k_{\partial,n}(W)\subseteq\Diff^k_{\partial,n}(W)\subseteq\Diff_\partial(W),\] where $\PDiff^k_{\partial,n}(W)$ are the diffeomorphisms that fix $n$ chosen embedded $k$-discs in $W$ and $\Diff^k_{\partial,n}(W)$ are the ones permuting them (see \cref{lemma:diffeospermutingdiscs}). If $W$ is orientable, the (sub)groups of orientation preserving diffeomorphisms are denoted with a (+)-superscript. In \cref{example:groupaction}, we explain how the canonical resolution of a graded $\tE_1$-module $\cM$ over an $\tE_2$-algebra $\cA$ relates to that of the $\tE_1$-module $EG\times_G\cM$ over $\cA$ in the presence of a graded action of a group $G$ on $\cM$ that commutes with the action of $\cA$. An application of this consideration to the situation at hand implies the following, carried out in \cref{section:embeddeddiscs}.

\begin{bigcor}\label{corollary:configurationspacesofdiscs}\label{corollary:diffeomorphismspermutingdiscs}Let $W$ be a $d$-dimensional manifold as in \cref{theorem:configurationspaces} and let $0\le k< d$.
\begin{enumerate}
\item For a local system $L$, the stabilisation maps
\[\maptwoshort{\oH_i(\Conf^k_n(W);s^*L)}{\oH_i(\Conf^k_{n+1}(W);L)}\quad\text{and}\quad\maptwoshort{\oH_i(\tB \Diff_{\partial,n}^k(W);s^*L)}{\oH_i(\tB \Diff_{\partial,n+1}^k(W);L)}\] are isomorphisms for $i\le\frac{n-1}{2}$ and epimorphisms for $i\le\frac{n}{2}$, if $L$ is constant. If $L$ is abelian, then they are isomorphisms for $i\le\frac{n-2}{3}$ and epimorphisms for $i\le\frac{n}{3}$.
\item If $F$ is a coefficient system of degree $r$ at $N\ge0$, then the maps induced by the stabilisation $(s;\sigma^F)$ \[\maptwoshort{\oH_i(\Conf^k_n(W);F)}{\oH_i(\Conf^k_{n+1}(W);F)}\quad\text{and}\quad\maptwoshort{\oH_i(\tB \Diff_{\partial,n}^k(W);F)}{\oH_i(\tB \Diff_{\partial,n+1}^k(W);F)}\] are isomorphisms for $i\le\frac{n-2r-2}{2}$ and epimorphisms for $i\le\frac{n-2r}{2}$, when $n>N$. If $F$ is of split degree $r$ at $N\ge0$, then they are isomorphisms for $i\le\frac{n-r-2}{2}$ and epimorphisms for $i\le\frac{n-r}{2}$, when $n>N$.
\end{enumerate}
If $W$ is oriented, then the analogous statements hold for the variants $\Conf^{d,\plus}_n(W)$ and $\tB\Diff_{\partial,n}^{d,\plus}(W)$.
\end{bigcor}

\begin{nrem}
The isomorphism range for constant coefficients in the previous theorem can be improved to $i\le\frac{n}{2}$ by virtue of \cref{remark:improvement}.\end{nrem}

For compact manifolds $W$, Tillmann \cite{Tillmann} has proved homological stability with constant coefficients for variants of $\tB \Diff^0_{\partial,n}(W)$ and $\tB \Diff^{d,\plus}_{\partial,n}(W)$ involving diffeomorphisms that are only required to fix a disc in the boundary instead of the whole boundary. A Serre spectral sequence argument shows that stability for these variants follows from stability of the spaces $\tB\Diff^{0}_{\partial,n}(W)$ and $\tB\Diff^{d,\plus}_{\partial,n}(W)$. Hatcher and Wahl \cite[Prop.\,1.5]{HatcherWahl} have shown stability with constant coefficients for the mapping class groups $\pi_0(\Diff^{0}_{\partial,n}(W))$, which can be seen to be equivalent to $\Diff^{0}_{\partial,n}(W)$ for compact $2$-dimensional $W$ as a result of the homotopy discreteness of the space of diffeomorphisms of a compact surface \cite{EarleEells,Gramain}. In this case, stability with respect to some of the twisted coefficient systems \cref{corollary:configurationspacesofdiscs} deals with is contained in work by Randal-Williams and Wahl \cite[Thm\,5.22]{RWW}.

\subsubsection*{Representation stability}
The first rational homology group of the ordered configuration space of the $2$-disc \[\oH_1(\FConf_n(D^2);\bfQ)\cong \bfQ^{n\choose2},\]as e.g.~computed in \cite{ArnoldPure}, exemplifies that---in contrast to unordered configuration spaces---the homology of the ordered variant does not stabilise. However, by incorporating the action of the symmetric groups $\Sigma_n$, it does stabilise in a more refined, representation theoretic sense. To make this precise, recall the correspondence between irreducible representations of $\Sigma_n$ and partitions of $n$ \cite[Ch.\,4]{FultonHarris}. We denote the irreducible $\Sigma_{|\lambda|}$-module corresponding to a partition $\lambda=(\lambda_1\ge\ldots\ge\lambda_k)\vdash |\lambda|$ of $|\lambda|$ by $V_\lambda$ and define for $n\ge |\lambda|+\lambda_1$, the \emph{padded partition} $\lambda[n]=(n-|\lambda|\ge\lambda_1\ge\ldots\ge\lambda_k)\vdash n$. Using the Totaro spectral sequence \cite{Totaro}, Church \cite{Church} has shown that for a connected orientable manifold of dimension at least two with finite-dimensional rational cohomology, the groups $\oH^i(\FConf_n(W);\bfQ)$ are \emph{uniformly representation stable}---a concept introduced by Church and Farb \cite{ChurchFarb}. This implies the existence of a constant $N(i)$, depending solely on $i$, such that the multiplicity of $V_{\lambda[n]}$ in the $\Sigma_n$-module $\oH^i(\FConf_n(W);\bfQ)$ is independent of $n$ for $n\ge N(i)$. Church's result has been extended in several directions \cite{ChurchEllenbergFarb, MillerWilson,Petersen,Tosteson}.

A twisted Serre spectral sequence argument (see \cref{lemma:multiplicityistwistedhomology}) shows that the multiplicity of an irreducible $\Sigma_n$-module $V_\lambda$ in $\oH^i(\FConf^\pi_n(W);\bfQ)$ agrees with the dimension of $\oH_i(\Conf^\pi_n(W);V_\lambda)$, where $\pi_1(\Conf^\pi_n(W))$ acts on $V_\lambda$ via the morphism $\maptwoshort{\pi_1(\Conf^\pi_n(W))}{\Sigma_n}$. This fact allows us to derive the stability of these multiplicities from \cref{theorem:configurationspaces}, at least for all manifolds to which the latter theorem applies (see \cref{section:representationstability}).

\begin{bigcor}\label{corollary:representationstability}
Let $W$ and $\pi$ be as in \cref{theorem:configurationspaces} and let $Z_n$ be one of the following sequences of $\Sigma_n$-spaces: 
{\setlength\multicolsep{3pt}\begin{multicols}{2}
\begin{enumerate}
\item $\FConf_n^\pi(W)$,
\item $\FConf_n^k(W)$ for $0\le k<d$,
\item $\FConf_n^{d,\plus}(W)$ if $W$ is oriented,
\item $\tB\PDiff^{k}_{\partial,n}(W)$ for $0\le k<d$,
\item $\tB\PDiff^{d,\plus}_{\partial,n}(W)$ if $W$ is orientable.
\end{enumerate}
\end{multicols}}
The $V_{\lambda[n]}$-multiplicity in $\oH^i(Z_n;\bfQ)$ for a fixed partition $\lambda$ is independent of $n$ for $n$ large relative to $i$.
\end{bigcor}

In \cref{remark:representationstability}, we discuss explicit ranges for \cref{corollary:representationstability} and compare them to Church's. Let us at this juncture record that our approach leads to ranges that depend on $|\lambda|$, so we do not recover \emph{uniform} representation stability. On the other hand, in contrast to Church's result, we neither require $W$ to be orientable nor to have finite dimensional rational cohomology or $\pi$ to be the identity.

Jiménez Rolland \cite{JimenezI, JimenezII} has shown uniform representation stability for the cohomology groups $\oH^i(\tB\PDiff^0_{\partial,n}(W);\bfQ)$ for compact orientable surfaces and for compact connected manifolds $W$ of dimension $d\ge 3$, assuming that $\tB \Diff_\partial(W)$ has the homotopy type of a CW-complex of finite type. Furthermore, she proved uniform representation stability for $\pi_0(\PDiff^0_{\partial,n}(W))$ for compact orientable surfaces, as well as for higher-dimensional manifolds under some further assumptions.

\subsection*{Moduli spaces of manifolds}
The moduli space $\cM$ of compact $d$-dimensional smooth manifolds with a fixed boundary $P$ forms an $\tE_1$-module over the $\tE_d$-algebra $\cA$ given by the moduli space of compact $d$-manifolds with a sphere as boundary (see \cref{lemma:modulestructuremanifolds}). The homotopy types of $\cM$ and $\cA$ are \[\textstyle{\cM\simeq\coprod_{[W]}\tB\Diff_\partial(W)\quad\text{and}\quad\cA\simeq\coprod_{[N]}\tB\Diff_\partial(N)},\] where $[W]$ runs over diffeomorphism classes relative to $P$ of compact $d$-manifolds with $P$-boundary and $[N]$ over the ones of compact $d$-manifolds with a sphere as boundary. Acting with a manifold $X\in\cA$ on $\cM$ corresponds to taking the boundary connected sum $(-\natural X)$ with $X$, so the resulting stabilisation map thus restricts on path components to a map of the form \begin{equation}\label{equation:manifoldstabilisationintroduction}\mapwoshort{s}{\tB \Diff_\partial(W)}{\tB\Diff_\partial(W\natural X)},\end{equation} which models the map on classifying spaces induced by extending diffeomorphisms by the identity. 

As shown in \cref{section:resolutionbyembeddings}, the canonical resolution of $\cM$ with respect to a choice of a stabilising manifold $X$ is equivalent to the \emph{resolution by embeddings}---an augmented semi-simplicial space of submanifolds $W\in\cM$, together with embeddings of $X$ with a fixed behaviour near their boundary. For specific manifolds $X$ and $W$, this resolution and its connectivity has been studied to prove homological stability of \eqref{equation:manifoldstabilisationintroduction}, first by Galatius and Randal-Williams in their work \cite{GRWI} for $X\cong D^{2p}\sharp (S^p\times S^p)$ and simply-connected $2p$-dimensional $W$ with $p\ge 3$. Their results extend the classical stability result for mapping class groups of surfaces by Harer \cite{Harer} to higher dimensions. As in Harer's theorem, the known connectivity of the resolution by embeddings, and hence the resulting stability ranges, depend on the \emph{$X$-genus} of $W$, \[g^X(W)=\max\{k\ge0\mid\text{there exists $M\in\cM$ such that }M\natural X^{\natural k}\cong W\text{ relative to $P$}\},\] which incidentally provides a method of grading $\tE_1$-modules $\cM$ in general (see \cref{section:stablegenus}). Perlmutter \cite{Perlmutter} succeeded in carrying out this strategy for $X\cong D^{p+q}\sharp (S^p\times S^q)$ with certain $p\neq q$ depending on which $W$ is required to satisfy a connectivity assumption. Recently, Friedrich \cite{Friedrich} extended the work of Galatius and Randal-Williams to manifolds $W$ with nontrivial fundamental group in terms of the \emph{unitary stable rank} \cite[Def.\,6.3]{MirzaiiVanderkallen} of the group ring $\bfZ[\pi_1(W)]$. These connectivity results can be restated in our context as graded connectivity for the canonical resolution of $\cM$ with respect to different gradings (see \cref{lemma:canonicalresolutionconnectedsmanifolds}), allowing us to apply Theorems~\ref{theorem:constant} and~\ref{theorem:twisted}. 

Employing the improvement of \cref{remark:improvement}, the ranges with constant and abelian coefficients obtained from \cref{theorem:constant} agree with the ones established in \cite{Friedrich,GRWI,Perlmutter} (after extending \cite{Perlmutter} to abelian coefficients by adapting the methods of \cite{GRWI}). The cancellation result for connected sums of manifolds that we derive from \cref{corollary:cancellation} coincides with their cancellation results as well. Our main contribution with respect moduli spaces of manifolds lies in the application of \cref{theorem:twisted}, i.e.~homological stability with respect to a large class of nontrivial coefficient systems, which has not yet been considered in the context of moduli spaces of high-dimensional manifolds. On path components, it reads as follows.

\begin{bigthm}\label{theorem:manifolds}Let $W$ be a compact $(p+q)$-manifold with nonempty boundary and $F$ a coefficient system of degree $r$. Denote by $g(W)$ the $(S^p\times S^q)$-genus of $W$, and set $u$ to be $1$ if $W$ is simply connected and to be the unitary stable rank of $\bfZ[\pi_1(W)]$ otherwise. The stabilisation map \[\mapwo{(s,\sigma^F)_*}{\oH_i(\tB\Diff_\partial(W);F)}{\oH_i(\tB\Diff_\partial(W\sharp(S^p\times S^q));F)}\] 
\begin{enumerate}
\item is an isomorphism for $i\le\frac{g(W)-2r-u-3}{2}$ and an epimorphism for $i\le\frac{g(W)-2r-u-1}{2}$, if $p=q\ge3$, and 
\item an isomorphism for $i\le\frac{g(W)-2r-m-4}{2}$ and an epimorphism for $i\le\frac{g(W)-2r-m-2}{2}$, if $W$ is $(q-p+2)$-connected and $0<p<q<2p-2$ with $m=\minimum\{i\in\bfN_0\mid\text{there exists an epimorphism }\maptwoshort{\bfZ^i}{\pi_q(S^p)}\}$.
\end{enumerate}
\end{bigthm}
If $F$ is (split) of degree $r$ at a number $N\ge 0$, the ranges in \cref{theorem:manifolds} change as per \cref{theorem:twisted}.

\begin{nrem}
The unitary stable rank \cite[Def.\,6.3]{MirzaiiVanderkallen} of a group ring $\bfZ[G]$ need not be finite. To provide a class of examples of finite unitary stable rank, recall that $G$ is called \emph{virtually polycyclic} if there is a series $1=G_0\subseteq G_1\subseteq\ldots \subseteq G_n=G$ such that $G_i$ is normal in $G_{i+1}$ and the quotients $G_{i+1}/G_i$ are either finite or cyclic. Its \emph{Hirsch length} $h(G)$ is the number of infinite cyclic factors. Crowley and Sixt \cite[Thm\,7.3]{CrowleySixt} showed $\usr(\bfZ[G])\le h(G)+3$ for virtually polycyclic groups $G$. In particular, we have $\usr(\bfZ[G])\le 3$ for finite groups and $\usr(\bfZ[G])\le \operatorname{rank}(G)+3$ for finitely generated abelian groups.
\end{nrem}

In \cref{rem:resolutionbyembeddingssurfaces}, we briefly elaborate on how to include the case of orientable surfaces in this picture by utilising high-connectivity of the \emph{complex of tethered chains}---a result of Hatcher and Vogtmann \cite{HatcherVogtmann}. For constant coefficients, this implies Harer's classical stability theorem \cite{Harer} with a better, but not optimal range (see \cite{Boldsen,RWresolutions}). For twisted coefficients, it extends a result by Ivanov \cite{Ivanov} to more general coefficient systems. However, in the case of surfaces, stability with respect to most of these more general coefficient systems was already known by \cite{RWW}.

In \cref{section:coefficientsystemsmanifolds}, we show that coefficient systems for $\cM$ are equivalent to certain families of modules over the mapping class groups $\pi_0(\Diff_\partial(W))\cong\pi_1(\tB\Diff_\partial(W))$ and explain how the action of the mapping class groups on the homology of the manifolds gives rise to a coefficient system of degree $1$ for $\cM$. This yields the following corollary.

\begin{bigcor}\label{corollary:homologyrepresentation}Let $W$ be a compact $(p+q)$-manifold with nonempty boundary and $k\ge0$. The stabilisation \[\maptwo{\oH_i\left(\tB\Diff_\partial(W);\oH_k(W)\right)}{\oH_i(\tB\Diff_\partial(W\sharp(S^p\times S^q));\oH_k(W\sharp(S^p\times S^q)))}\] 
is an epi- and isomorphism for the same $W$ as in \cref{theorem:manifolds} and with the same ranges, after replacing $r$ by $1$.
\end{bigcor}

Furthermore, in \cref{section:extensionsmanifolds}, we provide a short discussion of how our methods can be applied to the case of certain stably parallelisable $(2n-1)$-connected $(4n+1)$-manifolds $X$ and $2$-connected $W$, extending stability results by Perlmutter \cite{PerlmutterLinkingforms}. Similarly, we also briefly explain how to enhance work of Kupers \cite{Kupers} on homeomorphisms of topological manifolds and automorphisms of piecewise linear manifolds.

\subsection*{Modules over braided monoidal categories}
We close in \cref{section:modulesovercategories} by explaining applicability of our results to discrete situations, such as groups or monoids, and by drawing a comparison to \cite{RWW}.

The classifying space $\tB \cM$ of a graded module $\cM$ over a braided monoidal category is a graded $\tE_1$-module over an $\tE_2$-algebra (see \cref{lemma:classifyingspacesofmodules}), so forms a suitable input for Theorems~\ref{theorem:constant} and~\ref{theorem:twisted}. In \cref{lemma:equivalenttosimplicialsets}, we identify the space of destabilisations $W_\bullet(A)$ of $A\in\cM$ with a semi-simplicial set $W^{\RW}_\bullet(A)$ in the case of $\cM$ being a groupoid satisfying an injectivity condition. This identification gives rise to a framework for homological stability for modules over braided monoidal categories, phrased entirely in terms of $\cM$ and semi-simplicial sets instead of semi-simplicial spaces up to higher coherent homotopy (see \cref{remark:discreteframework}).

Using this, it can, for instance, be concluded that work of Hepworth on homological stability for Coxeter groups \cite{Hepworth} with constant coefficients implies their stability with respect to a large class of nontrivial coefficient systems without further effort, as well as stability of their commutator subgroups.

In the case of a braided monoidal groupoid acting on itself, the semi-simplicial sets $W^{\RW}_\bullet(A)$ were introduced by Randal-Williams--Wahl in \cite{RWW} as part of their stability results for the automorphisms of a braided monoidal groupoid, which this work enhances in various ways. We generalise from braided monoidal groupoids to modules over such, remove all hypotheses on the categories they impose, improve the stability ranges in certain cases (see \cref{remark:improverangesRWW}), and enlarge the class of coefficient systems (see \cref{remark:improvecoefficientsystemsRWW}). We refer to \cref{section:comparisonRWW} for a more detailed comparison of our results in the discrete setting to \cite{RWW} and also for an analysis of their assumptions on the braided monoidal groupoid.

\subsection*{Acknowledgements}
I wish to express my gratitude to Nathalie Wahl for her overall support, to Oscar Randal-Williams for enlightening discussions and his hospitality at the University of Cambridge, and to Søren Galatius as well as to Alexander Kupers for valuable comments. 

I was supported by the Danish National Research Foundation through the Centre for Symmetry and Deformation (DNRF92) and by the European Research Council (ERC) under the European Union’s Horizon 2020 research and innovation programme (grant agreement No 682922).

\addtocontents{toc}{\protect\setcounter{tocdepth}{3}}

\tableofcontents

\section{Preliminaries}\label{section:preliminaries}
This section is devoted to fix conventions and collect general techniques. We work in the category of compactly generated spaces, use Moore paths throughout, and denote the endpoint of a path $\mu$ by $\omega(\mu)$.

\subsection{Graded spaces and categories} \label{section:gradings}
We denote by $(\bfNext,+)$ the discrete abelian monoid obtained by extending the non-negative integers $(\bfN_0,+)$ by an element $\infty$ satisfying $k+\infty=\infty$ for all $k\ge0$. 

A \emph{graded space} is a space $X$ together with a continuous map $\mapwoshort{g_X}{X}{\bfNext}$. A \emph{map of graded spaces} is a map that preserves the grading and a \emph{map of degree $k$ between graded spaces} for a number $k\ge0$ is a map that increases the degree by $k$. The category of graded spaces is symmetric monoidal with the monoidal product of two graded spaces $(X,g_X)$ and $(Y,g_Y)$ given by $(X\times Y,g_X+g_Y)$. The subspace of elements of degree $n\in\bfNext$ is denoted by $X_n=g_X^{-1}(\{n\})\subseteq X$. By restricting the grading, subspaces of graded spaces are implicitly considered as being graded. A graded space $(X,g_X)$, or a map $\maptwoshort{(Y,g_Y)}{(X,g_X)}$ of graded spaces, is \emph{$\varphi(g_X)$-connected in degrees $\ge m$} for a function $\mapwoshort{\varphi}{\bfNext}{\bfQ\cup\{\infty\}}$ satisfying $\phi(\infty)=\infty$ if $X_n$ or $\maptwoshort{Y_n}{Z_n}$, respectively, is $\lfloor \varphi(n)\rfloor $-connected for all $ m\le n<\infty$ in the usual sense. Note that we do not require anything on $X_\infty$ or $\maptwoshort{Y_\infty}{X_\infty}$.

A \emph{graded set} $X$ is a graded space that is discrete. A \emph{graded category} $\cC$ is a category internal to graded sets, i.e.~a category $\cC$ with a function $\mapwoshort{g_\cC}{\ob\cC}{\bfNext}$ whose value on objects that are connected by morphisms is constant. This is equivalent to a grading on the classifying space $\tB\cC$. A \emph{graded monoidal category} is a monoid internal to graded categories with the monoidal product $(\cC,g_\cC)\times(\cD,g_\cD)=(\cC\times\cD,g_\cC+g_\cD)$, i.e.~a monoidal category $(\cA,\oplus,0)$ together with a grading $g_\cA$ on $\cA$ that satisfies $g_\cA(0)=0$ and $g_\cA(X\oplus Y)=g_\cA(X)+g_\cA(Y)$. A \emph{graded right-module} $(\cM,\oplus)$ over a graded monoidal category $(\cA,\oplus,0)$ is a graded category $\cM$ together with a right-action of $(\cA,\oplus,0)$ on $\cD$ internal to graded categories, i.e.~a functor $\mapwoshort{\oplus}{\cM\times\cA}{\cM}$ which is unital and associative up to coherent isomorphisms, and satisfies $g_\cM(A\oplus X)=g_\cM(A)+g_\cM(X)$.

\subsection{Homology with local coefficients} \label{section:localcoefficients}We adopt the convention of \cite[Ch.\,VI]{Whitehead}: for points $x$ and $y$ in a space $X$, a morphism in the fundamental groupoid $\Pi(X)$  from $x$ to $y$ is a homotopy class of paths from $y$ to $x$, resulting in the fundamental group $\pi_1(X,x)$ being a subgroupoid of $\Pi(X)$. A \emph{local system} on a pair of spaces $(X,A)$ with $A\subseteq X$ is a functor $F$ from the fundamental groupoid $\Pi(X)$ of $X$ to the category of abelian groups. It is \emph{constant} if it is constant as a functor. For a path-connected space $X$, local systems can equivalently be described as modules over $\pi_1(X,x)$, since the fundamental groupoid $\Pi(X)$ is equivalent to the one-object groupoid $\pi_1(X,x)$. Subspaces of spaces with local systems are implicitly equipped with the local system obtained by restriction along the inclusion. When we write $(X,A)$ for a map $\maptwoshort{A}{X}$ that is not necessarily an inclusion, we implicitly replace $X$ by the mapping cylinder of $\maptwoshort{A}{X}$. A \emph{morphism $(f;\eta)$ between pairs with local systems} $(X,A;F)$ and $(Y,B;G)$ is a map of pairs $\mapwoshort{f}{(X,A)}{(Y,B)}$ with a natural transformation $\mapwoshort{\eta}{F}{f^*G}$ of functors on $\Pi(X)$. A \emph{homotopy} between $(f_0;\eta_0)$ and $(f_1;\eta_1)$ from $(X,A;F)$ to $(Y,B;G)$ consists of a homotopy of pairs $\mapwoshort{H_t}{(X,A)}{(Y,B)}$ from $f_0$ and $f_1$ such that
\begin{equation}\label{equation:localcoefficientsdiagram}
\begin{tikzcd}[row sep=0.1cm,column sep=2cm]
&G(f_1(-))\arrow[dd,"G(H_t(-))"]\\
F(-)\arrow[ru,"\eta_1"]\arrow[rd,"\eta_0",swap]&\\
&G(f_0(-))
\end{tikzcd}
\end{equation} commutes. Taking singular chains with coefficients in a local system provides a homotopy invariant functor $\oC_*(-)$ from pairs with local systems to chain complexes. The homology $\oH_*(X,A;F)$ of $\oC_*(X,A;F)$ is the \emph{homology of the pair $(X,A)$ with coefficients in the local system $F$}. A grading on $X$ results in an additional grading $\bigoplus_{n\in\bfNext}\oH_*(X_n,A_n;F)$ on $\oH_*(X,A;F)$. For a morphism $\maptwoshort{(X,A;F)}{(Y,B;G)}$, the homology of the mapping cone of $\maptwoshort{\oC_*(X,A;F)}{\oC_*(Y,B;G)}$ is denoted by $\oH_*((Y,B;G),(X,A;F))$. If $X$ and $Y$ are graded and the underlying map $\maptwoshort{X}{Y}$ is of degree $k$, then $\oH_*((Y,B;G),(X,A;F))$ inherits an extra grading
\[\textstyle{\oH_*\big((Y,B;G),(X,A;F)\big)=\bigoplus_{n\in\bfNext}\oH_*\big((Y_{n+k},B_{n+k};G),(X_n,A_n;F)\big)}.\] We refer to \cite[Ch.\,VI]{Whitehead} for more details on homology with local coefficients. 

\subsection{Augmented semi-simplicial spaces}\label{section:semisimplicial}
Denoting by $[p]$ the ordered set $ \{0,1,\ldots,p\}$, the \emph{semi-simplicial category} is the category $\semisimp$ with objects $[0],[1],\ldots$ and order-preserving injections between them. A \emph{semi-simplicial space} $X_\bullet$ is a space-valued functor on $\semisimpop$, or equivalently, a collection of spaces $X_p$ for $p\ge0$, together with \emph{face maps} $\mapwoshort{d_i}{X_p}{X_{p-1}}$ for $0\le i\le p$ that satisfy the \emph{face relations} $d_id_j=d_{j-1}d_i$ for $i<j$. An \emph{augmented semi-simplicial space} $\maptwoshort{X_\bullet}{X_{-1}}$ is a semi-simplicial space $X_\bullet$ with maps $\maptwoshort{X_p}{X_{-1}}$ for $p\ge0$ that commute with the face maps. As for simplicial spaces, augmented semi-simplicial spaces $\maptwoshort{X_\bullet}{X_{-1}}$ have a \emph{geometric realisation}---a space over $X_{-1}$, denoted by $\maptwoshort{|X_\bullet|}{X_{-1}}$ (see \cite[Sect.\,1.2]{semisimplicial}).

Given an augmented semi-simplicial space $\maptwoshort{X_\bullet}{X_{-1}}$ and a local system $F$ on $X_{-1}$, we obtain local systems on the spaces of $p$-simplices $X_p$ and on the realisation $|X_\bullet|$ by pulling back $F$ along the augmentation. Filtering $|X_\bullet|$ by skeleta induces a strongly convergent homologically graded spectral sequence 
\begin{equation}\label{equation:sssemisimplicial}E_{p,q}^1\cong \oH_q(X_p;F)\implies \oH_{p+q+1}(X_{-1},|X_\bullet|;F),\end{equation}
 defined for $q\ge0$ and $p\ge -1$ (see \cites[Sect.\,1.4]{semisimplicial}[Lem.\,2.7]{MP}). The differential $\mapwoshort{d^1}{\oH_q(X_p;F)}{\oH_q(X_{p-1};F)}$ is the alternating sum $\sum_{i=0}^{p}(-1)^{i}(d_i;\id)_*$ of the morphisms induced by the face maps for $p>0$, and induced by the augmentation for $p=0$. Given a morphism of augmented semi-simplicial spaces
$\mapwoshort{(f_\bullet,f_{-1})}{(X_\bullet\rightarrow X_{-1})}{(Y_\bullet\rightarrow Y_{-1})},$
local systems $F$ on $X_{-1}$ and $G$ on $Y_{-1}$, and a morphism of local systems $\maptwoshort{F}{{f_{-1}}^*G}$, we obtain a morphism of augmented semi-simplicial objects in spaces with local systems, resulting in a relative version of the spectral sequence \eqref{equation:sssemisimplicial}, 
\begin{equation}\label{equation:sssemisimplicialrelative}E_{p,q}^1\cong \oH_q\big((Y_p;G),(X_p;F)\big)\implies \oH_{p+q+1}\big((Y_{-1},|Y_\bullet|;G),(X_{-1},|X_\bullet|;F)\big).\end{equation}
If $X_{-1}$ is graded, all spaces $X_p$ and $|X_\bullet|$ inherit a grading by pulling back $g_{X_{-1}}$ along the augmentation. This results in a third grading of the spectral sequence \eqref{equation:sssemisimplicial}, but since the differentials preserve the additional grading, it is just a sum of spectral sequences, one for each $n\in\bfNext$. Analogously, if the map $f_{-1}$ of $\mapwoshort{(f_\bullet,f_{-1})}{(X_\bullet\rightarrow X_{-1})}{(Y_\bullet\rightarrow Y_{-1})}$ is a map of degree $k$ for gradings on $X_{-1}$ and $Y_{-1}$, the spectral sequence \eqref{equation:sssemisimplicialrelative} splits as a sum, with the $n$th summand of the $E_1$-page being $E_{p,q,n}^1\cong \oH_q((Y_{p,n+k};G),(X_{p,n};F))$.

\subsection{$\cC$-spaces and their rectification}\label{section:homotopycolimits}
We set up an ad-hoc theory of spaces parametrised by a topologically enriched category, serving us as a convenient language in the body of this work.

We call an enriched space-valued functor $X_\bullet$ on a topologically enriched category $\cC$ a \emph{$\cC$-space}, and write $X_C$ for its value at an object $C$. An \emph{augmentation} $\mapwoshort{f_\bullet}{X_\bullet}{X_{-1}}$ of a $\cC$-space $X_\bullet$ over a space $X_{-1}$ is a lift of $X_\bullet$ to a functor with values in the overcategory $\TT/X_{-1}$, and an \emph{augmented $\cC$-space} is a $\cC$-space together with an augmentation. We denote the value of an augmented $\cC$-space $\mapwoshort{f_\bullet}{X_\bullet}{X_{-1}}$ at an object $C$ by $\mapwoshort{f_C}{X_C}{X_{-1}}$. A \emph{morphism of augmented $\cC$-spaces} is a natural transformation of functors $\maptwoshort{\cC}{\TT/X_{-1}}$, and it is called a \emph{weak equivalence} if it is a weak equivalence objectwise. A morphism between a $\cC$-space $X_\bullet$ augmented over $X_{-1}$ and a $\cC$-space $Y_{\bullet}$ over $Y_{-1}$ consists of a map $\mapwoshort{h}{X_{-1}}{Y_{-1}}$ and a morphism $\mapwoshort{h_\bullet}{h_*(X_\bullet)}{Y_\bullet}$ of $\cC$-spaces augmented over $Y_{-1}$, where $h_*(X_\bullet)$ denotes $X_\bullet$ considered augmented over $Y_{-1}$ via $h$. Such a morphism is a \emph{weak equivalence} if $h$ is a weak equivalence of spaces and $h_\bullet$ is one of $\cC$-spaces over $Y_{-1}$.  An augmented $\cC$-space $f_\bullet$ is \emph{fibrant} if all maps $f_C$ are Serre fibrations. 

\begin{ex}For $\cC$ being the opposite of the semi-simplicial category, the notion of a $\cC$-space agrees with the one of a semi-simplicial spaces (see \cref{section:semisimplicial}). This example motivated our choice of notation.
\end{ex}

\begin{dfn}\label{definition:fibrantreplacement} The \emph{fibrant replacement} of an augmented $\cC$-space $\maptwoshort{X_\bullet}{X_{-1}}$ is the augmented $\cC$-space $\maptwoshort{X_\bullet^{\fib}}{X_{-1}}$ obtained by applying the path-space construction objectwise, 
\[X^{\fib}_C=\{(x,\mu)\in X_C\times \Path X_{-1}\mid\omega(\mu)=f_C(x)\},\] considered as a space over $X_{-1}$ by evaluating paths at zero. It is fibrant and admits a canonical weak equivalence $\maptwoshort{X_\bullet}{X_\bullet^{\fib}}$ of augmented $\cC$-spaces, given by mapping $x\in X_C$ to $(x,\const_{f_C(x)})\in X_C^{\fib}$.
\end{dfn}

The \emph{fibre} $X_{x,\bullet}$ of an augmented $\cC$-space $\mapwoshort{f_\bullet}{X_\bullet}{X_{-1}}$ at $x\in X_{-1}$ is the $\cC$-space that assigns to an object $C$ the fibre $X_{x,C}=f_C^{-1}(x)$. Its \emph{homotopy fibre} $\hofib_x(X_\bullet)$ at $x$ is the fibre of $\maptwoshort{X_\bullet^{\fib}}{X_{-1}}$ at $x$. If $\maptwoshort{X_\bullet}{X_{-1}}$ is fibrant, then the weak equivalence $\maptwoshort{X_\bullet}{X_\bullet^{\fib}}$ induces a weak equivalence $\maptwoshort{X_{x,\bullet}}{\hofib_x(X_\bullet)}$.

\begin{dfn}\label{definition:bar-construction}
Let $\cC$ be a small topologically enriched category.
\begin{enumerate}
\item The \emph{bar construction} $\tB(Y_\bullet,\cC,X_\bullet)$ of a pair of $\cC$-spaces $(X_\bullet,Y_\bullet)$, where $X_\bullet$ is co- and $Y_\bullet$ is contravariant, is the realisation of the semi-simplicial space $\tB_{\smallsquare}(Y_\bullet,\cC,X_\bullet)$ with $p$-simplices \[\textstyle{\coprod_{C_0,\ldots,C_p\in\ob\cC}X_{C_0}\times\cC(C_0,C_1)\times\ldots\times\cC(C_{p-1},C_p)\times Y_{C_p}}.\] The $i$th face map is induced by composing morphisms in $\cC(C_{i-1},C_i)$ and $\cC(C_{i},C_{i+1})$ for $1\le i\le p-1$, and by the evaluations $\maptwoshort{X_{C_0}\times\cC(C_0,C_1)}{X_{C_1}}$ and  $\maptwoshort{\cC(C_{p-1},C_p)\times Y_{C_p}}{X_{C_{p-1}}}$ for $i=p-1$ and $i=p$, respectively. An augmentation $\maptwoshort{X_\bullet}{X_{-1}}$ naturally induces a map $\maptwoshort{\tB(Y_\bullet,\cC,X_\bullet)}{X_{-1}}$.
\item The \emph{homotopy colimit} \[\maptwo{\hocolim_{\cC}X_\bullet}{X_{-1}}\] of an augmented $\cC$-space $\maptwoshort{X_\bullet}{X_{-1}}$ is the bar construction $\maptwoshort{\tB(\ast,\cC,X_\bullet)}{X_{-1}}$.
\end{enumerate}
\end{dfn}

A $\cC$-space is \emph{$k$-connected} for a number $k\ge0$ if its homotopy colimit is so. If the base $X_{-1}$ of an augmented $\cC$-space $\maptwoshort{X_{\bullet}}{X_{-1}}$ is graded, then its values $X_C$ and its homotopy colimit inherit gradings by pulling back $g_{X_{-1}}$ from $X_{-1}$. It is \emph{graded $\varphi(g_{X_{-1}})$-connected in degrees $\ge m$} for $\mapwoshort{\varphi}{\bfNext}{\bfQ\cup\{\infty\}}$ if $\maptwoshort{\hocolim_\cC X_{\bullet}}{X_{-1}}$ is. A functor between topologically enriched categories is a \emph{weak equivalence} if it induces weak equivalences on morphism spaces and a bijection on the set of objects. Note that this notion of weak equivalence is slightly stronger than the usual one. With this choice, it is immediate to see that the map on bar constructions induced by a weak equivalence $\maptwoshort{(X_\bullet,\cC,Y_\bullet)}{(X_\bullet',\cC',Y_\bullet')}$ of triples, defined in the appropriate sense, is a weak equivalence, since levelwise weak equivalences of semi-simplicial spaces realise to weak equivalences (see e.g.~\cite[Thm 2.2]{semisimplicial}). In particular, taking homotopy colimits turns weak equivalences of $\cC$-spaces augmented over $X_{-1}$ into weak equivalences of spaces over $X_{-1}$.

\begin{lem}\label{section:homotopyfibrerealisation}Let $\maptwoshort{X_\bullet}{X_{-1}}$ be an augmented $\cC$-space and $x\in X_{-1}$. The canonical map
\[\maptwoshort{\hocolim_\cC(\hofib_x(\maptwoshort{X_\bullet}{X_{-1}}))}{\hofib_x(\maptwoshort{\hocolim_\cC X_\bullet}{X_{-1}})}\] is a weak equivalence.
\end{lem}
\begin{proof}
We show that the map in consideration is even a homeomorphism, provided $X_{-1}$ is a weak Hausdorff space. This implies the claim, since the two functors in comparison both preserve weak equivalences of augmented $\cC$-spaces and every augmented $\cC$-space $\maptwoshort{X_\bullet}{X_{-1}}$ can be replaced, up to weak equivalence, by one over a weak Hausdorff space, for instance by pulling back the fibrant replacement of $\maptwoshort{X_\bullet}{X_{-1}}$ along a CW-approximation of $X_{-1}$. We have $\hocolim_\cC(\hofib_x(\maptwoshort{X_\bullet}{X_{-1}}))=|\tB_{\smallsquare}(*,\cC,\hofib_x(\maptwoshort{X_\bullet}{X_{-1}}))|$ and $\hofib_x(\maptwoshort{\hocolim_\cC X_\bullet}{X_{-1}})=\hofib_x(|\tB_{\smallsquare}(*,\cC,\maptwoshort{X_\bullet}{X_{-1}})|)$, so the statement follows from proving that both the bar construction $\tB_{\smallsquare}(*,\cC,-)$ as well as the geometric realisation $|-|$ commute with taking homotopy fibres $\hofib_x(-)$. Unwrapping the definitions of $\tB_{\smallsquare}(*,\cC,-)$ and $|-|$, these two claims are implied by the fact that the functor $\mapwoshort{\hofib_x(-)}{\TT/X_{-1}}{\TT}$ commutes with colimits and also with taking products $-\times Z$ with a fixed space $Z$. The latter is clear, and the former follows from the fact that the functor $\mapwoshort{\omega^{*}}{\TT/X_{-1}}{\TT/(\Path_x X_{-1}})$ given by pulling back the path fibration $\mapwoshort{\omega}{\Path_x X_{-1}}{X_{-1}}$ is a left adjoint \cite[Prop.\,2.1.3]{MaySigurdsson}, so preserves colimits, together with the observation that the forgetful functor from $\TT/X_{-1}$ to $\TT$ is colimit-preserving as well.
\end{proof}

For an augmented $\cC$-space $\maptwoshort{X_\bullet}{X_{-1}}$, the composition in $\cC$ and the evaluation maps $\maptwoshort{X_C'\times \cC(C',C)}{X_{C}}$ combine to augmentations $\maptwoshort{B_\bullet(\cC({\smallsquare},C),\cC,X_{\smallsquare})}{X_C}$ for each $C$ in $\cC$, which realise to weak equivalences as they admit extra degeneracies by inserting the identity (see e.g.~\cite[Thm 2.2]{semisimplicial}). These equivalences are natural in $C$ and compatible with the augmentation to $X_{-1}$, so assemble to a weak equivalence \[\maptwoshort{B\big(\cC({\smallsquare},\bullet),\cC,X_{\smallsquare}\big)}{X_\bullet}\] of augmented $\cC$-spaces---the \emph{bar resolution} of $X_\bullet\rightarrow X_{-1}$. 

\begin{lem}\label{lemma:rectification}Let $\mapwoshort{p}{\cC}{\cD}$ be a weak equivalence of topologically enriched categories. There is a functor \[\mapwo{p_*}{(\TT/X_{-1})^\cC}{(\TT/X_{-1})^\cD}\] that fits into a zig-zag of natural transformations between endofunctors on $(\TT/X_{-1})^\cC$, \[p^*p_*\longleftarrow\cdot\longrightarrow \id_{(\TT/X_{-1})^\cC},\] where $\mapwoshort{p^*}{(\TT/X_{-1})^\cD}{(\TT/X_{-1})^\cC}$ is given by precomposition with $p$. When evaluated at an augmented $\cC$-space, the zig-zag consists of weak equivalences of augmented $\cC$-spaces.\end{lem}
\begin{proof}
The value $p_*X_\bullet$ for $X_\bullet\in(\TT/X_{-1})^\cC$ is the homotopy left Kan-extension of $X_\bullet$ along $p$, mapping an object $D$ in $\cD$ to $\tB(\cD(p({\smallsquare}),D),\cC,X_{\smallsquare})$. Its pullback $p^*p_*X_\bullet$ fits into a zig-zag of augmented $\cC$-spaces\[p^*p_*X_\bullet=\tB\big(\cD(p({\smallsquare}),p(\bullet)),\cC,X_{\smallsquare}\big)\longleftarrow \tB\big(\cC({\smallsquare},\bullet),\cC,X_{\smallsquare}\big)\longrightarrow X_\bullet,\] in which the left arrow is induced by $p$ and the right one is the bar resolution of $X_\bullet$, so both are weak equivalences and compatible with the augmentation. As the zig-zag is natural in $X_\bullet$, the claim follows.
\end{proof}

\begin{lem}\label{lemma:hocolimrealisation}The homotopy colimit of an augmented semi-simplicial space $\maptwoshort{X_\bullet}{X_{-1}}$ and its geometric realisation are weakly equivalent as spaces over $X_{-1}$.
\end{lem}
\begin{proof}The classifying space of the overcategory $\semisimp/[p]$ is isomorphic to the $p$th topological standard simplex $\Delta^p$, since the nerve of $\semisimp/[p]$ is the barycentric subdivision of the $p$th simplicial standard simplex. This extends to an isomorphism $\Delta^\bullet\cong B(\semisimp/\bullet)$ of co-semi-simplicial spaces from which \cite[Thm\,6.6.1]{Riehl} implies that, given an augmented semi-simplicial space $\maptwoshort{X_\bullet}{X_{-1}}$, the thin realisation (see \cite[Sect.\,1.2]{semisimplicial}) of $\tB_\bullet(\ast,\semisimpop,X_\bullet)$, considered as a simplicial space, is homeomorphic over $X_{-1}$ to the realisation of $X_\bullet$. But for augmented $\cC$-spaces $\maptwoshort{X_\bullet}{X_{-1}}$ on a discrete category $\cC$, the fat and the thin geometric realisation of $\tB_\bullet(\ast,\cC,X)$ are weakly equivalent over $X_{-1}$, because $\tB_\bullet(\ast,\cC,X)$ is \emph{good} in the sense of \cite[Prop.\,A.1]{Segal}.
\end{proof}

\subsection{Semi-simplicial spaces up to higher coherent homotopy}
\label{section:semisimplicialthick}
In the course of this work, a number of constructions that are key to the theory require choices of contractible ambiguity. To deal with such, we are led to consider objects that are as good as semi-simplicial spaces, but only in a homotopical sense. To model those, let us define an \emph{(augmented) semi-simplicial space up to higher coherent homotopy} as an (augmented) $\semisimpthick$-space $X$, defined on any topologically enriched category $\semisimpthick$ that comes with a weak equivalence $\maptwoshort{\semisimpthick}{\semisimp}$. Roughly speaking, these are categories with the same objects as $\semisimp$ and a (weakly) contractible space of choices for all morphisms in $\semisimp$. In particular, a $\semisimpthick$-space $X_\bullet$ includes spaces $X_p$ for $p\ge0$, together with face maps $\mapwoshort{\tilde{d}_i}{X_p}{X_{p-1}}$, unique up to homotopy. By precomposing with $\maptwoshort{\semisimpthick}{\semisimp}$, every semi-simplicial space is a $\semisimpthick$-space and in the light of \cref{lemma:rectification}, every $\semisimpthick$-space is equivalent to one arising in this way. By virtue of this rectification result and \cref{lemma:hocolimrealisation}, all homotopy invariant constructions for semi-simplicial spaces carry over to $\semisimpthick$-spaces, so in particular, we have analogues of the spectral sequences \eqref{equation:sssemisimplicial} and \eqref{equation:sssemisimplicialrelative}, the differentials being the alternating sum $\sum_{i=0}^{p}(-1)^{i}(\tilde{d}_i)_*$ of morphisms induced by (weakly) contractible choices $\tilde{d}_i$ of face maps. A $\semisimpthick$-space $X_\bullet$ induces a simplicial set $\pi_0(X_\bullet)$ by taking path components, together with a morphism $\maptwoshort{X_\bullet}{\pi_0(X_\bullet)}$ of $\semisimpthick$-spaces, which is a weak equivalence if and only if $X_\bullet$ is homotopy discrete, i.e.~takes values in homotopy discrete spaces. To emphasise similarities and by abuse of notation justified by \cref{lemma:hocolimrealisation}, we call the homotopy colimit of an augmented $\semisimpthick$-space $\maptwoshort{X_{\bullet}}{X_{-1}}$ its \emph{realisation}, and denote it by $\maptwoshort{|X_\bullet|}{X_{-1}}$.

\section{The canonical resolution of an $\tE_1$-module over an $\tE_2$-algebra}
\subsection{$\tE_1$-modules over $\tE_n$-algebras and their fundamental groupoids}\label{section:modules}
We recall the notion of an $\tE_1$-module over an $\tE_n$-algebra and explain its relation to modules over monoidal categories.

By an \emph{operad}, we mean a symmetric coloured operad in spaces (see e.g.~\cite[Sect.\,1.1]{BergerMoerdijkColoredOperads}), and an \emph{algebra} over such is understood in the usual sense (see e.g.~\cite[Sect.\,1.1]{BergerMoerdijkColoredOperads}). For a subspace $X\subseteq\bfR^n$, we let $\cD^k(X)$ be the space of tuples of $k$ embeddings of the closed disc $D^n$ into $X$ that have disjoint interiors and are compositions of scalings and translations. Recall the one-coloured operad $\cD^\bullet(D^n)$ of little $n$-discs \cite{BoardmanVogt,MayGeometry} with $k$-operations $\cD^k(D^n)$ and operadic composition induced by composing embeddings.

\begin{dfn}\label{definition:SC}Let $\SC_n$ be the coloured operad with colours $\ocolour{m}$ and $\ocolour{a}$ whose space of operations $\SC_n(\ocolour{m}^k, \ocolour{a}^l; \ocolour{m})$ is empty for $k\neq1$ and for $k=1$ the space of pairs $(s,\phi)\in[0,\infty)\times \cD^l(\bfR^n)$ such that $\phi\in\cD^l((0,s)\times(-1,1)^{n-1})$, allowing $(0,\emptyset)\in [0,\infty)\times \cD^0(\bfR^n)$ as a valid element of $\SC_n(\ocolour{m}^k, \ocolour{a}^0; \ocolour{m})$. The space $\SC_n(\ocolour{m}^k,\ocolour{a}^l,\ocolour{a})$ is empty for $k\neq0$ and equals $\cD^l(D^n)$ otherwise. The composition restricted to the $\ocolour{a}$-colour is given by the composition in $\cD^\bullet(D^n)$ and the composition \[\mapwo{\gamma}{\SC_n(\ocolour{m},\ocolour{a}^l;\ocolour{m})\times \left(\SC_n(\ocolour{m},\ocolour{a}^k;\ocolour{m})\times \SC_n(\ocolour{a}^{i_1};\ocolour{a})\times\ldots \times\SC_n(\ocolour{a}^{i_l};\ocolour{a})\right)}{\SC_n(\ocolour{m},\ocolour{a}^{k+i};\ocolour{m})}\] for $i=\sum_j i_j$ by mapping an element $((s,\phi),((s',\psi),(\varphi^1,\ldots,\varphi^l)))$ in the codomain to $(s'+s,(\psi,(\phi_1\circ\varphi^1)+s',\ldots,(\phi_l\circ\varphi^l)+s'))\in \SC_n(\ocolour{m},\ocolour{a}^{k+i};\ocolour{m})$, where $(-+s')$ denotes the translation by $s'$ in the first coordinate. In words, it is defined by adding the parameters, putting the discs of $\SC_n(m,a^k;m)$ to the left of the ones of $\SC_n(\ocolour{m},\ocolour{a}^l;\ocolour{m})$, and composing the embeddings of discs of the $\SC_n(\ocolour{a}^{i_j};\ocolour{a})$-factors with the ones of $\SC_n(\ocolour{m},\ocolour{a}^l;\ocolour{m})$ as in the operad of little $n$-discs. See Figure~\ref{figure:operad} for an illustration.

\end{dfn}
\begin{figure}[h]
\begin{subfigure}[b]{0.2\linewidth}
\centering
\begin{tikzpicture}[scale=0.4,node1/.style={inner sep=0pt,minimum size=1.5cm]}]
\draw (4,-2) circle (1.5cm);
\node[node1] (2) at (4,-2){$1$};
\draw [rounded corners] (0,0)--(6,0)--(6,-4)--node[below] {$(0,s)$}(0,-4)-- cycle;
\end{tikzpicture}
\caption*{$e\in \SC_2(\ocolour{m},\ocolour{a};\ocolour{m})$}
\end{subfigure}
\begin{subfigure}[b]{0.2 \linewidth}
\centering
\begin{tikzpicture}[scale=0.4,node1/.style={inner sep=0pt,minimum size=1.5cm]}]
\draw (1.2,-3) circle (0.75cm);
\node[node1] (1) at (1.2,-3){$1$};
\draw [rounded corners] (0,0)--(2.4,0)--(2.4,-4)--node[below] {$(0,s')$}(0,-4)-- cycle;
\end{tikzpicture}
\caption*{$d\in \SC_2(\ocolour{m},\ocolour{a};\ocolour{m})$}
\end{subfigure}
\begin{subfigure}[b]{0.2 \linewidth}
\centering
\begin{tikzpicture}[scale=0.4,node2/.style={inner sep=0pt,minimum size=0.75cm]}]
\draw (0,0) circle (1.5);
\draw (-0.7,0) circle (0.5);
\draw (0.7,0) circle (0.5);
\node[node2] (1) at (-0.7,0){$1$};
\node[node2] (1) at (0.7,0){$2$};
\node[node2] (1) at (0,-2){$\ $};
\end{tikzpicture}
\caption*{$f\in \SC_2(\ocolour{a}^2;\ocolour{a})$}
\end{subfigure}
\begin{subfigure}[b]{0.3\linewidth}
\centering
  \begin{tikzpicture}[scale=0.4,node1/.style={inner sep=0pt,minimum size=1.5cm]},]
  \draw (1.2,-3) circle (0.75);
  \draw[dashed] (6.4,-2) circle (1.5);
  \draw (5.75,-2) circle (0.5);
  \draw (7.05,-2) circle (0.5);
  \node[node1] (1) at (1.2,-3){$1$};
\draw [rounded corners] (2.4,-4)--node[below] {$(0,s')$}(0,-4)--(0,0)--(2.4,0);
\draw[dashed] (2.4,0)--(2.4,-4);
\node[node1] (2) at (5.75,-2){$2$};
\node[node1] (2) at (7.05,-2){$3$};
\draw [rounded corners] (2.4,0)--(8.4,0)--(8.4,-4)--node[below] {$(s',s'+s)$}(2.4,-4);
\end{tikzpicture}\hfill
\caption*{$\gamma(e;d,f)\in\SC_2(\ocolour{m},\ocolour{a}^3;\ocolour{m})$}
\end{subfigure}
\caption{The operadic composition of $\SC_n$}
\label{figure:operad}
\end{figure}
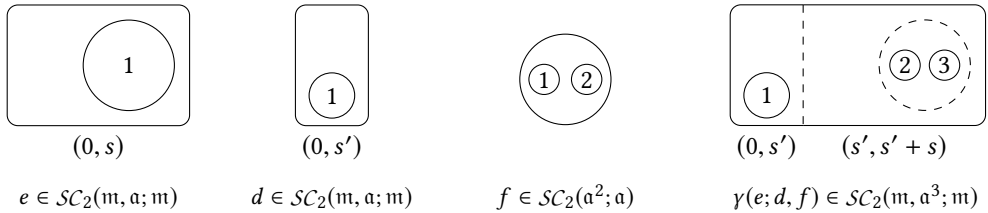
The canonical embedding $\maptwoshort{\cD^\bullet(D^n)}{\cD^\bullet(D^{n+1})}$ of little discs operads (see e.g.~\cite[Sect.\,4.1.5]{Fresse}), extends to an embedding of two-coloured operads $\SC_n\rightarrow \SC_{n+1}$ by taking products with $(-1,1)$ from the right. Consequently, any algebra over $\SC_{n+1}$ is also one over $\SC_{n}$. 

We call two coloured operads \emph{weakly equivalent} if there is a zig-zag between them that consists of morphisms of operads that are weak homotopy equivalences on all spaces of operations.

\begin{rem}The operad $\SC_n $ is weakly equivalent to a suboperad of the $n$-dimensional version of the Swiss-Cheese operad of \cite{Voronov}, motivating the notation.
\end{rem}

\begin{dfn}\label{definition:E1module}An \emph{$\tE_{1,n}$-operad} is an operad $\cO$ that is weakly equivalent to $\SC_n$. A graded \emph{$\tE_1$-module $\cM$ over an $\tE_n$-algebra $\cA$} is an algebra $(\cM,\cA)$ over an $\tE_{1,n}$-operad $\cO$, considered as an operad in graded spaces, where $\cM$ corresponds to the $\ocolour{m}$- and $\cA$ to the $\ocolour{a}$-colour. That is, it consists of two graded spaces $(\cM,g_{\cM})$ and $(\cA,g_{\cA})$, together with multiplication maps for $l\ge0$ of the form
\[\mapwo{\theta}{\cO(\ocolour{m},\ocolour{a}^l;\ocolour{m})\times \cM\times \cA^{l}}{\cM}\quad\text{and}\quad\mapwo{\theta}{\cO(\ocolour{a}^l;\ocolour{a})\times \cA^{l}}{\cA},\] which are graded, where $\cO(\ocolour{m},\ocolour{a}^l;\ocolour{m})$ and $\cO(\ocolour{a}^l;\ocolour{a})$ are equipped with the grading that is constantly $0$, i.e.~the degree of a multiplication of points is the sum of their degrees. These structure maps are required to satisfy the usual associativity, unitality, and equivariance axioms for an algebra over an coloured operad.
\end{dfn}

The fundamental groupoid of an algebra over the little $2$-discs operad has a braided monoidal groupoid structure; the multiplication is induced by the choice of a $2$-operation \cite[Ch.\,5--6]{Fresse}. Similarly, for a graded algebra $(\cM,\cA)$ over an $\tE_{1,2}$-operad $\cO$ and operations $c\in\cO(\ocolour{m},\ocolour{a};\ocolour{m})$ and $d\in\cO(\ocolour{a}^2;\ocolour{a})$, the fundamental groupoid $\Pi(\cA)$ is a graded braided monoidal groupoid with multiplication induced by $d$, and $\Pi(\cM)$ becomes a graded right-module over $\Pi(\cA)$ with the action induced by $c$. In other words, the functor $\mapwoshort{\oplus}{\Pi(\cM)\times\Pi(\cA)}{\Pi(\cM)}$ induced by $\theta(c;-,-)$ is associative, unital up to coherent natural isomorphisms, and compatible with the grading on $\Pi(\cM)$ and $\Pi(\cA)$ induced by the grading on $\cM$ and $\cA$.

\begin{rem}\label{remark:gradingmoduleisgradingoncategory}Since the path components of a space coincide with the path components of its fundamental groupoid in the categorical sense, a grading on an $\tE_1$-module over an $\tE_n$-algebra is equivalent to a grading of the induced right-module $(\Pi(\cM),\oplus)$ over the braided monoidal groupoid $(\Pi(\cA),\oplus,\braiding,0)$.
\end{rem}

\subsection{The canonical resolution}\label{section:complexes}
Let $\cM$ be a graded $\tE_1$-module over an $\tE_2$-algebra $\cA$ with underlying $\tE_{1,2}$-operad $\cO$ and structure maps $\theta$. We call a point $X\in\cA$ of degree $1$ a \emph{stabilising object} for $\cM$, and define the \emph{stabilisation map} with respect to a stabilising point $X$,  \[\mapwo{s}{\cM}{\cM},\] as the multiplication $\theta(c;-,X)$ by $X$, using an operation $c\in\cO(\ocolour{m},\ocolour{a},\ocolour{m})$, which we fix once and for all. As $X$ has degree $1$, so does the stabilisation map, which hence restricts to maps $\mapwoshort{s}{\cM_n}{\cM_{n+1}}$ between the subspaces of consecutive degrees for all $n\ge 0$. It will be convenient to denote the stabilisation map also by $\mapwoshort{(-\oplus X)}{\cM}{\cM}$ and we use the two notations interchangeably.

\begin{rem}We chose to restrict to stabilising objects of degree $1$ to simplify the exposition. However, by keeping track of the gradings, the developed theory generalises to stabilising objects of arbitrary degree.
\end{rem}

In the following, we assign to a graded $\tE_1$-module $\cM$ over an $\tE_2$-algebra with stabilising object $X$ an augmented semi-simplicial space $\maptwoshort{R_\bullet(\cM)}{\cM}$ up to higher coherent homotopy, called the \emph{canonical resolution}. It will be defined as an augmented $\semisimpthick$-space for a topologically enriched category $\semisimpthick$ weakly equivalent to the semi-simplicial category, constructed from the underlying $\tE_{1,2}$-operad $\cO$. We begin by recalling the braided analogue of the category of finite sets and injections, as introduced in \cite{RWW}.

\begin{dfn}\label{definition:UB}Define the category $\UB$ with objects $[0],[1],\ldots$ as in $\semisimp$, no morphisms from $[q]$ to $[p]$ for $q>p$ and $\UB([q],[p])$ for $q\le p$ given by the cosets $\braidgroup_{p+1}/\braidgroup_{p-q}$, where $\braidgroup_i$ denotes the braid group on $i$ strands and $\braidgroup_{p-q}$ acts on $\braidgroup_{p+1}$ from the right as the first $(p-q)$ strands. The composition is defined as \[\mapnoname{\UB([l],[q])\times\UB([q],[p])}{\UB([l],[p])}{[b],[b']}{[b'(1^{p-q}\oplus b)],}\] where $1^{p-q}\oplus b$ is the braid obtained by inserting $(p-q)$ trivial strands to the left of $b$ (see Figure~\ref{figure:UB}).
\end{dfn}
\begin{figure}[h]
\centering
\begin{subfigure}{0.25\textwidth}
  \centering
  \begin{tikzpicture}[scale=0.5,main_node/.style={circle,draw,inner sep=0pt,minimum size=5pt]},fat_node/.style={circle,fill=.,draw,inner sep=0pt,minimum size=5pt]}]
    \node[main_node] (1) at (0,0){};
    \node[main_node] (2) at (1,0){};
     \node[fat_node] (3) at (2,0){};
    \node[fat_node] (5) at (0,-3){};
    \node[fat_node] (6) at (1,-3){};
     \node[fat_node] (7) at (2,-3){};
     \draw[-] (1) to [out=-90,in=90]  (6);
     \draw[preaction={draw, line width=3pt, white}][-] (3) to[out=-90,in=90]  (5);
     \draw[preaction={draw, line width=3pt, white}][-] (2) to[out=-90,in=90]  (7);
\end{tikzpicture}
  \caption*{$f\in\UB([0],[2])$}
\end{subfigure}
\begin{subfigure}{0.3\textwidth}
  \centering
  \begin{tikzpicture}[scale=0.5,main_node/.style={circle,draw,inner sep=0pt,minimum size=5pt]},fat_node/.style={circle,fill=.,draw,inner sep=0pt,minimum size=5pt]}]
  \node[main_node] (9) at (0,0){};
    \node[fat_node] (10) at (0,-3){};
    \node[main_node] (1) at (1,0){};
    \node[fat_node] (2) at (2,0){};
     \node[fat_node] (3) at (3,0){};
     \node[fat_node] (4) at (4,0){};
    \node[fat_node] (5) at (1,-3){};
    \node[fat_node] (6) at (2,-3){};
     \node[fat_node] (7) at (3,-3){};
     \node[fat_node] (8) at (4,-3){};
     \draw[-] (9) to [out=-90,in=90]  (5);
     \draw[preaction={draw, line width=3pt, white}][-] (2) to [out=-90,in=90]  (10);
     \draw[-] (3) to [out=-90,in=90]  (8);
     \draw[preaction={draw, line width=3pt, white}][-] (1) to [out=-90,in=90]  (7);
     \draw[preaction={draw, line width=3pt, white}][-] (4) to [out=-90,in=90]  (6);
\end{tikzpicture}
    \caption*{$g\in\UB([2],[4])$}
\end{subfigure}
\begin{subfigure}{0.4\textwidth}
  \centering
  \begin{tikzpicture}[scale=0.5,main_node/.style={circle,draw,inner sep=0pt,minimum size=5pt]},fat_node/.style={circle,fill=.,draw,inner sep=0pt,minimum size=5pt]},former_node/.style={shape=circle, fill=gray, draw=none,inner sep=0pt,minimum size=5pt]},]
       \node[main_node] (13) at (0,0){};
       \node[former_node] (14) at (0,-1.5){};
      \node[fat_node] (15) at (0,-3){};
     \node[main_node] (1) at (1,0){};
      \node[main_node] (2) at (2,0){};
    \node[main_node] (3) at (3,0){};
     \node[fat_node] (4) at (4,0){};
      \node[former_node] (5) at (1,-1.5){};
    \node[former_node] (6) at (2,-1.5){};
    \node[former_node] (7) at (3,-1.5){};
     \node[former_node] (8) at (4,-1.5){};
     \node[fat_node] (9) at (1,-3){};
    \node[fat_node] (10) at (2,-3){};
    \node[fat_node] (11) at (3,-3){};
     \node[fat_node] (12) at (4,-3){};     
        \draw[-] (13) to [out=-90,in=90]  (14);      
      \draw[-] (1) to [out=-90,in=90]  (5);   
     \draw[-] (2) to [out=-90,in=90]  (7);
     \draw[preaction={draw, line width=3pt, white}][-] (4) to[out=-90,in=90]  (6);
     \draw[preaction={draw, line width=3pt, white}][-] (3) to[out=-90,in=90]  (8);    
          \draw[-] (14) to [out=-90,in=90]  (9);
     \draw[preaction={draw, line width=3pt, white}][-] (6) to [out=-90,in=90]  (15);
     \draw[-] (7) to [out=-90,in=90]  (12);
     \draw[preaction={draw, line width=3pt, white}][-] (5) to [out=-90,in=90]  (11);
      \draw[preaction={draw, line width=3pt, white}][-] (8) to [out=-90,in=90]  (10);     
     \end{tikzpicture}
    \caption*{$(g\circ f)\in\UB([0],[4])$}
\end{subfigure}
\caption{The categorical composition of $\UB$}
\label{figure:UB}
\end{figure}
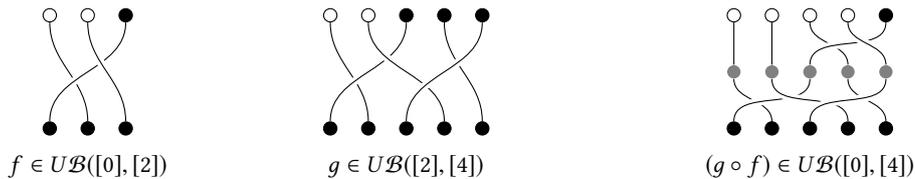
The category $\UB$ admits a canonical functor to the category $\FI$ of finite sets and injections by sending a class in $\braidgroup_{p+1}/\braidgroup_{p-q}$ to the injection obtained by following the last $(q+1)$ strands of a representing braid. Visualising $\UB$ as indicated by Figure~\ref{figure:UB}, two braids represent the same morphism if and only if they differ by a braid of the $\circ$-ends. Following the braids of the upper $\bullet$-ends to the lower ends gives the induced injections. This functor admits a section on the subcategory $\semisimp\subseteq \FI$, as shown by the next lemma for the statement of which we consider the collection of braid groups $\coprod_{n\ge0}B_n$ as the free braided monoidal category on one object $X$.

\begin{lem}\label{lemma:section}There is a unique functor $\maptwoshort{\semisimp}{\UB}$ that maps the face map $d_i\in\semisimp([p-1],[p])$ to \[[\braiding_{X^{\oplus i},X}^{-1}\oplus X^{\oplus p-i}]\in\UB([p-1],[p]).\] The composition of this functor with the functor $\maptwoshort{\UB}{\FI}$ described above agrees with the inclusion $\semisimp\subseteq \FI$.
\end{lem}
\begin{proof}To prove the first part, it is sufficient to check the face relations 
\[[\braiding_{X^{\oplus j},X}^{-1}\oplus X^{\oplus p+1-j}]\circ[\braiding_{X^{\oplus i},X}^{-1}\oplus X^{\oplus p-i}]=[\braiding_{X^{\oplus i},X}^{-1}\oplus X^{\oplus p+1-i}]\circ[\braiding_{X^{\oplus {j-1}},X}^{-1}\oplus X^{\oplus p-j+1}]\] for $i<j$ in $\UB([p-1,p+1])$. The left hand side agrees with the class of the braid
\[(\braiding_{X^{\oplus j},X}^{-1}\oplus X^{\oplus p+1-j})(X\oplus\braiding_{X^{\oplus i},X}^{-1}\oplus X^{\oplus p-i}),\] which, by applying braid relations, can be seen to agree with the braid \[(\braiding_{X^{\oplus i},X}^{-1}\oplus X^{\oplus p+1-i})(X\oplus\braiding_{X^{\oplus {j-1}},X}^{-1}\oplus X^{\oplus p-j+1})(\braiding_{X,X}^{-1}\oplus X^{\oplus p}),\] whose class in $\UB([p-1,p+1])=B_{p+2}/B_2$ coincides with the right hand side of the claimed equation. The proof is concluded by observing that the two functors $\maptwoshort{\semisimp}{\FI}$ in question agree on the face maps by construction, and thus on all of $\semisimp$.
\end{proof}

\begin{rem}In the language of \cite{RWW}, the category $\UB$ is the \emph{free pre-braided monoidal category on one object} \cite[Sect.\,1.2]{RWW}. Unwinding the definitions, their semi-simplicial set $W_n(A,X)$ associated to objects $A$ and $X$ of a pre-braided monoidal category $\cC$ (see \cite[Sect.\,2]{RWW}) agrees with the composition\[\semisimpop\longrightarrow\UB^{\op}
\longrightarrow\cC^{\op}\longrightarrow\Sets,\] in which the first arrow is the described section, the second is induced by $X$, and the third is $\cC(-,A\oplus X^{\oplus n})$.
\end{rem}
In the following, we introduce topological analogues of $\UB$ and $\semisimp$ for any $\tE_{1,2}$-operad $\cO$. To that end, we denote by $\cO(k)$ the space obtained from $\cO(\ocolour{m},\ocolour{a}^{k};\ocolour{m})$ by quotienting out the action of the symmetric group $\Sigma_k$ on the $\ocolour{a}$-inputs. To simplify the construction, we assume that the quotient maps $\maptwoshort{\cO(\ocolour{m},\ocolour{a}^{k};\ocolour{m})}{\cO(k)}$ are covering spaces, although this is not strictly necessary (see \cref{remark:generalisation}). As the operadic composition $\gamma$ on $\cO$ is equivariant, it induces composition maps $\mapwoshort{\gamma(-;-,1_{\ocolour{a}}^k)}{\cO(k)\times\cO(l)}{\cO(k+l)}$. The fixed operation $c\in \cO(1)$, used to define the stabilisation map, yields iterated operations $c_k\in\cO(k)$ by setting $c_0$ as the unit $1_\ocolour{m}$ and $c_{k+1}$ inductively as $\gamma(c;c_k,1_{\ocolour{a}})$. As a last preparotary step before defining the category $\UB$, we recall that we denote the endpoint of a Moore path $\mu$ by $\omega(\mu)$.

\begin{dfn}Define a topologically enriched category $U\cO=U(\cO,c)$ with objects $[0],[1],\ldots$ and
\[U\cO([q],[p])=\{(d,\mu)\in{\cO}(p-q)\times\Path_{c_{p+1}}{\cO}(p+1)\mid \omega(\mu)=\gamma(c_{q+1};d,1_{\ocolour{a}}^{q+1})\},\] where $\Path_{c_{p+1}}{\cO}(p+1)$ is the space of Moore paths in $\cO(p+1)$ starting at $c_{p+1}$. The composition is \[\mapnoname{U\cO([l],[q])\times U\cO([q],[p])}{U\cO([l],[p])}{\big((e,\zeta),(d,\mu)\big)}{\big(\gamma(e;d,1_{\ocolour{a}}^{q-l}),\mu\cdot\gamma(\zeta;d,1_{\ocolour{a}}^{q+1})\big),}\] as visualised by Figure~\ref{figure:UO}. Since we are using Moore paths, associativity and unitality follow from the respective properties of the operadic composition. 
\end{dfn}
\begin{figure}[h]
\centering
\begin{subfigure}{0.22\linewidth}
\centering
\begin{tikzpicture}[scale=0.45]
\draw [rounded corners] (0,-4)--(6,-4)--(5,-5)--(-1,-5)--cycle;
\draw[dashed] (2,-4)--(1,-5);
\draw[dashed] (4,-4)--(3,-5);

\draw[dotted] (0,0)--(0,-4);
\draw[dotted] (6,0)--(6,-4);
\draw[dotted] (5,-1)--(5,-5);	
\draw[dotted] (-1,-1)--(-1,-5);

\draw[preaction={draw, line width=3pt, white}][-] (1.5,-0.25) to [out=-90,in=90]  (2.5,-4.5);
\draw[preaction={draw, line width=3pt, white}][-] (4.5,-0.5) to [out=-90,in=90]  (0.5,-4.5);
\draw[preaction={draw, line width=3pt, white}][-] (2.5,-0.75) to [out=-90,in=90]  (4.5,-4.5);

\draw[rotate around={-60:(1.5,-0.25)},fill=white] (1.5,-0.25) ellipse (0.1cm and 0.2cm);
\draw[rotate around={-60:(2.5,-0.75))},fill=white] (2.5,-0.75) ellipse (0.1cm and 0.2cm);
\draw[rotate around={-60:(4.5,-0.5)},fill] (4.5,-0.5) ellipse (0.27cm and 0.4cm);

\draw[rotate around={-60:(0.5,-4.5)},fill] (0.5,-4.5) ellipse (0.27cm and 0.4cm);
\draw[rotate around={-60:(2.5,-4.5)},fill] (2.5,-4.5) ellipse (0.27cm and 0.4cm);
\draw[rotate around={-60:(4.5,-4.5)},fill] (4.5,-4.5) ellipse (0.27cm and 0.4cm);

\draw[rounded corners,preaction={draw, line width=3pt, white}] (0,0)--(6,0)--(5,-1)--(-1,-1)--cycle;
\draw[dashed] (4,0)--(3,-1);
\end{tikzpicture}
\caption*{$(e,\zeta)\in U\cO([0],[2])$}
\end{subfigure}
\begin{subfigure}{0.35\linewidth}
\centering
\begin{tikzpicture}[scale=0.45]
\draw [rounded corners] (0,-4)--(10,-4)--(9,-5)--(-1,-5)--cycle;
\draw[dashed] (2,-4)--(1,-5);
\draw[dashed] (4,-4)--(3,-5);
\draw[dashed] (6,-4)--(5,-5);
\draw[dashed] (8,-4)--(7,-5);

\draw[dotted] (0,0)--(0,-4);
\draw[dotted] (10,0)--(10,-4);
\draw[dotted] (9,-1)--(9,-5);
\draw[dotted] (-1,-1)--(-1,-5);

\draw[preaction={draw, line width=3pt, white}][-] (0.5,-0.75) to [out=-90,in=90]  (2.5,-4.5);
\draw[preaction={draw, line width=3pt, white}][-] (4.5,-0.5) to [out=-90,in=90]  (0.5,-4.5);
\draw[preaction={draw, line width=3pt, white}][-] (2.5,-0.25) to [out=-90,in=90]  (6.5,-4.5);
\draw[preaction={draw, line width=3pt, white}][-] (6.5,-0.5) to [out=-90,in=90]  (8.5,-4.5);
\draw[preaction={draw, line width=3pt, white}][-] (8.5,-0.5) to [out=-90,in=90]  (4.5,-4.5);

\draw[rotate around={-60:(0.5,-0.75)},fill=white] (0.5,-0.75) ellipse (0.1cm and 0.2cm);
\draw[rotate around={-60:(2.5,-0.25))},fill=white] (2.5,-0.25) ellipse (0.1cm and 0.2cm);
\draw[rotate around={-60:(4.5,-0.5)},fill] (4.5,-0.5) ellipse (0.27cm and 0.4cm);
\draw[rotate around={-60:(6.5,-0.5)},fill] (6.5,-0.5) ellipse (0.27cm and 0.4cm);
\draw[rotate around={-60:(8.5,-0.5)},fill] (8.5,-0.5) ellipse (0.27cm and 0.4cm);

\draw[rotate around={-60:(0.5,-4.5)},fill] (0.5,-4.5) ellipse (0.27cm and 0.4cm);
\draw[rotate around={-60:(2.5,-4.5)},fill] (2.5,-4.5) ellipse (0.27cm and 0.4cm);
\draw[rotate around={-60:(4.5,-4.5)},fill] (4.5,-4.5) ellipse (0.27cm and 0.4cm);
\draw[rotate around={-60:(6.5,-4.5)},fill] (6.5,-4.5) ellipse (0.27cm and 0.4cm);
\draw[rotate around={-60:(8.5,-4.5)},fill] (8.5,-4.5) ellipse (0.27cm and 0.4cm);

\draw [rounded corners,preaction={draw, line width=3pt, white}] (0,0)--(10,0)--(9,-1)--(-1,-1)--cycle;
\draw[dashed] (4,0)--(3,-1);
\draw[dashed] (6,0)--(5,-1);
\draw[dashed] (8,0)--(7,-1);
\end{tikzpicture}
\caption*{$(d,\mu)\in U\cO([2],[4])$}
\end{subfigure}
\begin{subfigure}{0.3\linewidth}
\centering
\begin{tikzpicture}[scale=0.45]
\draw [rounded corners] (0,-4)--(10,-4)--(9,-5)--(-1,-5)--cycle;
\draw[dashed] (2,-4)--(1,-5);
\draw[dashed] (4,-4)--(3,-5);
\draw[dashed] (6,-4)--(5,-5);
\draw[dashed] (8,-4)--(7,-5);

\draw[rounded corners] (0.2,-2)--(9.8,-2);

\draw[dotted] (0,0)--(0,-4);
\draw[dotted] (10,0)--(10,-4);
\draw[dotted] (9,-1)--(9,-5);
\draw[dotted] (-1,-1)--(-1,-5);

\draw[preaction={draw, line width=3pt, white}][-] (0.5,-0.75) to [out=-90,in=90]  (0.5,-2.75);
\draw[preaction={draw, line width=3pt, white}][-] (2.5,-0.25) to [out=-90,in=90]  (2.5,-2.25);
\draw[preaction={draw, line width=3pt, white}][-] (5.5,-0.25) to [out=-90,in=90]  (6.5,-2.5);
\draw[preaction={draw, line width=3pt, white}][-] (8.5,-0.5) to [out=-90,in=90]  (4.5,-2.5);
\draw[preaction={draw, line width=3pt, white}][-] (6.5,-0.75) to [out=-90,in=90]  (8.5,-2.5);

\draw[preaction={draw, line width=3pt, white}][-] (0.5,-2.75) to [out=-90,in=90]  (2.5,-4.5);
\draw[preaction={draw, line width=3pt, white}][-] (4.5,-2.5) to [out=-90,in=90]  (0.5,-4.5);
\draw[preaction={draw, line width=3pt, white}][-] (2.5,-2.25) to [out=-90,in=90]  (6.5,-4.5);
\draw[preaction={draw, line width=3pt, white}][-] (6.5,-2.5) to [out=-90,in=90]  (8.5,-4.5);
\draw[preaction={draw, line width=3pt, white}][-] (8.5,-2.5) to [out=-90,in=90]  (4.5,-4.5);

\draw[rotate around={-60:(0.5,-0.75)},fill=white] (0.5,-0.75) ellipse (0.1cm and 0.2cm);
\draw[rotate around={-60:(2.5,-0.25))},fill=white] (2.5,-0.25) ellipse (0.1cm and 0.2cm);
\draw[rotate around={-60:(6.5,-0.75))},fill=white] (6.5,-0.75) ellipse (0.1cm and 0.2cm);
\draw[rotate around={-60:(5.5,-0.25)},fill=white] (5.5,-0.25) ellipse (0.1cm and 0.2cm);
\draw[rotate around={-60:(8.5,-0.5)},fill] (8.5,-0.5) ellipse (0.27cm and 0.4cm);

\fill[rotate around={-60:(0.5,-2.75)},gray] (0.5,-2.75) ellipse (0.1cm and 0.2cm);
\fill[rotate around={-60:(2.5,-2.25))},fill=gray] (2.5,-2.25) ellipse (0.1cm and 0.2cm);
\fill[rotate around={-60:(4.5,-2.5)},fill=gray] (4.5,-2.5) ellipse (0.27cm and 0.4cm);
\fill[rotate around={-60:(6.5,-2.5)},gray] (6.5,-2.5) ellipse (0.27cm and 0.4cm);
\fill[rotate around={-60:(8.5,-2.5)},gray] (8.5,-2.5) ellipse (0.27cm and 0.4cm);

\draw[rotate around={-60:(0.5,-4.5)},fill] (0.5,-4.5) ellipse (0.27cm and 0.4cm);
\draw[rotate around={-60:(2.5,-4.5)},fill] (2.5,-4.5) ellipse (0.27cm and 0.4cm);
\draw[rotate around={-60:(4.5,-4.5)},fill] (4.5,-4.5) ellipse (0.27cm and 0.4cm);
\draw[rotate around={-60:(6.5,-4.5)},fill] (6.5,-4.5) ellipse (0.27cm and 0.4cm);
\draw[rotate around={-60:(8.5,-4.5)},fill] (8.5,-4.5) ellipse (0.27cm and 0.4cm);

\draw [rounded corners,preaction={draw, line width=3pt, white}] (0,0)--(10,0)--(9,-1)--(-1,-1)--cycle;
\draw[dashed] (4,0)--(3,-1);
\draw[dashed] (8,0)--(7,-1);

\draw [rounded corners,preaction={draw, line width=3pt, white}] (9.7,-2)--(10,-2)--(9,-3)--(-1,-3)--(0,-2)--(0.3,-2);
\draw[dashed] (4,-2)--(3,-3);
\draw[dashed] (6,-2)--(5,-3);
\draw[dashed] (8,-2)--(7,-3);
\end{tikzpicture}
\caption*{$(d,\mu)(e,\zeta)\in U\cO([0],[4])$}
\end{subfigure}
\caption{The categorical composition of $U\cO$}
\label{figure:UO}
\end{figure}

The construction $U(-)$ is functorial in $(\cO,c)$ and preserves weak equivalences, since $U\cO([q],[p])$ agrees with the homotopy fibre at $c_{p+1}$ of the map $\mapwoshort{\gamma(c_{q+1};-,1_{\ocolour{a}}^{q+1})}{{\cO}(p-q)}{{\cO}(p+1)}$.

\begin{rem}Using Quillen's bracket-construction $\langle-,-\rangle$ for modules over monoidal categories (see \cite[p.\,219]{Quillen}), the category $\UB$ is given by $\langle\cB,\cB\rangle$, where $\cB=\coprod_{n\ge0}\braidgroup_n$ is the free braided monoidal category acting on itself. Similarly, $U\cO$ can be obtained via an analogue of Quillen's construction for monoidal categories internal to spaces, applied to the path-category of the monoid $\coprod_{n\ge0}{\cO}(n)$.
\end{rem}
\begin{lem}\label{lemma:homotopydiscrete}The category $U\cO$ is homotopy discrete and satisfies $\pi_0(U\cO)\cong \UB$.
\end{lem}
Before turning to the proof of \cref{lemma:homotopydiscrete}, we suggest the reader to compare Figure~\ref{figure:UB} with Figure~\ref{figure:UO}.
\begin{proof}As $U(-)$ preserves weak equivalences, it suffices to prove the claim for $\cO=\SC_2$. Mapping embeddings of discs to their centre yields a homotopy equivalence from the space of operations $\SC_2(n)$  to the unordered configuration space $\Conf_n(\bfR^2)$ of the plane, which is an Eilenberg--MacLane space $K(B_n,1)$ for the braid group $\braidgroup_n$. On fundamental groups, the map $\mapwoshort{\gamma(c_{q+1};-,1_{\ocolour{a}}^{q+1})}{{\cO}(p-q)}{{\cO}(p+1)}$ is injective, since it is given by including $B_{p-q}$ in $B_{p+1}$ as the first $(p-q)$ strands. From this, one concludes that its homotopy fibre $\hofib_{c_{p+1}}(\gamma(c_{q+1};-,1_{\ocolour{a}}^{q+1}))=U\cO([q],[p])$ is homotopy discrete with path components $\braidgroup_{p+1}/\braidgroup_{p-q}$ and that, via this equivalence, the composition coincides with that of $\UB$, proving the claim.
\end{proof}

Equipped with \cref{lemma:homotopydiscrete}, we fix an isomorphism $\pi_0(U\cO)\cong \UB$ once and for all, which we use, for instance, to identify $\pi_1(\cO(p+1),c_{p+1})\cong\pi_0(U\cO([p],[p]))$ with the braid group $B_{p+1}$.

\begin{dfn}\label{definition:semisimplicialthickening}The \emph{thickening of the semisimplicial category} associated to an $\tE_{1,2}$-operad $\cO$ is the subcategory $\semisimpthick\subseteq U\cO$ obtained by restricting $U\cO$ to the path components hit by the section $\semisimp\rightarrow\UB\cong\pi_0(U\cO)$ of \cref{lemma:section}. It comes with a weak equivalence to $\semisimp$, induced by the functor $\maptwoshort{U\cO}{\FI}$.
\end{dfn}

Before proceeding to the central definitions of this section, we remind the reader of the theory of augmented $\cC$-spaces for a topologically enriched category $\cC$, set up in \cref{section:homotopycolimits}.

\begin{dfn}Let $\cM$ be a graded $\tE_1$-module over an $\tE_2$-algebra with structure maps $\theta$ and stabilising object $X$. Define the contravariant $U\cO$-space $B_\bullet(\cM)$ by sending $[p]$ to the path-space construction of $s^{p+1}$,
\[B_p(\cM)=\{(A,\zeta)\in\cM\times\Path\cM\mid\omega(\zeta)=s^{p+1}(A)\},\] and by
\[\mapnoname{U\cO([q],[p])\times B_p(\cM)}{B_q(\cM)}{\big((d,\mu),(A,\zeta)\big)}{\big(\theta(d;A,X^{p-q}),\zeta\cdot\theta(\mu;A,X^{p+1})\big).}\] Functoriality follows from the associativity of the module-structure $\theta$ and the composition of Moore paths. Evaluating paths at zero defines an augmentation $\maptwoshort{B_\bullet(\cM)}{\cM}$, which is a levelwise fibration.
\end{dfn}

\begin{dfn}\label{definition:canonicalresolution}Let $\cM$ be a graded $\tE_1$-module over an $\tE_2$-algebra with stabilising object $X$.
\begin{enumerate}
\item The \emph{canonical resolution} of $\cM$ is the fibrant augmented $\semisimpthick$-space
\[\maptwo{R_\bullet(\cM)}{\cM}\] obtained by restricting the augmented $U\cO$-space $B_\bullet(\cM)$ to the semi-simplicial thickening $\semisimpthick\subseteq U\cO$.
\item The \emph{space of destabilisations} of a point $A\in\cM$ is the $\semisimpthick$-space $W_\bullet(A)$ defined as the fibre of the canonical resolution $\maptwoshort{R_\bullet(\cM)}{\cM}$ at $A$.
\end{enumerate}
\end{dfn}

Unwrapping the definition, the canonical resolution $\maptwoshort{R_\bullet(\cM)}{\cM}$ is an augmented semi-simplicial space up to higher coherent homotopy with $p$-simplices \[R_p(\cM)=\{(A,\zeta)\in\cM\times\Path\cM\mid\omega(\zeta)=s^{p+1}(A)\},\] augmented over $\cM$ by evaluating paths at zero. There is a contractible space of $i$th face maps, but the following lemma provides a particularly convenient one after choosing a loop $\mu_i\in\Omega_{c_{p+1}}\cO(p+1)$ corresponding to the braid $\braiding_{X^{\oplus i},X}^{-1}\oplus X^{\oplus p-i}$ via the fixed isomorphism $B_{p+1}\cong \pi_1(\cO(p+1),c_{p+1})$. 

\begin{lem}\label{lemma:facemaps}The morphism $(c,\mu_i)\in U\cO([p-1],[p])$ lies in the path component of the image of the $i$th face map $d_i\in\semisimp([p-1],[p])$ in $\UB([p-1],[p])\cong\pi_0U\cO([p-1],[p])$ via the section of \cref{lemma:section}, so the map
\[\map{\tilde{d_i}}{R_p(\cM)}{R_{p-1}(\cM)}{(A,\zeta)}{\big(s(A),\zeta\cdot\theta(\mu_i;A,X^{p+1})\big)}\] is an $i$th face map of the canonical resolution $\maptwoshort{R_\bullet(\cM)}{\cM}$.
\end{lem}
\begin{proof}The choice of $\mu_i$ ensures that, via the isomorphism $\pi_0(U\cO([p-1],[p]))\cong \UB([p-1],[p])$, the element $(c,\mu_i)$ is in the component of the class $[\braiding_{X^{\oplus i},X}^{-1}\oplus X^{\oplus p-i}]$ in $\UB([p-1],[p])$. This is exactly the image of $d_i\in\semisimp([p-1],[p])$ in $\UB([p-1],[p])$, as claimed.
\end{proof}

\begin{rem}We borrowed the term \emph{space of destabilisations} from \cite{RWW}, where it stands for certain semi-simplicial sets $W_n(A,X)$ associated to a braided monoidal groupoid. In \cref{section:comparisonRWW}, it is explained that these semi-simplicial sets are special cases of the spaces of destabilisations in our sense.
\end{rem}

\begin{rem}\label{remark:connectivityresolutionbyspacesofdestabilisation}As $\maptwoshort{R_\bullet(\cM)}{\cM}$ is fibrant, its fibre $W_\bullet(A)$ is equivalent to its homotopy fibre $\hofib_A(R_\bullet(\cM))$, so by virtue of \cref{section:homotopyfibrerealisation}, the homotopy fibre at $A$ of the realisation $\maptwoshort{|R_\bullet(\cM)|}{\cM}$ is equivalent to $|W_\bullet(A)|$. In particular, the canonical resolution of $\cM$ is graded $\varphi(g_\cM)$-connected in degree $\ge m$ for a function $\mapwoshort{\varphi}{\bfNext}{\bfQ\cup\{\infty\}}$ satisfying $\phi(\infty)=\infty$ if and only if the spaces of destabilisations $W_\bullet(A)$ are $(\lfloor \varphi(g_\cM(A))\rfloor-1)$-connected for all points $A\in\cM$ with finite degree $g_\cM(A)\ge m$. As points in the same component have equivalent homotopy fibres, it is sufficient to check one point in each component.
\end{rem}

\begin{ex}\label{example:freealgebra}Recall the \emph{free $\tE_2$-algebra on a point} $\cO^{\ocolour{a}}=\coprod_{n\ge0}\cO(\ocolour{a}^n;\ocolour{a})/\Sigma_n$, graded in the evident way, with the \emph{free $\tE_1$-module on a point} $\cO^{\ocolour{m}}=\coprod_{n\ge0}\cO(\ocolour{m};\ocolour{a}^n;\ocolour{m})/\Sigma_n$ as a graded $\tE_1$-module over it. Choosing the unit $1_{\ocolour{a}}\in\cO(\ocolour{a};\ocolour{a})$ as the stabilising object, the space of destabilisations $W_\bullet(c_{p+1})$ is the $\semisimpthick$-space obtained by restricting the $U\cO$-space $U\cO(\bullet,[p])$ to $\semisimpthick$. As the category $U\cO$ is homotopy discrete with $\pi_0(U\cO)\cong \UB$ by \cref{lemma:homotopydiscrete}, the $\semisimpthick$-space $W_\bullet(c_{p+1})$ is equivalent to the semi-simplicial set given as the composition of the section $\maptwoshort{\semisimpop}{\UB^{\op}}$ of \cref{lemma:section} with $\UB(\bullet,[p])$. Using \cite[Prop.~3.2]{HatcherVogtmann}, the realisation of this semi-simplicial set can be seen to be contractible, but we do not go into details, since the consequences of Theorems~\ref{theorem:constant} and~\ref{theorem:twisted} regarding the twisted homological stability of $K(B_n,1)\simeq\cO(\ocolour{m};\ocolour{a}^n;\ocolour{m})/\Sigma_n$ correspond to the case $M=D^{2}$ and $\pi=\id$ of \cref{theorem:configurationspaces}, which is proved in \cref{section:configurationspaces}.
\end{ex}

\begin{rem}\label{remark:independentofalgebra}The choice of a stabilising object $X\in\cA$ for a graded $\tE_1$-module $\cM$ over an $\tE_2$-algebra $\cA$ induces a graded $\tE_1$-module structure on $\cM$ over $\cO^{\ocolour{a}}$. The two canonical resolutions of $\cM$ when considered as an module over $\cA$ or over $\cO^{\ocolour{a}}$ are identical. In fact, all our constructions and results solely depend on the induced module structure of $\cM$ over $\cO^{\ocolour{a}}$ and are in that sense independent of $\cA$.
\end{rem}

\begin{rem}\label{remark:localisation}Let $\cM$ be a graded $\tE_1$-module over an $\tE_2$-algebra with stabilising object $X$ and consider $\cM$ as a graded $\tE_1$-module over $\cO^{\ocolour{a}}$ (see \cref{remark:independentofalgebra}). For a union of path components $\cM'\subseteq \cM$ that is closed under the multiplication by $X$, we define a new grading on $\cM$ as an $\tE_1$-module over $\cO^{\ocolour{a}}$ by modifying the original grading on $\cM'$ by assigning the complement of $\cM'$ degree $\infty$ and leaving the grading on $\cM'$ unchanged. We call $\cM$ with this new grading the \emph{localisation at $\cM'$}. An example for such a subspace $\cM'$ is given by the \emph{objects stably isomorphic to an object $A\in\cM$} by which we mean the union of the path components of objects $B$ for which $B\oplus X^{\oplus n}$ is in the component of  $A\oplus X^{\oplus m}$ for some $n,m\ge0$.
\end{rem}

\begin{ex}\label{example:groupaction}Let $\cM$ be a graded $\tE_1$-module over an $\tE_2$-algebra $\cA$ and let $G$ be a group acting on $\cM$, preserving the grading. If the actions of $\cA$ and $G$ on $\cM$ commute, then the Borel construction $EG\times_G\cM$ inherits a graded $\tE_1$-module structure. The choice of a point in $EG$ induces a morphism 
\begin{center}
\begin{tikzcd}[row sep=0.5cm]
R_\bullet(\cM)\arrow[r]\arrow[d] & R_\bullet(EG\times_G\cM)\arrow[d]\\
\cM \arrow[r] & EG\times_G\cM
\end{tikzcd}
\end{center}
of augmented $\semisimpthick$-spaces, which induces weak equivalences on homotopy fibres. An application of \cref{section:homotopyfibrerealisation} implies that the respective canonical resolutions have the same connectivity.
\end{ex}

\begin{rem}\label{remark:generalisation}Some constructions of this section work in greater generality. The category $U\cO$ and the augmented $U\cO$-space $B_\bullet(\cM)$ can be defined for any coloured operad. $U\cO$ then still admits a functor to $\FI$, but might not be homotopy discrete nor admit a section on $\semisimp\subseteq\FI$. The point-set assumption on the action of $\Sigma_k$ on $\cO(\ocolour{m},\ocolour{a}^k;\ocolour{m})$ can be avoided by constructing $U\cO$ using $\cO(\ocolour{m},\ocolour{a}^k;\ocolour{m})$ instead of ${\cO}(k)$, which involves taking care of permutations corresponding to preimages of the quotient map $\maptwoshort{\cO(\ocolour{m},\ocolour{a}^k;\ocolour{m})}{{\cO}(k)}$.
\end{rem}

\subsection{The stable genus}\label{section:stablegenus}We extend the notion of the \emph{stable genus} of a manifold, as introduced in \cite{GRWI}, to our context, providing us with a general way of grading modules over braided monoidal categories and, by \cref{remark:gradingmoduleisgradingoncategory}, also of grading $\tE_1$-modules over $\tE_2$-algebras.

Let $(\cM,\oplus)$ be a right-module over a braided monoidal category $(\cA,\oplus,\braiding,0)$. Recall the free braided monoidal category on one object $\cB=\coprod_{n\ge0}B_n$, consisting of the the braid groups $B_n$. A choice of an object $X$ in $\cA$ induces a functor $\cB\rightarrow \cA$ and hence a right-module structure on $\cM$ over $\cB$. With respect to this module structure, a grading of $\cM$ that is compatible with the canonical grading on $\cB$ is equivalent to a grading $g_\cM$ on $\cM$ as a category such that $g_\cM(A\oplus X)=g_\cM(A)+1$ holds for all objetcs $A$ in $\cM$.

\begin{dfn}\label{definition:stablegenus}Let $X$ be an object of $\cA$ and $A$ an object of $\cM$.
\begin{enumerate}
\item The \emph{$X$-genus} of $A$ is defined as \[g^X(A)=\sup\{k\ge0\mid\text{there exists an object $B$ in $\cM$ with }B\oplus X^{\oplus k}\cong A\}\in\bfNext.\]
\item The \emph{stable $X$-genus} of $A$ is defined as \[\bar{g}^X(A)=\sup\{g^X(A\oplus X^{\oplus k})-k\mid k\ge0\}\in\bfNext.\]
\end{enumerate}
\end{dfn}

As $\bar{g}^X(A\oplus X)=\bar{g}^X(A)+1$ holds by definition, the stable $X$-genus provides a grading of $\cM$ when considered as a module over $\cB$ via $X$. This stands in contrast with the (unstable) $X$-genus, which does in general not define a grading, because the inequality $g^X(A)+1\le g^X(A\oplus X)$ might be strict. For an $\tE_1$-module $\cM$ over an $\tE_2$-algebra $\cA$, the choice of a point $X\in\cA$ induces an $\tE_1$-module structure on $\cM$ over the free $\tE_2$-algebra on a point $\cO^{\ocolour{a}}$ (see \cref{remark:independentofalgebra}). After taking fundamental groupoids, this results in the module structure of $\Pi(\cM)$ over $\cB$ discussed above, so the stable $X$-genus provides a grading for $\cM$ as an $\tE_1$-module over $\cO^{\ocolour{a}}$.

\begin{rem}\label{remark:genusisstablegenus}If the connectivity assumption of \cref{theorem:constant} is satisfied for an $\tE_1$-module $\cM$, graded with the stable $X$-genus, then the cancellation result \cref{corollary:cancellation} implies $g^X(A\oplus X)=g^X(A)+1$ for objects $A$ of positive stable genus, which, in turn, implies that for such $A$, the genus and the stable genus coincide.
\end{rem}

\section{Stability with constant and abelian coefficients}
\label{section:proofconstant}
Let $\cM$ be a graded $\tE_1$-module over an $\tE_2$-algebra with stabilising object $X$ and structure maps $\theta$. We prove \cref{theorem:constant} via a spectral sequence obtained from the canonical resolution $\maptwoshort{R_\bullet(\cM)}{\cM}$. All spaces $R_p(\cM)$ and $|R_\bullet(\cM)|$ are considered graded by pulling back the grading from $\cM$ along the augmentation.

\subsection{The spectral sequence}\label{section:constructionspectralsequence}Given a local system $L$ on $\cM$, the canonical resolution $\maptwoshort{R_\bullet(\cM)}{\cM}$ (see \cref{section:complexes}) gives rise to a tri-graded spectral sequence
\begin{equation}\label{equation:ssconstant}
E^1_{p,q,n}\cong \begin{cases} \oH_q(R_p(\cM)_n;L) &\mbox{if } p\ge 0 \\  \oH_q(\cM_{n};L) & \mbox{if } p=-1 \end{cases} \implies \oH_{p+q+1}(\cM_{n},|R_\bullet(\cM)|_n;L),
\end{equation}
with differential $\mapwoshort{d^1}{E^1_{p,q,n}}{E^1_{p-1,q,n}}$, induced by the augmentation for $p=0$ and the alternating sum $\sum_{i=0}^p(-1)^i(\tilde{d}_i;\id)_*$ for $p>0$, where $\tilde{d}_i$ is any choice of $i$th face map of $R_\bullet(\cM)$ (see Sections~\ref{section:semisimplicial} and~\ref{section:semisimplicialthick}). As the differentials do not change the $n$-grading, it is a sum of spectral sequences, one for each $n\in\bfNext$. To identify the $E^1$-page in terms of the stabilisation $\mapwoshort{s}{\cM}{\cM}$, recall from \cref{section:modules} that the fundamental groupoid $(\Pi(\cM),\oplus)$ is a graded module over the graded braided monoidal groupoid $(\Pi(\cA),\oplus,\braiding,0)$.

\begin{lem}\label{lemma:differentialszero}
We have $E^1_{p,q,n+1}\cong\oH_q(\cM_{n-p};(s^{p+1})^*L)$ and $\mapwoshort{d^1}{E^1_{p,q,n+1}}{E^1_{p-1,q,n+1}}$ identifies with
\[\textstyle{\mapwoshort{\sum_{i=0}^p(-1)^i(s;\eta_i)_*}{\oH_q\big(\cM_{n-p};(s^{p+1})^*L\big)}{\oH_q\big(\cM_{n-p+1};(s^{p})^*L\big)}},\] where $\eta_i$ denotes the natural transformation \[\mapwo{L(-\oplus\braiding_{X^{\oplus i},X}\oplus X^{\oplus p-i})}{L(-\oplus X^{\oplus p+1})}{L(-\oplus X^{\oplus p+1})}.\] In particular, $d^1$ corresponds for $p=0$ to the stabilisation $\mapwoshort{(s;\id)_*}{\oH_q(\cM_{n};s^*L)}{\oH_q(\cM_{n+1};L)}$. Thus, if $L$ is constant, $d^1$ identifies with $\mapwoshort{s_*}{\oH_q(\cM_{n-p};L)}{\oH_q(\cM_{n-p+1};L)}$ for $p$ even and vanishes for $p$ odd.
\end{lem}
\begin{proof}Using the choice of face maps $\mapwoshort{\tilde{d}_i}{R_{p}(\cM)_{n+1}}{R_{p-1}(\cM)_{n+1}}$ of \cref{lemma:facemaps}, we consider the square
\begin{equation*}
\begin{tikzcd}[row sep=0.6cm]
{(\cM_{n-p};(s^{p+1})^*L)}\arrow[d,"{(s;\eta_i)}",swap]\arrow[r,"(\iota_{p};\id)"]&{(R_{p}(\cM)_{n+1};L)}\arrow[d,"(\tilde{d_i};\id)"]\\
{(\cM_{n-p+1};(s^{p})^*L)}\arrow[r,"(\iota_{p-1};\id)"]&{(R_{p-1}(\cM)_{n+1};L)},
\end{tikzcd}
\end{equation*} where $\iota_{q}$ denotes the canonical equivalence $\maptwoshort{\cM_{n-q}}{R_q(\cM)_{n+1}}$, mapping $A$ to $(A,\text{const}_{s^{q+1}(A)})$. A point $A\in\cM_{n-p}$ is mapped by the two compositions in the square to $(s(A),\theta(\mu_i;A,X^{p+1}))$ and $(s(A),\const_{s^{p+1}(A)})$, respectively, which are connected by a preferred homotopy following the path $\mu_i$ chosen in \cref{lemma:facemaps} to its endpoint. 
The commutativity of the triangle \eqref{equation:localcoefficientsdiagram} ensuring that this homotopy extends to one of spaces with local systems (see \cref{section:localcoefficients}) is equivalent to $L(\theta(\mu_i;-,X^{p+1}))\eta_i=\id$. But, by the choice of $\mu_i$, the path $\theta(\mu_i;-,X^{p+1})$ corresponds to the braid $-\oplus\braiding_{X^{\oplus i},X}^{-1}\oplus X^{\oplus p-i}$, so the required relation holds and the square commutes up to homotopy. Taking vertical mapping cones and homology results in the claimed identification. If $L$ is constant, the $\eta_i$ coincide for all $i$, so the terms in the alternating sum cancel out.
\end{proof}

\begin{lem}\label{lemma:compositionshomotopic}If the local system $L$ is abelian, then the following compositions are homotopic for all $0\le i\le p$ \[(\cM;(s^{p+2})^*L)\xlongrightarrow{(s;\id)}(\cM;(s^{p+1})^*L)\xlongrightarrow{(s;\eta_i)}(\cM;(s^{p})^*L).\]
\end{lem}

The proof of \cref{lemma:compositionshomotopic} uses a self-homotopy of $\mapwoshort{s^2}{\cM}{\cM}$ which is crucial for various other arguments. Using the notation of \cref{section:complexes}, it is given by \begin{equation}\label{equation:selfhomotopy}\mapnonameshort{[0,1]\times\cM}{\cM}{(t,A)}{\theta(\mu(t);A,X^{2}),}\end{equation} where $\mu$ is a choice of loop of length $1$, based at $c_2\in\cO(2)$, such that $[(1_\ocolour{m},\mu)]\in\pi_0(U\cO([1],[1]))$ corresponds to the class $[\braiding_{X,X}^{-1}]\in\UB([1],[1])$ via the isomorphism $\pi_0(U\cO)\cong \UB$ fixed in \cref{section:complexes}. Since $\mu$ is unique up to homotopy, this describes the homotopy of $s^2$ uniquely up to homotopy of homotopies.

\begin{proof}[Proof of \cref{lemma:compositionshomotopic}]By the recollection of \cref{section:localcoefficients}, the selfhomotopy \eqref{equation:selfhomotopy} of $s^2$ extends to a homotopy of maps of spaces with local systems between the $i$th and $(i+1)$st composition in question if the triangle
\begin{equation*}
\begin{tikzcd}[row sep=0.6cm,column sep=4cm]
L(-\oplus X^{\oplus p+2})
\arrow[r,"L(-\oplus X\oplus \braiding_{X^{\oplus i+1},X}\oplus X^{\oplus p-i-1})"]\arrow[dr,swap,"L(-\oplus X\oplus \braiding_{X^{\oplus i},X}\oplus X^{\oplus p-i})"] 
& L(-\oplus X^{\oplus p+2})\arrow[d,"L(-\oplus \braiding^{-1}_{X,X}\oplus X^{\oplus p})"]
\\
&L(-\oplus X^{\oplus p+2})
\end{tikzcd}
\end{equation*}
commutes. The braid relations give $(-\oplus X\oplus \braiding_{X^{\oplus i},X}\oplus X^{\oplus p-i})=(-\oplus \braiding^{-1}_{X,X}\oplus X^{\oplus p+1})(-\oplus \braiding_{X^{\oplus i+1},X}\oplus X^{\oplus p-i})$, so the claim follows if we show that $[\braiding_{X^{\oplus i+1},X}\oplus X]=[X\oplus\braiding_{X^{\oplus i+1},X}]$ holds in the abelianisation. But the braid relation $(\braiding_{X,X}\oplus X)(X\oplus\braiding_{X,X})(\braiding_{X,X}\oplus X)=(X\oplus\braiding_{X,X})(\braiding_{X,X}\oplus X)(X\oplus \braiding_{X,X})$ abelianises to $[\braiding_{X,X}\oplus X]=[X\oplus\braiding_{X,X}]$ from which the claimed identity follows by induction on $i$.
\end{proof}

\subsection{The proof of \cref{theorem:constant}}
We prove \cref{theorem:constant} by induction on $n$, using the spectral sequence \eqref{equation:ssconstant}. As $\maptwoshort{|R_\bullet(\cM)|_{n+1}}{\cM_{n+1}}$ is assumed to be $(\frac{n-1+k}{k})$-connected for a $k\ge2$ in the constant or a $k\ge 3$ in the abelian coefficients case, the summand of degree $(n+1)$ of the spectral sequence converges to zero in the range $p+q\le\frac{n-1}{k}$. By \cref{lemma:differentialszero}, the differential $\mapwoshort{d^1}{E^1_{0,i,n+1}}{E^1_{-1,i,n+1}}$ identifies with the stabilisation map $\mapwoshort{(s;\id)_*}{\oH_i(\cM_{n};s^*L)}{\oH_i(\cM_{n+1};L)}$. Since there are no differentials targeting $E^k_{-1,0,n+1}$ for $k\ge1$, the stabilisation has to be surjective for $i=0$ if $E^\infty_{-1,0,n+1}$ vanishes, which is the case, since we have $-1\le\frac{n-1}{k}$ for all $n\ge 0$. In particular, this implies the case $n=0$ for both constant and abelian coefficients, because the isomorphism claims for $n=0$ are vacuous.

\begin{proof}[Constant coefficients]Assume the claim for constant coefficients holds in degrees smaller than $n$. By \cref{lemma:differentialszero}, the differential $\mapwoshort{d^1}{E^1_{p,q,n+1}}{E^1_{p-1,q,n+1}}$  identifies with $\mapwoshort{s_*}{\oH_q(\cM_{n-p};L)}{\oH_q(\cM_{n-p+1};L)}$ for $p$ even and is zero for $p$ odd. From the induction assumption, we draw the conclusion that $E^2_{p,q,n+1}$ vanishes for $(p,q)$ if $p$ is even with $0<p\le n$ and $q\le\frac{n-p-1}{k}$, and for $(p,q)$ if $p$ is odd with $0\le p<n$ and $q\le\frac{n-p-3+k}{k}$. So in particular, $E^2_{p,q,n+1}$ vanishes if both $0<p<n$ and $q\le\frac{n-p-1}{k}$ hold. As $\mapwoshort{d^1}{E^1_{1,i,n+1}}{E^1_{0,i,n+1}}$ is zero for all $i$, injectivity of $\mapwoshort{s_*}{\oH_i(\cM_{n};L)}{\oH_i(\cM_{n+1};L)}$ holds if $E^\infty_{0,i,n+1}=0$ and $E^2_{p,q,n+1}=0$ hold for $p+q=i+1$ with $q<i$. This is the case for $i\le\frac{n-1}{k}$, as claimed, which follows from the established vanishing ranges of $E^{\infty}$ and $E^2$. Similarly, the map in question is surjective in degree $i$ if $E^\infty_{-1,i,n+1}=0$ and $E^2_{p,q,n+1}=0$ hold for $p+q=i$ with $q<i$, which is true for $i\le\frac{n-1+k}{k}$.
\end{proof}

\begin{proof}[Abelian coefficients]Assume the statement holds for degrees smaller than $n$. The differential $\mapwoshort{d^1}{E^1_{p,q,n+1}}{E^1_{p-1,q,n+1}}$ identifies with $\mapwoshort{\sum_{i}(-1)^i(s,\eta_i)_*}{\oH_q(\cM_{n-p};(s^{p+1})^*L)}{\oH_q(\cM_{n-p+1};(s^{p})^*L)}$. By \cref{lemma:compositionshomotopic}, in the range where $\mapwoshort{(s,\id)_*}{\oH_q(\cM_{n-p-1};(s^{p+2})^*L)}{\oH_q(\cM_{n-p};(s^{p+1})^*L)}$ is surjective, $d^1$ identifies with the stabilisation map $\mapwoshort{(s,\id)_*}{\oH_q(\cM_{n-p};(s^{p+1})^*L)}{\oH_q(\cM_{n-p+1};(s^{p})^*L)}$ for $p$ even and vanishes for $p$ odd. By induction, this happens for $(p,q)$ with $0\le p\le n-1$ and $q\le\frac{n-p-1}{k}$, so by using the induction again, we conclude $E^2_{p,q,n+1}=0$ for $(p,q)$ if $p$ is even satisfying $0<p\le n-1$ and $q\le\frac{n-p+1-k}{k}$, and for $(p,q)$ with $p$ odd satisfying $0\le p< n-1$ and $q\le\frac{n-p-2}{k}$. The rest of the argument proceeds as in the constant case, adapting the ranges and using that $\mapwoshort{d^1}{E^1_{1,i,n+1}}{E^1_{0,i,n+1}}$ is zero for $i\le\frac{n-1}{k}$. 
\end{proof}

\begin{rem}\label{remark:improvement}If $g_\cM$ is a grading of $\cM$, then so is $g_\cM+m$ for any fixed number $m\ge0$. Consequently, if the canonical resolution of $\cM$ is graded $(\frac{g_\cM-m+k}{k})$-connected for an $m\ge2$, then we can apply Theorems~\ref{theorem:constant} and~\ref{theorem:twisted} to $\cM$, graded by $g_\cM+(m-2)$, which results in a shift in the stability range. By adapting the ranges in the previous proof appropriately, requiring more specific connectivity assumptions improve the stability ranges in \cref{theorem:constant} as follows.
\begin{enumerate}
\item If the canonical resolution is graded $(\frac{g_\cM-m+k}{k})$-connected for an $m\ge 3$, the surjectivity range in \cref{theorem:constant} for constant coefficients can be improved from $i\le\frac{n-m+k}{k}$ to $i\le\frac{n-m+k+1}{k}$ and the one for abelian coefficients from $i\le\frac{n-m+2}{k}$ to $i\le\frac{n-m+3}{k}$.
\item If the canonical resolution is graded $(g_\cM-1)$-connected in degrees $\ge1$, then the isomorphism range in \cref{theorem:constant} for constant coefficients can be improved from $i\le\frac{n-1}{2}$ to $i\le\frac{n}{2}$, similar to the proof of \cite[Thm 5.1]{RWConf} for symmetric groups.
\end{enumerate}
\end{rem}

\section{Stability with twisted coefficients}
This section serves to introduce a notion of twisted coefficient systems and to prove \cref{theorem:twisted}. Many ideas in this section are inspired by \cite[Sect.\,4]{RWW}, which is itself a generalisation of work by Dwyer, van der Kallen, and Ivanov \cite{Dwyer, Ivanov,vanderKallen}. We use similar notation to \cite{RWW} to emphasise analogies and refer to Remarks~\ref{remark:RWWcoefficients} and~\ref{remark:quillencoefficients} for a comparison of their notion of coefficient systems to ours. 

\subsection{Coefficient systems of finite degree}\label{section:coefficientsystems}
We define coefficient systems of finite degree for graded modules over graded braided monoidal categories, such as fundamental groupoids of graded $\tE_1$-modules over $\tE_2$-algebras, as described in \cref{section:modules}.

Let $(\cM,\oplus)$ be a graded right-module over a braided monoidal category $(\cA,\oplus,\braiding,0)$ in the sense of \cref{section:gradings}. We fix a \emph{stabilising object} $X$, i.e.~an object of $\cA$ of degree $1$, and recall the free braided monoidal category $\cB=\coprod_{n\ge0}\braidgroup_n$ on one object, built from the braid groups $\braidgroup_n$. The choice of $X$ induces a functor $\maptwoshort{\cB}{\cA}$, so in particular homomorphisms $\maptwoshort{B_n}{\aut_\cA(X^{\oplus n})}$ and a module-structure on $\cM$ over $\cB$.

\begin{dfn}\label{definition:coefficientsystems}A \emph{coefficient system} $F$ for $\cM$ is a functor \[\mapwo{F}{\cM}{\Ab}\] to the category of abelian groups, together with a natural transformation \[\mapwo{\sigma^F}{F}{F(-\oplus X)},\] called the \emph{structure map of $F$}, such the image of the canonical morphism $\maptwoshort{\braidgroup_m}{\aut_{\cA}(X^{\oplus m})}$ acts trivially on the image of $\mapwoshort{(\sigma^F)^m}{F}{F(-\oplus X^{\oplus m})}$ for all $m\ge0$. A \emph{morphism between coefficient systems} $F$ and $G$ for $\cM$ is a natural transformation $\maptwoshort{F}{G}$ that commutes with the structure maps $\sigma^F$ and $\sigma^G$.
\end{dfn}

\begin{rem}The category of coefficient systems for $\cM$ is abelian, so in particular has (co)kernels. More concretely, it is a category of abelian group-valued functors on a category $\langle\cM,\cB\rangle$ (see \cref{remark:quillencoefficients}).
\end{rem}

\begin{dfn}\label{definition:suspension}Define the \emph{suspension} $\Sigma F$ of a coefficient system $F$ for $\cM$ as \[\Sigma F=F(-\oplus X),\] together with the structure map $\mapwoshort{\sigma^{\Sigma F}}{\Sigma F}{\Sigma F(-\oplus X)}$, defined as the composition \[\Sigma F=F(-\oplus X)\xrightarrow{\sigma^F(-\oplus X)}F(-\oplus X^{\oplus 2})\xrightarrow{F(-\oplus\braiding_{X,X})}F(-\oplus X^{\oplus 2})=\Sigma F(-\oplus X).\] The structure map $\sigma^F$ of $F$ induces a morphism $\maptwoshort{F}{\Sigma F}$ of coefficient systems, called the \emph{suspension map}, whose (co)kernel is the \emph{kernel} $\kernel(F)$ respectively \emph{cokernel} $\cokernel(F)$ of $F$. We call $F$ \emph{split} if the suspension map is split injective in the category of coefficient systems. \end{dfn}

\begin{lem}The suspension $\Sigma F$ and the suspension map $\maptwoshort{F}{\Sigma F}$ are well-defined.
\end{lem}
\begin{proof}The triviality condition for $\Sigma F$ is implied by the one for $F$, since $(\sigma^{\Sigma F})^m$ agrees with $F(-\oplus\braiding_{X^{\oplus m},X})\sigma^F(-\oplus X)^m$, which follows by induction on $m$, using the braid relation $(X\oplus\braiding_{X,X^{\oplus m-1}})(\braiding_{X,X}\oplus X^{\oplus m-1})=\braiding_{X^{\oplus m},X}$. The fact that the suspension map is a morphism of coefficient system is a consequence of the triviality condition on $F$, more specifically of $F(-\oplus\braiding_{X,X}) (\sigma^F)^2=(\sigma^F)^2$.
\end{proof}

\begin{rem}The suspension map gives rise to a natural transformation $\maptwoshort{\id}{\Sigma}$ of endofunctors on the category of coefficient systems for $\cM$.
\end{rem}

For the remainder of the section, we fix a coefficient system $F$ for the module $\cM$.

\begin{dfn}\label{definition:degree}We denote by $F_n$ for $n\ge0$ the restriction of $F$ to the full subcategory $\cM_n\subseteq \cM$ of objects of degree $n$ and define the \emph{degree} and \emph{split degree} of $F$ at an integer $N$ inductively by saying that $F$ has
\begin{enumerate}
\item (split) degree $\le-1$ at $N$ if $F_n=0$ for $n\ge N$,
\item degree $r$ at $N$ for a $r\ge0$ if $\kernel(F)$ has degree $-1$ at $N$ and $\cokernel(F)$ has degree $(r-1)$ at $(N-1)$, and
\item split degree $r$ at $N$ for a $r\ge0$ if $F$ is split and $\cokernel(F)$ is of split degree $(r-1)$ at $(N-1)$.
\end{enumerate}
\end{dfn}

\begin{rem}Note that for all $N\le0$, $F$ is of (split) degree $r$ at $0$ if and only if it is of (split) degree $r$ at $N$, and that the property of being of (split) degree $r$ at $0$ is independent of the chosen grading. However, being of degree $r$ at $N$ depends on the grading if $N$ is positive. If $g_\cM$ is a grading for $\cM$, then so is $g_\cM+k$ for any $k\ge0$ and by induction on $r$, one proves that for $k\ge0$, $F$ is of (split) degree $r$ at $N$ with respect to a grading $g_\cM$ if and only if it is of (split) degree $r$ at $(N-k)$ with respect to the grading $g_\cM+k$
\end{rem}

\begin{lem}\label{lemma:iteratedsuspension}The iterated suspension $\Sigma^iF$ for $i\ge0$ is given by $\Sigma^iF=F(-\oplus X^{\oplus i})$ with structure map  \[\Sigma^iF=F(-\oplus X^{\oplus i})\xrightarrow{\sigma^F(-\oplus X^{\oplus i})}F(-\oplus X^{\oplus i}\oplus X)\xrightarrow{F(-\oplus \braiding_{X^{\oplus i},X})}F(-\oplus X\oplus X^{\oplus i})=\Sigma^iF(-\oplus X).\]
\end{lem}
\begin{proof}
This follows by induction on $i$, using the braid relation $(\braiding_{X,X}\oplus X^{\oplus i}) (X\oplus\braiding_{X^{\oplus i},X})=\braiding_{X^{\oplus i}\oplus X,X}$.
\end{proof}

\begin{lem} \label{lemma:degreelemma}
Let $F$ be a coefficient system for $\cM$.
\begin{enumerate}
\item For all $i\ge 0$, $\Sigma^i(\kernel(F))$ and $\Sigma^i(\cokernel(F))$ are isomorphic to $\kernel(\Sigma^iF)$ and $\cokernel(\Sigma^iF)$, respectively.
\item If $F$ is split, then $\Sigma^iF$ is split for all $i\ge0$.
\item If $F$ is of (split) degree $r$ at $N$, then $\Sigma^iF$ is of (split) degree $r$ at $(N-i)$ for $i\ge0$.
\end{enumerate}
\end{lem}
\begin{proof}
Using \cref{lemma:iteratedsuspension} and $(X\oplus\braiding^{-1}_{X^{\oplus i},X})(\braiding_{X^{\oplus i+1},X})=(\braiding_{X^{\oplus i+1},X})(\braiding^{-1}_{X^{\oplus i},X}\oplus X)$,
the natural transformation \[\textstyle{\Sigma^{i+1}F(-)=F(-\oplus X^{\oplus i+1})\xrightarrow{F(-\oplus \braiding^{-1}_{X^{\oplus i},X})} F(-\oplus X^{\oplus i+1})=\Sigma^{i+1}F(-)}\] can be seen to commute with the structure map of $\Sigma^{i+1}F$, so defines an automorphism $\Phi\colon \Sigma^{i+1}F\rightarrow\Sigma^{i+1}F$. \cref{lemma:iteratedsuspension} also implies the relation $\Sigma^i(\sigma^F)=\Phi\sigma^{\Sigma^iF}$ and therefore $\Sigma^i(\operatorname{(co)\kernel}(\sigma^F))=\operatorname{(co)\kernel}(\Phi\sigma^{\Sigma^i F})$. Hence, the coefficient systems in comparison are (co)kernels of morphisms that differ by an automorphism. This proves the first claim. Given a splitting $\mapwoshort{s}{\Sigma F}{F}$ for $F$, the composition $\Sigma^{i}(s)\Phi$ splits $\Sigma^{i} F$, which shows the second. Finally, the third follows from the first two by induction on $r$.
\end{proof}

\begin{rem}\label{remark:coefficientsystemsasmodules}If $\cM$ is a groupoid such that all subcategories $\cM_n$ are connected, then a coefficient system for $\cM$ is equivalently given as a sequence of $\aut(A\oplus X^{\oplus n-g(A)})$-modules $F_n$ for an element $A$ of minimal degree $g(A)$, together with $(-\oplus X)$-equivariant morphisms $\maptwoshort{F_{n}}{F_{n+1}}$ such that the image of $B_m$ in $\aut(X^{\oplus m})$ acts via $(A\oplus X^{\oplus n-g(A)}\oplus-)$ trivially on the image of $F_n$ in $F_{n+m}$ for all $n$ and $m$.
\end{rem}

\begin{rem}\label{remark:RWWcoefficients}A \emph{pre-braided} monoidal category in the sense of \cite{RWW} is a monoidal category $(\cC,\oplus,\braiding,0)$ whose unit $0$ is initial and whose underlying groupoid $\cC^{\sim}$ is braided monoidal satisfying a certain condition (see \cite[Def.\,1.5]{RWW}). In their work, a \emph{coefficient system} for $\cC$ at a pair of objects $(A,X)$ is an abelian group valued functor $F^{\RW}$ defined on the full subcategory $\cC_{A,X}\subseteq\cC$ generated by $A\oplus X^{\oplus n}$ for $n\ge0$. Considering $\cC_{A,X}^\sim$ as a module over the braided monoidal groupoid $\cC_{0,X}^\sim$, such a functor $F^{\RW}$ gives a coefficient system $F$ in our sense by restricting $F^{\RW}$ to $\cC_{A,X}^\sim$ and defining the structure map as $\sigma^F(-)\coloneq F^{\RW}(-\oplus \iota_X)$, where $\mapwoshort{\iota_X}{0}{X}$ is the unique morphism. In \cite{RWW}, the transformation $\mapwoshort{-\oplus \iota_X}{\id_{\cC}}{-\oplus X}$ is denoted by $\sigma^X$, so we have the suggestive identity $F^{\RW}(\sigma^X)=\sigma^F$. Assigning a coefficient system for $\cC$ at $(A,X)$ in the sense of \cite{RWW} to one for $\cC_{A,X}^\sim$ in our sense yields a functor between the respective categories of coefficient systems, which can be seen to preserve the \emph{suspension} and \emph{degree} in the sense of \cite{RWW} and in ours, at least up to isomorphism. See \cref{section:comparisonRWW} for a general comparison between \cite{RWW} and our work.
\end{rem}

\begin{rem}\label{remark:quillencoefficients}The category of coefficient systems for $\cM$ is isomorphic to the category of abelian group-valued functors on a category $\langle\cM,\cB\rangle$. To construct this category, recall \emph{Quillen's bracket construction} $\langle\cE,\cF\rangle$ of a monoidal category $\cF$ that acts via $\mapwoshort{\oplus}{\cE\times\cF}{\cE}$ on a category $\cE$ \cite[p.219]{Quillen}. It has the same objects as $\cE$, and a morphism from $C$ to $C'$ is an equivalence class of pairs $(D,f)$ with $D\in\ob\cF$ and $f\in\cE(C\oplus D,C')$, where $(D,f)$ and $(D',f')$ are equivalent if there is an isomorphism $g\in\cF(D,D')$ satisfying $f'=f(C\oplus g)$. Using this construction, we obtain the category $\langle\cM,\cB\rangle$ encoding coefficient systems by letting the free braided monoidal category on one object $\cB$ act on $\cM$ via the functor $\cB\rightarrow\cA$ induced by $X$, followed by the action of $\cA$ on $\cM$. The multiplication by $X$ on $\cM$ induces an endofunctor \[\mapwo{\Sigma}{\langle\cM,\cB\rangle}{\langle\cM,\cB\rangle}\] by mapping a morphism $\mapwoshort{[D,f]}{C}{C'}$ to $\mapwoshort{[D,(f\oplus X) (C \oplus\braiding_{X,D})]}{C\oplus X}{C'\oplus X}$. This functor comes together with a natural transformation $\mapwoshort{\sigma}{\id}{\Sigma}$ given by $[X,\id]$, such that the suspension of a coefficient system $F$, seen as a functor on $\langle\cM,\cB\rangle$, is the composition $(F\circ \Sigma)$ and its suspension map is $\mapwoshort{F(\sigma)}{F}{(F\circ\Sigma)}$. From this point of view and using the notation of the previous remark, the functor from coefficient systems in the sense of \cite{RWW} to ones in ours, described in the previous remark, is given by precomposition with a functor $\maptwoshort{\langle\cC_{A,X}^\sim,\cB\rangle}{\cC_{A,X}}$ that is the identity on objects and maps a morphism $[X^{\oplus k},f]$ in $\langle\cC_{A,X}^\sim,\cB\rangle$ from $C$ to $C'$ to $f(C\oplus\iota_{X^{\oplus k}})$.
\end{rem}

\subsection{Twisted stability of $\tE_1$-modules over $\tE_2$-algebras}\label{section:twistedstability}
We fix a graded $\tE_1$-module $\cM$ over an $\tE_2$-algebra $\cA$ with stabilising object $X$ for the rest of the section. Recall from \cref{section:modules} that its fundamental groupoid $(\Pi(\cM),\oplus)$ is a graded right-module over the graded braided monoidal category $(\Pi(\cA),\oplus,\braiding,0)$. 

\begin{dfn}\label{definition:coefficientsystemsE1}A \emph{coefficient system} for $\cM$ is a coefficient system for $\Pi(\cM)$ in the sense of \cref{definition:coefficientsystems}.
\end{dfn}

The structure map of a coefficient system $F$ for $\cM$ enhance the stabilisation map $\mapwoshort{s}{\cM}{\cM}$ to a map \[\mapwo{(s;\sigma^F)}{(\cM;F)}{(\cM;F)}\] of spaces with local systems, which stabilises homologically by \cref{theorem:twisted} if the canonical resolution is sufficiently connected and $F$ is of finite degree. This remainder of this section is devoted to the proof.

\begin{rem}In the course of the proof of \cref{theorem:twisted}, it will be convenient to have fixed a notion of a \emph{homotopy commutative square of spaces with local systems}, by which we mean a square
\begin{center}
\begin{tikzcd}[row sep=0.5cm,column sep=1cm]
(X;F)\arrow[r]\arrow[d]&(Y;G)\arrow[d]\\
(X';F')\arrow[r]&(Y';G'),
\end{tikzcd}
\end{center}
together with a specified homotopy between the two compositions, which might be nontrivial, even if the diagram is strictly commutative. Taking singular chains results in a homotopy commutative square of chain complexes (see \cref{section:localcoefficients}), and taking vertical mapping cones of the square induces a morphism \begin{equation}\label{equ:map}\maptwoshort{\oH_*\big((X';F'),(X;F)\big)}{\oH_*\big((Y';G'),(Y;G)\big)},\end{equation}which depends on the homotopy. However, homotopies that are homotopic as homotopies give homotopic morphisms on mapping cones, hence they induce the same morphism \eqref{equ:map}. Horizontal composition of such squares, including the homotopies, induces the respective composition of \eqref{equ:map}. Even though \eqref{equ:map} depends on the homotopy, the long exact sequences of the mapping cones still fit into a commutative ladder.
\end{rem}
We denote by $\Rel_*(F)=\oH_*((\cM;F),(\cM;F))$ the relative groups  with respect to the stabilisation $(s;\sigma^F)$, equipped with the additional grading $\Rel_*(F)=\bigoplus_{n\in\bfNext}\oH_*((\cM_{n+1};F),(\cM_n;F))$. Although the square
\begin{equation}\label{diagram:square}
\begin{tikzcd}[row sep=0.6cm,column sep=2cm]
(\cM;F)\arrow[r,"(s;\sigma^F)"]\arrow[d,swap,"(s;\sigma^F)"]&(\cM;F)\arrow[d,"(s;\sigma^F)"]\\
(\cM;F)\arrow[r,"(s;\sigma^F)"]&(\cM;F),
\end{tikzcd}
\end{equation}
commutes strictly, we consider it as homotopy commutative via the homotopy \eqref{equation:selfhomotopy} of \cref{section:constructionspectralsequence}, which extends to one of spaces with local systems (see \cref{section:localcoefficients}), since the triviality condition on coefficient systems gives $F(-\oplus\braiding_{X,X}^{-1})(\sigma^F)^2=(\sigma^F)^2$. This homotopy commutative square induces a relative stabilisation \[\mapwo{(s;\sigma^F)^{\sim}_*}{ \Rel_*(F)}{\Rel_*(F)}\] of degree $1$, where the superscript $\sim$ indicates the twist by the homotopy. The homotopy commutative square \eqref{diagram:square} admits a factorisation into a composition of homotopy commutative squares
\begin{equation*}
\begin{tikzcd}[column sep=2cm,row sep=0.6cm]
(\cM;F)\arrow[r,"{(\id;\sigma^F)}"]\arrow[d,swap,"{(s;\sigma^F)}"]&(\cM;\Sigma F)\arrow[d,"{(s;\sigma^{\Sigma F})}"] \arrow[r,"{(s;\id)}"]& (\cM;F)\arrow[d,"{(s;\sigma^F)}"]\\
(\cM;F)\arrow[r,"{(\id;\sigma^F)}"]&(\cM;\Sigma F)\arrow[r,"{(s;\id)}"]&(\cM;F)),
\end{tikzcd}
\end{equation*}
in which the square on the left strictly commutes because of the triviality condition, and we equip it with the trivial homotopy. The square on the right is homotopy commutative using the same homotopy as for \eqref{diagram:square}. This induces a factorisation of the relative stabilisation map as
\begin{equation}\label{equation:factorisation}\Rel_*(F)\xlongrightarrow{(\id;\sigma^F)_*}\Rel_*(\Sigma F)\xlongrightarrow{(s;\id)^{\sim}_*}\Rel_*(F),\end{equation} with the first map being of degree $0$ and the second of degree $1$.

\begin{lem}\label{lemma:compositionzero}The composition
$\Rel_*(F)\xrightarrow{(s;\sigma^F)^{\sim}_*}\Rel_*(F)\xrightarrow{(s;\sigma^F)^{\sim}_*}\Rel_*(F)$ is trivial.
\end{lem}
\begin{proof}
The mapping cones defining $\Rel_*(F)$ induce a commutative diagram of long exact sequences 
\begin{equation*}\begin{tikzcd}[row sep=0.6cm,column sep=0.7cm]
\ldots\arrow[r]&\oH_*(\cM;F)\arrow[d]\arrow[r]&\oH_*(\cM;F)\arrow[r]\arrow[d]&\Rel_*(F)\arrow[r,"h_3"]\arrow{d}{h_4}&\oH_{*-1}(\cM;F)\arrow[r,"h_1"]\arrow[d,"h_2"]&\ldots\\
\ldots\arrow[r]&\oH_*(\cM;F)\arrow[d]\arrow[r]&\oH_*(\cM;F)\arrow[r,"h_6"]\arrow[d,"h_8"]&\Rel_*(F)\arrow[r,"h_5"]\arrow{d}{h_7}&\oH_{*-1}(\cM;F)\arrow[r]\arrow[d]&\ldots\\
\ldots\arrow[r]&\oH_*(\cM;F)\arrow[r,"h_{10}"]&\oH_*(\cM;F)\arrow[r,"h_9"]&\Rel_*(F)\arrow[r]&\oH_{*-1}(\cM;F)\arrow[r]&\ldots
\end{tikzcd}
\end{equation*} in which $h_7h_4$ agrees with the composition in consideration. As $h_1$ and $h_2$ both equal $(s;\sigma^F)_*$, we conclude $0=h_1h_3=h_2h_3=h_5h_4$, so the image of $h_4$ is in the kernel of $h_5$, which is the image of $h_6$. Hence it is enough to show $h_7h_6=0$. Since $h_8=h_{10}$ for the same reason as $h_1=h_2$, we obtain the claim from the identity $h_7h_6=h_9h_8=h_9h_{10}=0$.
\end{proof}

\subsection{The relative spectral sequence}
We prove \cref{theorem:twisted} via a relative analogue of the spectral sequence \eqref{equation:ssconstant} of \cref{section:constructionspectralsequence}, which we derive from a map of augmented $\semisimpthick$-spaces
\begin{equation*}
\begin{tikzcd}[row sep=0.3cm,column sep=0.6cm]
R_\bullet(\cM)\arrow[r,dotted]\arrow[d]&R_\bullet(\cM)\arrow[d]\\
\cM\arrow[r,"s"]&\cM
\end{tikzcd}
\end{equation*} covering the stabilisation map $s$. Indicated by the dotted arrow, this morphism will only be defined up to up to higher coherent homotopy, and we obtain it from replacing the canonical resolution $R_\bullet(\cM)$ with an equivalent bar construction $B(\UB(\bullet,{\smallsquare}),U\cO,B_{\smallsquare}(\cM))$ which admits a strict morphism of the desired form.

To this end, recall from \cref{section:complexes} the homotopy discrete category $U\cO$, the isomorphism $\pi_0(U\cO)\cong \UB$, and the augmented $U\cO$-space $B_\bullet (\cM)$ whose restriction to the subcategory $\semisimpthick\subseteq U\cO$ is $R_\bullet(\cM)$, where $\semisimpthick$ is the union of components hit by the section $\maptwoshort{\semisimp}{\UB}$. Define the $(\semisimpopthick\times U\cO)$-space $U\cO(\bullet,{\smallsquare})$ and the $(\semisimpop\times \UB)$-space $\UB(\bullet,{\smallsquare})$ by restricting the hom-functors of $U\cO$ and $\UB$ appropriately. Taking components gives a weak equivalence $\maptwoshort{\semisimpopthick\times U\cO}{\semisimpop\times \UB}$ of enriched categories and one $\maptwoshort{U\cO(\bullet,{\smallsquare})}{\UB(\bullet,{\smallsquare})}$ of $(\semisimpopthick\times U\cO)$-spaces, which fits into a chain of weak equivalences of $\semisimpthick$-spaces 
\begin{equation}\label{equation:chainofweakequivalences}R_\bullet(\cM)\xleftarrow{\simeq} B\big(U\cO(\bullet,{\smallsquare}),U\cO,B_{\smallsquare}(\cM)\big)\xrightarrow{\simeq}B\big(\UB(\bullet,{\smallsquare}),U\cO,B_{\smallsquare}(\cM)\big),\end{equation} augmented over $\cM$, the left arrow being the restriction of the bar resolution of $B_\bullet(\cM)$ to $\semisimpthick$ (see \cref{section:homotopycolimits}). Consider the functor $\mapwoshort{t}{U\cO}{U\cO}$ that maps $[p]$ to $[p+1]$ and is on morphisms defined as \[\map{t}{U\cO\big([q],[p]\big)}{U\cO\big(t([q]),t([p])\big)}{(d,\mu)}{(d,\gamma(c;\mu,1_{\ocolour{a}})),}\] using the operadic composition $\gamma$ and the element $c$ with which we defined the iterated operations $c_p\in\cO(p)$ in \cref{section:complexes}. Accompanying this functor, there is a morphism of augmented $U\cO$-spaces
\begin{equation}\label{equation:rightmorphism}
\begin{tikzcd}[row sep=0.5cm,column sep=0.6cm]
B_{\smallsquare}(\cM)\arrow[r]\arrow[d]&B_{t({\smallsquare})}(\cM)\arrow[d]\\
\cM\arrow[r,"s"]&\cM,
\end{tikzcd}
\end{equation} defined by making use of the module structure $\theta$ of $\cM$ to assign to a $p$-simplex $(A,\zeta)$ in $B_p(\cM)$ the element $(A,\theta(c;\zeta,X))$ in $B_{p+1}(\cM)$. Last but not least, we define a morphism of $(\semisimpopthick\times U\cO)$-spaces \begin{equation}\label{equation:middlemorphism}\maptwo{\UB\big(\bullet,{\smallsquare}\big)}{\UB\big(\bullet,t({\smallsquare})\big)}\end{equation} by considering the braid groups $\coprod_{n\ge0}B_n$ as the free braided monoidal category in one object $X$ to define 
\[\mapnoname{\UB\big([q],{[p]}\big)=B_{p+1}/B_{p-q}}{B_{p+2}/B_{p-q+1}=\UB\big([q],t([p])\big)}{[b]}{[(b\oplus X)(X^{\oplus p-q}\oplus\braiding^{-1}_{X^{\oplus p+1},X})].}\]

\begin{lem}The assignment \eqref{equation:middlemorphism} defines indeed a morphism of $(\semisimpopthick\times U\cO)$-spaces.
\end{lem}
\begin{proof}
Recall that $U\cB(\bullet,{\smallsquare})$ is induced from a $(\semisimpop\times \UB)$-space via the equivalence $\maptwoshort{\semisimpopthick\times U\cO}{\semisimpop\times \UB}$. The semi-simplicial direction of $U\cB(\bullet,{\smallsquare})$ comes from the section $\maptwoshort{\semisimp}{\UB}$ of \cref{lemma:section}, which maps a face map $d_i$ in $\semisimp([q-1],[q])$ to the class $[\braiding^{-1}_{X^{\oplus i},X}\oplus X^{\oplus q-i}]$ in $\UB([p-1],[p])$, so \eqref{equation:middlemorphism} is natural in the semi-simplicial direction if the two braids
\[(X^{\oplus p-q}\oplus\braiding^{-1}_{X^{\oplus i},X}\oplus X^{\oplus q-i}\oplus X)(X^{\oplus p-q+1}\oplus\braiding^{-1}_{X^{\oplus q},X})\quad\text{and}\quad(X^{\oplus p-q}\oplus\braiding^{-1}_{X^{\oplus q+1},X})(X^{\oplus p-q+1}\oplus\braiding^{-1}_{X^{\oplus i},X}\oplus X^{\oplus q-i})\] define the same class in $\UB([q-1],[p+1])=B_{p+2}/B_{p-q+2}$. An application of braid relations shows that these two braids agree up to right-multiplication with $(X^{\oplus p-q}\oplus\braiding^{-1}_{X,X}\oplus X^{q})$, so coincide in $B_{p+2}/B_{p-q+2}$, which shows the claim since the naturality in the $\UB$-direction is immediate.
\end{proof}

The functor $\mapwoshort{t}{U\cO}{U\cO}$, together with the morphisms \eqref{equation:rightmorphism} and \eqref{equation:middlemorphism}, induces a map \begin{center}
\begin{tikzcd}[row sep=0.5cm,column sep=0.6cm]
B(\UB(\bullet,{\smallsquare}),U\cO,B_{\smallsquare}(\cM))\arrow[r,"t"]\arrow[d]&B(\UB(\bullet,{\smallsquare}),U\cO,B_{\smallsquare}(\cM))\arrow[d]\\
\cM\arrow[r,"s"]&\cM.
\end{tikzcd}
\end{center}
of augmented $\semisimpthick$-spaces. Pulling back a coefficient system $F$ for the graded module $\cM$ along the augmentations, this morphism enhances to one of graded $\semisimpthick$-spaces with local coefficients that covers the stabilization map $\mapwo{(s;\sigma^F)}{(\cM;F)}{(\cM;F)}$. Identifying $R_\bullet(\cM)$ with $B(\UB(\bullet,{\smallsquare}),U\cO,B_{\smallsquare}(\cM))$ via the zig-zag \eqref{equation:chainofweakequivalences} by abuse of notation, we get a tri-graded spectral sequence of the form
\begin{multline}
\label{equation:spectralsequencetwisted}
E_{p,q,n}^1\cong
\begin{cases} \oH_q\big((R_{p}(\cM)_{n+1};F),(R_p(\cM)_n;F)\big) &\mbox{if } p\ge 0 \\  \oH_q\big((\cM_{n+1};F),(\cM_{n};F)\big) & \mbox{if } p=-1 \end{cases} \\
\implies \oH_{p+q+1}\big((\cM_{n+1},|R_\bullet(\cM)|_{n+1};F),(\cM_{n},|R_\bullet(\cM)|_n;F)\big), 
\end{multline} which is a sum of spectral sequences, one for each $n\in\bfNext$ (see Sections~\ref{section:semisimplicial} and~\ref{section:semisimplicialthick}). Using \cref{lemma:iteratedsuspension} and
\[(\braiding_{X,X}^{-1}\oplus X^{\oplus p})(X\oplus\braiding_{X^{\oplus i},X}\oplus X^{\oplus p-i})(\braiding_{X^{\oplus p+1},X})=(X\oplus \braiding_{X^{\oplus p},X})(\braiding_{X^{\oplus i},X}\oplus X^{\oplus p-i+1}),\] one checks that selfhomotopy \eqref{equation:selfhomotopy} of $s^2$ witnesses homotopy commutativity of the square
\begin{equation*}
\begin{tikzcd}[column sep=3cm,row sep=0.6cm]
{(\cM;\Sigma^{p+1}F)}\arrow[d,swap,"{(s;\sigma^{\Sigma^{p+1}F})}"]
\arrow[r,"{(s;F(-\oplus \braiding_{X^{\oplus i},X}\oplus X^{\oplus p-i}))}"]&{(\cM;\Sigma^{p}F)}\arrow[d,"{(s;\sigma^{\Sigma^{p}F})}"]\\
{(\cM;\Sigma^{p+1}F)}\arrow[r,"{(s;F(-\oplus \braiding_{X^{\oplus i},X}\oplus X^{\oplus p-i}))}"]&{(\cM;\Sigma^{p}F)},
\end{tikzcd}
\end{equation*} which thus induces a morphism $\mapwoshort{(s;\eta_i)_*^\sim}{\Rel_*(\Sigma^{p+1}F)}{\Rel_*(\Sigma^{p}F)}$ of degree $1$; the superscript $\sim$ indicates the twist by the homotopy \eqref{equation:selfhomotopy}. This morphism serves us to identify the spectral sequence \eqref{equation:spectralsequencetwisted} as follows.

\begin{lem}\label{lemma:spectralsequenceidentificationtwisted}We have $E^1_{p,q,n+1}\cong \Rel_q({\Sigma^{p+1}F})_{n-p}$, and the $d^1$-differential identifies with
\[\textstyle{\mapwoshort{\sum_{i=0}^p(-1)^i(s;\eta_i)_*^\sim}{\Rel_q({\Sigma^{p+1}F})_{n-p}}{\Rel_q({\Sigma^{p}F})_{n-p+1}}}.\] In particular, the differential $\mapwoshort{d^1}{E^1_{0,*,n+1}}{E^1_{-1,*,n+1}}$ corresponds to the second map of \eqref{equation:factorisation} in degree $n$.
\end{lem}
\begin{proof}On $p$-simplices, the first equivalence of \eqref{equation:chainofweakequivalences} has a preferred homotopy inverse induced by the extra degeneracy given by inserting the identity of $U\cO([p],[p])$ (see \cref{section:homotopycolimits}). Composing it with the second equivalence of \eqref{equation:chainofweakequivalences}  yields an equivalence that forms the vertical arrows of a square \begin{center}
\begin{tikzcd}[row sep=0.6cm]
R_p(\cM)\arrow[d,swap,"\simeq"]\arrow[r,"\tilde{t}"]&R_p(\cM)\arrow[d,"\simeq"]\\
B(\UB([p],{\smallsquare}),U\cO,B_{\smallsquare}(\cM))\arrow[r,"t"]&B(\UB([p],{\smallsquare}),U\cO,B_{\smallsquare}(\cM)),
\end{tikzcd}
\end{center}
where $\tilde{t}$ is defined by mapping $(A,\zeta)$ to $(s(A),s(\zeta)\cdot\theta(\alpha_p;A,X^{p+2}))$. Here $\alpha_p\in\Omega_{c_{p+2}}\cO(p+2)$ is any loop that corresponds to $\braiding^{-1}_{X^{\oplus p+1},X}$ under the equivalence $\pi_1(\cO(p+2);c_{p+2})\cong B_{p+2}$ (see \cref{section:complexes}). This choice of $\alpha_p$ guarantees that the previous square commutes, which is why it is sufficient to show that
\begin{center}
\begin{tikzcd}[row sep=0.6cm,column sep=1cm]
{(\cM_{n-p};\Sigma^{p+1}F)}\arrow[d,"{(\iota;\id)}",swap]
\arrow[r,"{(s;\sigma^{\Sigma^{p+1}F})}"]&{(\cM_{n-p+1};\Sigma^{p+1}F)}\arrow[d,"{(\iota;\id)}"]\\
{(R_p(\cM)_{n+1};F)}\arrow[r,"{(\tilde{t};\sigma^F)}"]&{(R_p(\cM)_{n+2};F)}
\end{tikzcd}
\end{center} 
homotopy commutes in order to prove $E^1_{p,q,n+1}\cong \Rel_q({\Sigma^{p+1}F})_{n-p}$, where $\iota$ denotes the canonical equivalence mapping $A$ to $(A,\const_{s^{p+1}(A)})$. On the space-level, the two compositions are given by assigning to $A\in\cM_{n-p}$ the elements $(s(A),\const_{s^{p+2}(A)})$ and $(s(A),\theta(\alpha_p;A,X^{p+2}))$, respectively. As we have $\sigma^{\Sigma^{p+1}F}(-)=F(-\oplus X^{\oplus p+1})\sigma^F(-\oplus X^{\oplus p+1})$ by \cref{lemma:iteratedsuspension}, the homotopy induced by following $\alpha_p$ to its endpoint is one of maps with local cofficients, which implies the first claim of the lemma. A relative version of the proof of \cref{lemma:differentialszero} establishes the identification of the differentials and finishes the proof.
\end{proof}

\begin{lem}\label{lemma:differentialtwistedzero}
The following composition is zero for $n\ge1$,
\[\Rel_*(F)_{n-1}\xlongrightarrow{(\id;\sigma^F)}\Rel_*({\Sigma F})_{n-1}\xlongrightarrow{(\id;\sigma^{\Sigma F})}\Rel_*(\Sigma^2 F)_{n-1}\cong E^1_{1,*,n+1}\xlongrightarrow{d^1}E^1_{0,*,n+1}\cong\Rel_*(\Sigma F)_n.\]
\end{lem}
\begin{proof}
Using \cref{lemma:spectralsequenceidentificationtwisted}, the composition in question is the difference between the morphisms in degree $(n-1)$ induced by the compositions of the two homotopy commutative squares
\begin{equation*}
\begin{tikzcd}[column sep=3cm, row sep=0.6cm]
(\cM;F)\arrow[d,swap,"{(s;\sigma^F)}"]\arrow[r,"(\id;\sigma^{\Sigma F}\sigma^F)"]&(\cM;\Sigma^2 F)\arrow[d,"{(s;\sigma^{\Sigma^2 F})}"]\arrow[r,"(s;F(-\oplus\braiding_{X^{\oplus i},X}\oplus X^{\oplus 1-i}))"]&(\cM;\Sigma F)\arrow[d,"{(s;\sigma^{\Sigma F})}"]\\
(\cM_;F)\arrow[r,"{(\id;\sigma^{\Sigma F}\sigma^F)}"]&(\cM;\Sigma^2 F)\arrow[r,"(s;F(-\oplus\braiding_{X^{\oplus i},X}\oplus X^{\oplus 1-i}))"]&(\cM;\Sigma F),
\end{tikzcd}
\end{equation*} for $i=0$ and $i=1$, where the homotopy of the left square is trivial. The nontrivial homotopy of the right square becomes trivial after composing with the left square by the triviality condition for coefficient systems, so the composition in question is the difference of the morphisms induced by the two strictly commutative outer squares. But, again by the triviality condition, we have $F(-\oplus\braiding_{X,X})\sigma^{\Sigma F}\sigma^F=\sigma^{\Sigma F}\sigma^F$, so the two outer squares coincide and hence the difference of the induced morphisms vanishes.
\end{proof}

\subsection{The proof of \cref{theorem:twisted}}
The long exact sequence
\[\ldots \longrightarrow\Rel_{*+1}(F)\longrightarrow\oH_*(\cM;F)\xlongrightarrow{(s;\sigma^F)_*}\oH_*(\cM;F)\longrightarrow \Rel_*(F)\longrightarrow\oH_{*-1}(\cM;F)\longrightarrow\ldots\] exhibits
 \cref{theorem:twisted} as a consequence of the next result.
 
 \begin{thm}\label{theorem:twistedprecise}Let $F$ be a coefficient system for $\cM$ of degree $r$ at $N\ge0$. If the canonical resolution of $\cM$ is graded $(\frac{g_\cM-2+k}{k})$-connected in degrees $\ge1$ for a $k\ge2$, then 
 \begin{enumerate}
 \item the group $\Rel_i(F)_n$ vanishes for $n\ge\max(N+1,k(i+r))$, and
\item if $F$ is of degree split $r$ at $N\ge0$, then $\Rel_i(F)_n$ vanishes for $n\ge\max(N+1,ki+r)$.
\end{enumerate}
 \end{thm}
We prove \cref{theorem:twistedprecise} via a double induction on $r$ and $i\ge 0$ by considering the following statement.
\begin{itemize}
\item[$(\bfH_{r,i})$]The vanishing ranges of \cref{theorem:twistedprecise} hold for all $F$ of degree $<r$ at any $N\ge0$ in all homological degrees $i$, and for all $F$ of degree $r$ at any $N\ge0$ in homological degrees $<i$.
\end{itemize}
The claim $(\bfH_{r,i})$ holds trivially if $r<0$ or if $(r,i)=(0,0)$. If $(\bfH_{r,i})$ holds for a fixed $r$ and all $i$, then $(\bfH_{r+1,0})$ follows, since there is no requirement on coefficient systems of degree $(r+1)$. Hence, to prove the theorem, it is sufficient to show that $(\bfH_{r,i})$ implies $(\bfH_{r,i+1})$ for $i,r\ge0$. As the composition \[\Rel_i(F)_n\xlongrightarrow{(s;\sigma^F)^{\sim}_*}\Rel_i(F)_{n+1}\xlongrightarrow{(s;\sigma^F)^{\sim}_*}\Rel_i(F)_{n+2}\] is zero by \cref{lemma:compositionzero}, it is enough to show injectivity of both maps in the claimed range. Using the factorisation \eqref{equation:factorisation}, this is implied by the following lemma.
\begin{lem}\label{lemma:twistedmainlemma}Let $r\ge0$ and $i\ge0$ satisfying $(\bfH_{r,i})$, and let $F$ be of degree $r$ at some $N\ge0$.
\begin{enumerate}
\item The morphism $\mapwoshort{(\id,\sigma_X)_*}{\Rel_*(F)_n}{\Rel_*(\Sigma F)_n}$ is injective for $n\ge\max(N,k(i+r))$ and surjective for $n\ge\max(N,k(i+r-1))$. If $F$ is of split degree $r$ at $N\ge0$, then the map is split injective for all $n$ and surjective for $n\ge\max(N,ki+r-1)$.
\item The morphism $\mapwoshort{(s,\id)^{\sim}_*}{\Rel_i(\Sigma F)_n}{\Rel_i(F)_{n+1}}$ is injective in degrees $n\ge\max(N+1,k(i+r))$. If $F$ is of split degree $r$ at $N\ge0$, then the map is injective for $n\ge\max(N+1,ki+r)$.
\end{enumerate}
\end{lem}

\begin{proof}We begin by proving the first part of the statement. As $\Rel_*{(-)}$ is functorial in the coefficient system, injectivity of the split case is clear. The remaining claims of the first statement follow from the long exact sequences in $\Rel_*{(-)}$ induced by the short exact sequences
\[0\rightarrow\kernel(F)\rightarrow F\rightarrow \im(F\rightarrow \Sigma F)\rightarrow 0\quad\text{and}\quad 0\rightarrow\im(F\rightarrow \Sigma F)\rightarrow \Sigma F\rightarrow \cokernel(F)\rightarrow 0\]
by applying $(\bfH_{r,i})$, using that $\kernel(F)$ has degree $-1$ at $N$ and that $\cokernel(F)$ has (split) degree $(r-1)$ at $(N-1)$. To prove (ii), we use the spectral sequence \eqref{equation:spectralsequencetwisted} and \cref{lemma:spectralsequenceidentificationtwisted}.
Since $\maptwoshort{|R_{\bullet}(\cM)|_m}{\cM_m}$ is assumed to be $(\frac{m-2+k}{k})$-connected for $m\ge 1$, the groups $\oH_*(\cM_m,|R_\bullet(\cM)|_m;F)$ vanish for $*\le\frac{m-2}{k}$, from which we conclude $E^\infty_{p,q,n+1}=0$ for $p+q\le \frac{n}{k}$. We claim that the differential $\maptwoshort{E^1_{1,i,n+1}}{E^1_{0,i,n+1}}$ vanishes for $n\ge\max(N+1,k(i+r))$ in the nonsplit case, and for $n\ge\max(N+1,ki+r)$ in the split one. By \cref{lemma:differentialtwistedzero}, this is the case if the maps $\Rel_*(F)_{n-1}\rightarrow\Rel_*(\Sigma F)_{n-1}\rightarrow\Rel_*(\Sigma^2 F)_{n-1}$ are surjective in that range, which holds by $(i)$. Since the map we want to prove injectivity of identifies with the differential $\maptwoshort{E^1_{0,i,n+1}}{E^1_{-1,i,n+1}}$ by \cref{lemma:spectralsequenceidentificationtwisted}, it is therefore enough to show that, in the ranges of the statement, $E^\infty_{0,i,n+1}=0$ and $E^2_{p,q,n+1}=0$ holds for $(p,q)$ with $p+q=i+1$ and $q<i$. By the vanishing range of $E^\infty$ noted above, we have $E^\infty_{0,i,n+1}=0$ in the required range. The claimed vanishing of $E^2$ follows from the vanishing even on the $E^1$-page, which is proved by observing that, by $(\bfH_{r,i})$ and \cref{lemma:spectralsequenceidentificationtwisted}, the groups $E^1_{p,q,n+1}\cong\Rel_q(\Sigma^{p+1}F)_{n-p}$ vanish for $(p,q)$ with $q<i$ and $n\ge\max(N-p,k(q+r))$ in the nonsplit, and for $(p,q)$ satisfying $q<i$ and $n\ge\max(N-p,kq+r)$ in the split case, since $\Sigma^{p+1}F$ has (split) degree $r$ at $(N-p-1)$ by \cref{lemma:degreelemma}. 
\end{proof}

\section{Configuration spaces}\label{section:configurationspaces}
The \emph{ordered configuration space} of a manifold $W$ with labels in a Serre fibration $\mapwoshort{\pi}{E}{W}$ is given by \[\FConf^\pi_n(W)=\{(e_1,\ldots,e_n)\in E^n\mid \pi(e_i)\neq \pi(e_j)\text{ for }i\neq j\text{ and }\pi(e_i)\in W\backslash\partial W\},\] and the \emph{unordered configuration space} is the quotient by the canonical action of the symmetric group,
\[\Conf^\pi_n(W)=\FConf^\pi_n(W)/\Sigma_n.\]
To establish an $\tE_1$-module structure on the unordered configuration spaces of $W$, we assume that $W$ has nonempty boundary, fix a collar $\maptwoshort{(-\infty,0]\times\partial W}{W}$, and attach an infinite cylinder to the boundary, \[\widetilde{W}=W\cup_{\{0\}\times \partial W}[0,\infty)\times\partial W.\] Collar and cylinder assemble to an embedding $\bfR\times\partial W\subseteq \widetilde{W}$ of which we make frequent use henceforth. We extend the fibration $\pi$ over $\widetilde{W}$ by pulling it back along the retraction $\maptwoshort{\widetilde{W}}{W}$ and define the space \[\pConf^\pi_n(W)=\{(s,e)\in[0,\infty)\times \Conf^\pi_n(\widetilde{W})\mid \pi(e)\subseteq W\cup(-\infty,s)\times\partial W\},\] which is an equivalent model for $\Conf^\pi_n(W)$, since the inclusion in $\pConf^\pi_n(W)$ as the subspace with $s=0$ can be seen to be an equivalence by choosing an isotopy of $\widetilde{W}$ that pushes $[0,\infty)\times\partial W$ into $(-\infty,0)\times\partial W$. We furthermore fix an embedded cube ${(-1,1)^{d-1}\subseteq \partial W}$ of codimension $0$, together with a section $\mapwoshort{l}{(-1,1)^{d-1}}{E}$ of $\pi$, which we extend canonically to a section $l$ on $[0,\infty)\times(-1,1)^{d-1}\subseteq\bfR\times\partial W\subseteq \widetilde{W}$. 

\begin{lem}\label{lemma:moduleconfigurations}Configurations $\coprod_{n\ge0}\Conf_n(D^d)$ in a disc form a graded $\tE_d$-algebra with configurations $\coprod_{n\ge 0}\pConf^\pi_n(W)$ in a $d$-manifold $W$ with nonempty boundary as an $\tE_1$-module over it, graded by the number of points.\end{lem}

\begin{proof}
The operad $\cD^\bullet(D^d)$ of little $d$-discs acts on $\coprod_{n\ge0}\Conf_n(D^d)$ by
\[\map{\theta}{\cD^k(D^d)\times\big(\coprod_{n\ge0}\Conf_n(D^d)\big)^k}{\coprod_{n\ge0}\Conf_n(D^d)}{\big((\phi_1,\ldots,\phi_k),(\{d^1_i\},\ldots,\{d^k_i\})\big)}{\bigcup_{j=1}^k\phi_j(\{d_i^j\}),}\] and this action extends to one of $\SC_d$ (see \cref{definition:SC}) on the pair $(\coprod_{n\ge 0}\pConf^\pi_n(W),\coprod_{n\ge0}\Conf_n(D^d))$ via 
\[\textstyle{\map{\theta}{\SC_d(\ocolour{m},\ocolour{a}^k;\ocolour{m})\times\coprod_{n\ge 0}\pConf^\pi_n(W)\times\big(\coprod_{n\ge0}\Conf_n(D^d)\big)^k}{\coprod_{n\ge 0}\pConf^\pi_n(W)}{\big((s,\phi_1,\ldots,\phi_k),(s',\{e_i\}),\{d^1_i\},\ldots,\{d^k_i\}\big)}{\big(s'+s,\{e_i\}\cup(\cup_{j=1}^kl(\phi_j(\{d^j_i\})+s'))\big),}}\] using the section $l$ and the translation $(-+s')$ by $s'$ in the $[0,\infty)$-coordinate, as illustrated in Figure~\ref{figure:moduleconfigurations}.
\end{proof}
\begin{figure}[h]
\centering{
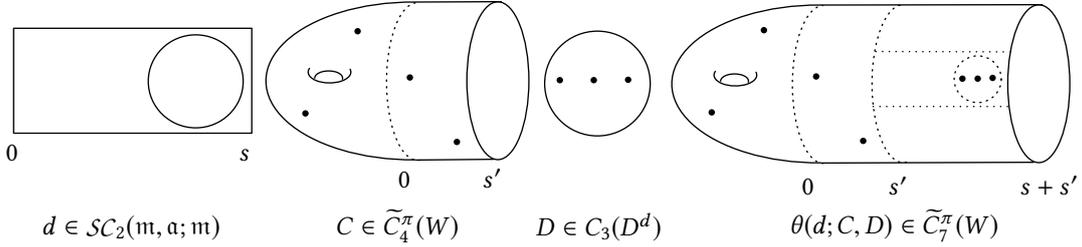}
\caption{The $\tE_1$-module structure on unordered configuration spaces}
\label{figure:moduleconfigurations}
\end{figure}

\subsection{The resolution by arcs}\label{section:arcresolution}
Let $W$ be a smooth connected manifold of dimension $d\ge 2$ with nonempty boundary and $\mapwoshort{\pi}{E}{W}$ a Serre fibration with path-connected fibre. By \cref{lemma:moduleconfigurations}, the configuration spaces $\cM=\coprod_{n\ge 0}\pConf^\pi_n(W)$ form a graded $E_1$-module  over $\cA=\coprod_{n\ge0}\Conf_n(D^d)$ considered as an $E_2$-algebra via the canonical morphism $\maptwoshort{\SC_2}{\SC_d}$ (see \cref{section:modules}). The stabilisation map $\mapwoshort{s}{\cM}{\cM}$ with respect to a chosen stabilising object $X\in \Conf_1(D^d)$, restricted to the subspace of elements of degree $n$, has the form \[\mapwo{s}{\pConf^\pi_n(W)}{\pConf^\pi_{n+1}(W)}.\]

\begin{rem}With respect to the described equivalence $\Conf^\pi_n(W)\simeq\pConf^\pi_n(W)$, the stabilisation map corresponds to the map $\maptwoshort{\Conf^\pi_n(W)}{\Conf^\pi_{n+1}(W)}$ that adds a point ``near infinity" \cite{McDuffConf,SegalConfII}.
\end{rem}

We prove high-connectivity of the canonical resolution of $\cM$ (see \cref{section:complexes}) by identifying it with the \emph{resolution by arcs}---an augmented semi-simplicial space of geometric nature, known to be highly connected. 

\begin{dfn}
The \emph{resolution by arcs} is the augmented semi-simplicial space $\maptwoshort{R^{\arcsymbol}_\bullet(\cM)}{\cM}$ with \[\textstyle{R^{\arcsymbol}_p(\cM)\subseteq\cM\times\big(\Emb([-1,0],\widetilde{W})\times \Maps([-1,0],E)\big)^{p+1}},\] consisting of tuples $((s,\{e_i\}),(\varphi_0,\eta_0),\ldots,(\varphi_p,\eta_p))$ such that
\begin{enumerate}
\item the arcs $\varphi_i$ are pairwise disjoint and connect points in the configuration $\varphi_i(-1)\in\pi(\{e_i\})\subseteq \widetilde{W}$ to points $\varphi(0)\in\{s\}\times(-1,1)\times\{0\}^{d-2}\subseteq[0,\infty)\times\partial W$ in the order $\varphi_0(0)<\ldots<\varphi_{p}(0)$,
\item the interior of the arcs lie in $W\cup[0,s)\times\partial W$ and are disjoint from the configuration $\pi(\{e_i\})$,
\item the path of labels $\eta_i$ satisfies $(\pi\circ\eta_i)=\varphi_i$ and connects the label of $\varphi_i(-1)\in\pi(\{e_i\})$ to $\eta_i(0)=l(\varphi_i(0))$,
\item there exists $\varepsilon\in(0,1)$ with $\varphi_i(t)=(s+t,\varphi_i(0),0,\ldots,0)\in(-\infty,s]\times(-1,1)^{d-1}\subseteq \widetilde{W}$ for $t\in(-\varepsilon,0]$. 
\end{enumerate}
The space $R^{\arcsymbol}_p(\cM)$ is topologised using the compact-open topology on $\Maps([-1,0],E)$ and the $\cC^\infty$-topology on $\Emb([-1,0],\widetilde{W})$.
The $i$th face map forgets $(\varphi_i,\eta_i)$. The rightmost graphic of Figure~\ref{figure:actionconfiguration} depicts an example.
\end{dfn}
\begin{thm}\label{theorem:arcresolutionconnected}The resolution by arcs $\maptwoshort{R^{\arcsymbol}_\bullet(\cM)}{\cM}$ is graded $(g_\cM-1)$-connected.
\end{thm}
\begin{proof}
Setting $s=0$ in the definition of $R^{\arcsymbol}_p(\cM)$ yields a sub-semi-simplicial space $\bar{R}^{\arcsymbol}_\bullet(\cM)\subseteq R^{\arcsymbol}_\bullet(\cM)$, augmented over $\bar{\cM}=\coprod_{n\ge0}\Conf^\pi_n(W)$. As the inclusion is a weak equivalence by the same argument as for $\Conf^\pi_n(W)\subseteq \pConf^\pi_n(W)$, the augmented semi-simplicial space $\maptwoshort{R^{\arcsymbol}_\bullet(\cM)}{\cM}$ is as connected as $\maptwoshort{\bar{R}^{\arcsymbol}_\bullet(\cM)}{\bar{\cM}}$ is. The latter is the standard resolution by arcs for configurations of unordered points with labels in $W$, which is known to have the claimed connectivity (see e.g.~the proof of \cite[Thm\,A.1]{KM}).\end{proof}

\begin{thm}\label{theorem:highconnectivityconfigurationspaces} The canonical resolution and the one by arcs are weakly equivalent as augmented $\semisimpthick$-spaces.
\end{thm}

Assuming \cref{theorem:highconnectivityconfigurationspaces}, \cref{theorem:arcresolutionconnected} ensures graded $(g_\cM-1)$-connectivity of the canonical resolution $\maptwoshort{R_\bullet(\cM)}{\cM}$ (see \cref{section:complexes}), which in turn implies  \cref{theorem:configurationspaces} by an application of Theorems~\ref{theorem:constant} and~\ref{theorem:twisted}.

We prove \cref{theorem:highconnectivityconfigurationspaces} by constructing a zig-zag of weak equivalences of augmented $\semisimpthick$-spaces
\begin{equation}\label{equation:zigzag}R_\bullet(\cM)\xlongleftarrow{\circled{$1$}} \tB\big(U\cO(\bullet,{\smallsquare}),U\cO,B_{\smallsquare}(\cM)\big)\stackrel{\circled{$2$}}{\simeq}\tB\big(U\cO^{\arcsymbol}_{\bullet,{\smallsquare}},U\cO,B_{\smallsquare}(\cM)\big)\xlongrightarrow{\circled{$3$}} R^{\arcsymbol}_\bullet(\cM)^{\fib}\end{equation} between the canonical resolution $R_\bullet(\cM)$ and the fibrant replacement $R^{\arcsymbol}_\bullet(\cM)^{\fib}$ of the resolution by arcs, which is weakly equivalent to the resolution by arcs itself (see \cref{section:homotopycolimits}). The remainder of this subsection serves to explain the weak equivalences \circled{$1$}--\circled{$3$}. We abbreviate the $\tE_{1,2}$-operad $\SC_2$ by $\cO$.

\addtocontents{toc}{\protect\setcounter{tocdepth}{0}}

\subsubsection*{$\circled{$1$}$}Recall from \cref{section:complexes} the category $U\cO$ and the contravariant $U\cO$-space $B_\bullet(\cM)$ over $\cM$ whose restriction to the subcategory $\semisimpthick\subseteq U\cO$ is $R_\bullet(\cM)$. Using the $(\semisimpopthick\times U\cO)$-space $U\cO(\bullet,{\smallsquare})$ obtained by restricting the hom-functor of $U\cO$, the equivalence $\circled{$1$}$ is defined as the restriction of the bar resolution of $B_\bullet(\cM)$ to $\semisimpthick$ (see \cref{section:homotopycolimits}).

For the other parts of the zig-zag \eqref{equation:zigzag}, we define an analogue of the resolution by arcs for the free graded $\tE_1$-module $\cO^\ocolour{m}=\coprod_{n\ge0}\cO(\ocolour{m},\ocolour{a}^n;\ocolour{m})/\Sigma_n$ (see \cref{example:freealgebra}). For simplification, we choose the centre $X=\{0\}\in\Conf_1(D^d)$ as stabilising object and write $s_d$ for the parameter of elements $d=(s_d,\{\phi_i\})\in\cO^\ocolour{m}$ and $g(d)$ for their degree, i.e.~the cardinality of the set of embeddings $\{\phi_i\}$.

\begin{dfn}\label{definition:tethereddiscs}
Define the augmented semi-simplicial space $\maptwoshort{R^{\arcsymbol}_\bullet(\cO^\ocolour{m})}{\cO^\ocolour{m}}$ with $p$-simplices \[R_p^{\arcsymbol}(\cO^\ocolour{m})\subseteq{\cO^\ocolour{m}}\times\Emb\big([-1,0],(0,\infty)\times(-1,1)\big)^{p+1},\] consisting of tuples $((s,\{\phi_j\}),\varphi_0,\ldots,\varphi_{p})$ such that
\begin{enumerate}
\item the arcs $\varphi_i$ are pairwise disjoint and connect centre points  $\varphi_i(-1)\in\{\phi_j(0)\}$ of the discs to $\varphi_i(0)\in\{s\}\times(-1,1)$ in the order $\varphi_0(0)<\ldots<\varphi_{p}(0)$,
\item the interior of the arcs lie in $(0,s)\times(-1,1)$ and are disjoint from the centre points $\{\phi_j(0)\}$, and
\item there exists an $\varepsilon\in(0,s)$ such that $\varphi_i(t)=(s+t,\varphi_i(0))\in(0,s]\times(-1,1)$ holds for all $t\in(-\varepsilon,0]$. 
\end{enumerate}
The third graphic of Figure~\ref{figure:actionconfiguration} exemplifies a $0$-simplex in $R^{\arcsymbol}_\bullet(\cO^\ocolour{m})$.\end{dfn}

\subsubsection*{$\circled{$2$}$}To explain the second equivalence of \eqref{equation:zigzag}, we note that $\cO^{\ocolour{m}}$ becomes a topological monoid by multiplying elements $d$ and $e$ in $\cO^{\ocolour{m}}$ by $\gamma(e;d,1^{g(e)})$. The multiplication map is covered by a simplicial action
\begin{equation*}\map{\Psi}{\cO^\ocolour{m}\times R_\bullet^{\arcsymbol}(\cO^\ocolour{m})}{R_\bullet^{\arcsymbol}(\cO^\ocolour{m})}{\big(d,(e,\varphi_0,\ldots,\varphi_{p})\big)}{\big(\gamma(e;d,1^{g(e)}),\varphi_0+s_d,\ldots,\varphi_p+s_d\big),}\end{equation*} where $(-+s_d)$ is the translation in the $(0,\infty)$-coordinate. This action leads to a $(\semisimpopthick\times U\cO)$-space $U\cO^{\arcsymbol}_{\bullet,{\smallsquare}}$, serving us as mediator between the canonical resolution and the one by arcs. On objects $([p],[k])$, it is \[U\cO^{\arcsymbol}_{p,{k}}=\hofib_{c_{k+1}}\big(\maptwoshort{R^{\arcsymbol}_p(\cO^{\ocolour{m}})}{\cO^{\ocolour{m}}}\big)=\{(e,\varphi_0,\ldots,\varphi_p,\mu)\in R^{\arcsymbol}_p(\cO^{\ocolour{m}})\times\Path_{c_{k+1}}\cO^{\ocolour{m}}\mid\omega(\mu)=e\},\] where $\omega(-)$ denotes the endpoint of a Moore path and the elements $c_{i}\in \cO^\ocolour{m}$ are defined as in \cref{section:complexes}. The $\semisimpopthick$-direction of $U\cO^{\arcsymbol}_{\bullet,{\smallsquare}}$ is induced by the semi-simplicial structure of $R^{\arcsymbol}_\bullet(\cO^{\ocolour{m}})$ via the functor $\maptwoshort{\semisimpthick}{\semisimp}$ (see \cref{section:complexes}). The $U\cO$-direction is defined by\begin{equation*}\mapnoname{U\cO([k],[l])\times U\cO^{\arcsymbol}_{\bullet,{k}}}{U\cO^{\arcsymbol}_{\bullet,{l}}}{\big((d,\mu),(e,\varphi_0,\ldots,\varphi_p,\zeta)\big)}{\big(\Psi(d,(e,\varphi_0,\ldots,\varphi_p)),\mu\cdot\gamma(\zeta;d,1^{k+1})\big).}\end{equation*}The claimed functoriality of $\cO^{\arcsymbol}_{\bullet,{\smallsquare}}$ follows directly from the associativity of the operadic composition $\gamma$. Having introduced the objects involved, the following lemma provides the weak equivalence $\circled{$2$}$.

\begin{lem}\label{lemma:frompathstoarcs}The $(\semisimpopthick\times U\cO)$-spaces $U\cO^{\arcsymbol}_{\bullet,{\smallsquare}}$ and $U\cO(\bullet,{\smallsquare})$ are weakly equivalent.
\end{lem}
\begin{proof}
Choose arcs $\varphi^p=(\varphi^p_0,\ldots,\varphi^p_p)\in\Emb([-1,0],(0,\infty)\times(-1,1))^{p+1}$ such that $(c_{p+1},\varphi^p)$ forms an element of $R_p^{\arcsymbol}(\cO^{\ocolour{m}})$ for which the order of the embeddings $\{\phi_i\}$ in $c_{p+1}=(s_{c_{p+1}},\{\phi_i\})$, induced by the order of the arcs $\varphi^p_i$ they are connected to, agrees with the order of $\{\phi_i\}$ induced by the $(0,\infty)$-coordinate. Acting on $(c_{p+1},\varphi^p,\const_{c_{p+1}})\in U\cO^{\arcsymbol}_{p,{p}}$, the $(\semisimpopthick\times U\cO)$-space $U\cO^{\arcsymbol}_{\bullet,{\smallsquare}}$ induces a morphism of $U\cO$-spaces \begin{equation}\label{equation:inducedmorphism}\maptwoshort{U\cO([p],{\smallsquare})}{U\cO^{\arcsymbol}_{p,{\smallsquare}}}, \end{equation} agreeing on $[k]\in\ob(U\cO)$ with the induced map on diagonal homotopy fibres of the commuting triangle
\begin{center}
\begin{tikzcd}[row sep=0.1cm]\cO^{\ocolour{m}}\arrow[rr]\arrow[dr,"{\gamma(c_{p+1};-,1^{p+1})}",swap]&&R_p^{\arcsymbol}(\cO^{\ocolour{m}})\arrow[ld]\\
&\cO^{\ocolour{m}}&
\end{tikzcd}
\end{center} at $c_{k+1}$, where the right diagonal map is the augmentation and the horizontal arrow is given by acting on $(c_{p+1},\varphi^p)$ via $\Psi$. There is a map $\maptwoshort{R_p^{\arcsymbol}(\cO^{\ocolour{m}})}{\cO^{\ocolour{m}}}$ that forgets the arcs and the discs attached to them using which the horizontal map can be seen to be an equivalence by following discs along arcs they are attached to. Hence, \eqref{equation:inducedmorphism} is an equivalence of $U\cO$-spaces, which in particular shows that $U\cO^{\arcsymbol}_{\bullet,{\smallsquare}}$ is homotopy discrete, as $U\cO(\bullet,{\smallsquare})$ is so by \cref{lemma:homotopydiscrete}. Therefore, to prove the claim, it is sufficient to show that the equivalence \eqref{equation:inducedmorphism} is natural in $[p]$ up to homotopy, which would follow from the homotopy commutativity of \begin{center}
\begin{tikzcd}[row sep=0.6cm]
U\cO([p],[k])\arrow[d,"(\tilde{d_i})_*",swap]\arrow[r]&U\cO^{\arcsymbol}_{p,{k}}\arrow[d,"(d_i)_*"]\\
U\cO([p-1],[k])\arrow[r]&U\cO^{\arcsymbol}_{p-1,{k}},
\end{tikzcd}
\end{center}using the choice of face maps $\tilde{d_i}=(c,\mu_i)\in\semisimpthick([p-1],[p])$ provided by \cref{lemma:facemaps}. The two compositions of the latter diagram map an element $(d,\zeta)$ in $U\cO([p],[k])$ to \[\Big(\Psi\big(d,(c_{p+1},\varphi^p_0,\ldots,\widehat{\varphi^p_i},\ldots,\varphi^p_p)\big),\zeta\Big)\quad\text{and}\quad \Big(\Psi\big(d,(c_{p+1},\varphi^{p-1}_0+s_c,\ldots,\varphi^{p-1}_{p-1}+s_c)\big),\zeta\cdot\gamma(\mu_i;d,1^{p+1})\Big),\] respectively, where $\widehat{(-)}$ indicates that the element is omitted. Recalling that, via the isomorphism $\pi_1(\cO^{\ocolour{m}},c_{p+1})\cong B_{p+1}$ fixed in \cref{section:complexes}, the loop $\mu_i\in\Omega_{c_{p+1}}\cO^{\ocolour{m}}$ corresponds to the braid $\braiding_{X^{\oplus i},X}^{-1}\oplus X^{\oplus p-i}$ in $B_{p+1}$, we see that our choice of the arcs $\varphi^j_i$ ensures the existence of a path in $R_{p-1}^{\arcsymbol}(\cO^{\ocolour{m}})$ between \[(c_{p+1},\varphi^p_0,\ldots,\widehat{\varphi^p_i},\ldots,\varphi^p_p)\quad\text{and}\quad(c_{p+1},\varphi^{p-1}_0+s_{c},\ldots,\varphi^{p-1}_{p-1}+s_c)\] that maps via the augmentation $\maptwoshort{R_{p-1}^{\arcsymbol}(\cO^{\ocolour{m}})}{\cO^{\ocolour{m}}}$ to $\mu_i$, or at least to its homotopy class. Such a path induces a homotopy between the two compositions of the square, which finishes the proof.
\end{proof}

\subsubsection*{$\circled{$3$}$}For the rightmost equivalence of \eqref{equation:zigzag}, we use the module structure $\theta$ to define the simplicial map \begin{equation}\label{equation:themapphi}\map{\Phi}{R_\bullet^{\arcsymbol}(\cO^{\ocolour{m}})\times\cM}{R_\bullet^{\arcsymbol}(\cM)}{\big((e,\varphi_0,\ldots, \varphi_p),A\big)}{\big(\theta(e,A,1^{g(e)}),(\varphi_0+s_A,l(\varphi_0+s_A)),\ldots,(\varphi_0+s_A,l(\varphi_p+s_A))\big),}\end{equation} using the embedding $[0,\infty)\times(-1,1)\times{\{0\}^{d-2}}\subseteq [0,\infty)\times \partial W$, the section $\mapwoshort{l}{[0,\infty)\times(-1,1)^{d-1}}{E}$, and the translation in the $[0,\infty)$-coordinate, as illustrated in Figure~\ref{figure:actionconfiguration}. This yields simplicial maps for $k \ge0$, 
\begin{equation}\label{equation:includingarcs}\mapnoname{U\cO^{\arcsymbol}_{\bullet,{k}}\times B_k(\cM)}{R^{\arcsymbol}_\bullet(\cM)^{\fib}}{\big((e,\varphi_0,\ldots,\varphi_p,\mu),(A,\zeta)\big)}{\big(\Phi((e,\varphi_0,\ldots,\varphi_p),A),\zeta\cdot\theta(\mu;A,X^{k+1})\big),}\end{equation} which 
induce a morphism $\tB\big(U\cO^{\arcsymbol}_{\bullet,{\smallsquare}},U\cO,B_{\smallsquare}(\cM)\big)\rightarrow R^{\arcsymbol}_\bullet(\cM)^{\fib}$, since they equalise the diagram 
\begin{equation*}\begin{tikzcd}[column sep=1cm]
\coprod_{f\in U\cO([k],[l])}U\cO^{\arcsymbol}_{\bullet,{k}}\times B_l(\cM)\arrow[r, shift left,"\id\times f^*"]\arrow[r, shift right,"f_*\times\id",swap]&\coprod_{[k]}U\cO^{\arcsymbol}_{\bullet,{k}}\times B_k(\cM).
\end{tikzcd}
\end{equation*} This explains the morphism $\circled{$3$}$, which is a weak equivalence by the following lemma that completes the proof of \cref{theorem:highconnectivityconfigurationspaces}, as the morphisms $\circled{$1$}$--$\circled{$3$}$ are all compatible with the augmentation to $\cM$.
\begin{lem}\label{lemma:thirdequivalence}The morphism $\circled{$3$}$ is a weak equivalence.
\end{lem}
\begin{proof} On $p$-simplices, the weak equivalence $\circled{$1$}$ and the morphism $\circled{$3$}$ fit into a commutative square
\begin{center}
\begin{tikzcd}[row sep=0.5cm]
\tB\big(U\cO([p],{\smallsquare}),U\cO,B_{\smallsquare}(\cM)\big)\arrow[r,"\textcircled{$1$}"]\arrow[d,"\simeq",swap]& R_p(\cM)\arrow[d,"\simeq"]\\
\tB\big(U\cO^{\arcsymbol}_{p,{\smallsquare}},U\cO,B_{\smallsquare}(\cM)\big)\arrow[r,"\textcircled{$3$}"]& R^{\arcsymbol}_p(\cM)^{\fib},
\end{tikzcd}
\end{center} in which the left morphism is induced by the equivalence \eqref{equation:inducedmorphism} which has the form $\maptwoshort{U\cO([p],{\smallsquare})}{U\cO^{\arcsymbol}_{p,{\smallsquare}}}$ and was defined via the action of $U\cO^{\arcsymbol}_{\bullet,{\smallsquare}}$ on a certain element $(c_{p+1},\varphi^p,\const_{c_{p+1}})\in U\cO^{\arcsymbol}_{p,{p}}$. The right map is induced by the same element, but using \eqref{equation:includingarcs}, and it is a weak equivalence by an analogous argument as for \eqref{equation:inducedmorphism}. Consequently, the arrow $\circled{$3$}$ is a weak equivalence as well and we conclude the assertion.
\end{proof}

\begin{rem}As demonstrated in the previous proof, the right vertical arrow of the preceding square is a weak equivalence for all $p\ge0$. However, it does not form a (strict) morphism of $\semisimpthick$-spaces between $R_\bullet(\cM)$ and $R^{\arcsymbol}_\bullet(\cM)^{\fib}$. There is an alternative proof of \cref{theorem:highconnectivityconfigurationspaces} by showing that these weak equivalences can be enhanced to a morphism up to higher coherent homotopy between $R_\bullet(\cM)$ and $R^{\arcsymbol}_\bullet(\cM)^{\fib}$.
\end{rem}

\begin{figure}[h]
\centering{
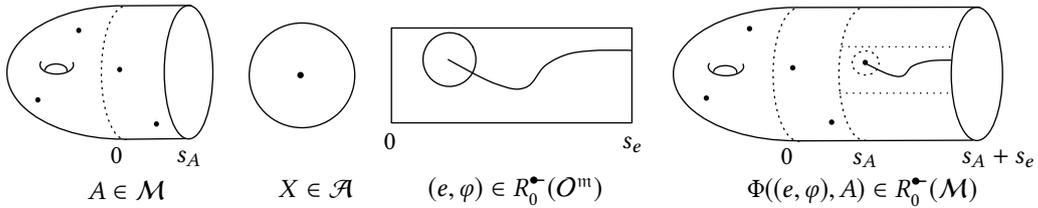}
\caption{The resolution by arcs and the map $\Phi$}
\label{figure:actionconfiguration}
\end{figure}

\addtocontents{toc}{\protect\setcounter{tocdepth}{3}}

\subsection{Coefficient systems for configuration spaces}\label{section:coefficientsystemsconfigurationspaces}Recall from \cref{section:modules} that the $\tE_1$-module structure on $\cM=\coprod_{n\ge 0}\pConf^\pi_n(W)$ over $\cA=\coprod_{n\ge0}\Conf_n(D^d)$ induces a right-module structure $\oplus$ on the fundamental groupoid $\Pi(\cM)$ over the braided monoidal category $(\Pi(\cA),\oplus,\braiding,0)$ and hence, after fixing a stabilising object $X\in\Conf_1(D^n)$, also one over the free braided monoidal category $\cB$ on one object. Denoting by $A\in\Conf_0^{\pi}(W)$ the empty configuration, a coefficient system for $\coprod_{n\ge 0}\pConf^\pi_n(W)$ is by \cref{remark:coefficientsystemsasmodules} specified by 
\begin{enumerate}
\item a $\pi_1(\pConf^\pi_n(W),A\oplus X^{\oplus n})$-module $M_n$ for each $n\ge 0$, together with 
\item $(-\oplus X)$-equivariant morphisms $\mapwoshort{\sigma}{M_n}{M_{n+1}}$ such that $B_m$ acts via $(A\oplus X^{\oplus n}\oplus-)$ trivially on the image of $\mapwoshort{\sigma^m}{M_n}{M_{n+m}}$. 
\end{enumerate}
Equivalently, a coefficient system is an abelian group-valued functor on Quillen's bracket construction \[\cC^\pi(W)\coloneqq\textstyle{\langle\coprod_{n\ge0}\pi_1(\pConf^\pi_n(W)),\cB\rangle},\] compare \cref{remark:quillencoefficients}. Using the ordering of $A\oplus X^{\oplus n}$ induced by the $[0,\infty)$-coordinate, a loop $\gamma$ in $\pConf^\pi_n(W)$ induces a permutation in $n$ letters, as well as $n$ ordered loops in $E$ by connecting the  paths in $E$ forming $\gamma$ to a basepoint in $E$ via paths in the image of the section $\mapwoshort{l}{[0,\infty)\times(-1,1)^{d-1}}{E}$. This induces a morphism \begin{equation}\label{equation:fundamentalgroupsconfig}\maptwo{\pi_1(\pConf^\pi_n(W))}{\pi_1(E)\wr\Sigma_n}\end{equation} to the wreath product, which we use to relate $\cC^\pi(W)$ to other categories via a commutative diagram
\begin{equation}\label{diagram:coefficientsystems}
\begin{tikzcd}[row sep=0.4cm]
\cC^\pi(W)\arrow[d,]\arrow[r]&\langle\pi_1(E)\wr\Sigma,\Sigma\rangle\arrow[d]&\\
\cB^\pi(W)\arrow[d,description,phantom,"\rotatebox{270}{$\subseteq$}"]\arrow[r]&\FI_{\pi_1(E)}\arrow[d,description,phantom,"\rotatebox{270}{$\subseteq$}"]\arrow[r]&\FI\arrow[d,description,phantom,"\rotatebox{270}{$\subseteq$}"]\\
\cB^\pi(W)^\sharp\arrow[r]&\FI^\sharp_{\pi_1(E)}\arrow[r]&\FI^\sharp
\end{tikzcd}
\end{equation}
on which we elaborate in the following.

The category $\langle\pi_1(E)\wr\Sigma,\Sigma\rangle$ results from the action of $\Sigma=\coprod_{n\ge0}\Sigma_n$ on $\pi_1(E)\wr\Sigma=\coprod_{n\ge0} \pi_1(E)\wr \Sigma_n$. It receives a functor from $\cC^\pi(W)$, induced by the morphisms \eqref{equation:fundamentalgroupsconfig}. The category $\FI_{\pi_1(E)}$ of finite sets and injective $\pi_1(E)$-maps \cite{Casto,GanLi,Ramos,SamSnowden} is isomorphic to $\langle\pi_1(E)\wr\Sigma,\pi_1(E)\wr\Sigma\rangle$, so is the target of a functor from $\langle\pi_1(E)\wr\Sigma,\Sigma\rangle$, induced by the inclusion $\Sigma\subseteq\pi_1(E)\wr\Sigma$. By forgetting $\pi_1(E)$, the category $\FI_{\pi_1(E)}$ maps to the category $\FI$ of finite sets and injections, on which functors are studied in the context of representation stability (see e.g.~\cite{ChurchEllenbergFarb,ChurchEllenbergFarbNagpal}). Both $\FI$ and $\FI_{\pi_1(E)}$ are subcategories of larger categories $\FI^\sharp$ and $\FI_{\pi_1(E)}^\sharp$ of partially defined ($\pi_1(E)$-)-injections \cite{ChurchEllenbergFarb,SamSnowden}. The category of \emph{partial braids} $\cB^\pi(W)^\sharp$ has the nonnegative integers as its objects and a morphism from $n$ to $m$ is a pair $(k,\mu)$ with $k\le\text{min}(n,m)$ and $\mu$ a morphism in $\Pi(\pConf_k^\pi(W))$ from a subset of $A\oplus X^{\oplus n}$ to one of $A\oplus X^{\oplus m}$. For trivial $\pi$, the category $\cB^\pi(W)^\sharp$ was studied by Palmer \cite{Palmer}, who also introduced the subcategory $\cB^\pi(W)\subseteq\cB^\pi(W)^\sharp$ of \emph{full braids}, consisting of morphisms $\mapwoshort{(k,\mu)}{n}{m}$ with $k=n$. There is a functor $\maptwoshort{\cC^\pi(W)}{\cB^\pi(W)}$ which is the identity on objects and maps a morphism \[[\gamma]\in\cC^\pi(W)(n,m)=\pi_1(\pConf^\pi_m(W),A\oplus X^{\oplus m})/B_{m-n}\] to the path in $\pConf_n^\pi(W)$ that forms the first $n$ paths in $E$ of $\gamma$, i.e.~the ones starting at $A\oplus X^{\oplus n}\subseteq A\oplus X^{\oplus m}$. For $W=D^2$ and $\pi=\id_{D^2}$, the category $\cB^\pi(W)$ was considered by Schlichtkrull and Solberg \cite{SchlichtkrullSolberg}.

\begin{rem}If $W$ is of dimension $d\ge3$, then the morphisms \eqref{equation:fundamentalgroupsconfig} are isomorphisms \cite[Lem.\,4.1]{Tillmann}, from which it follows that the three left horizontal functors in the diagram \eqref{diagram:coefficientsystems} are isomorphisms. If $E$ is in addition simply connected, then all functors except for the lower vertical inclusions are isomorphisms. 
\end{rem}

We call an abelian group valued functor on a category $\cC$ of the diagram \eqref{diagram:coefficientsystems} a \emph{coefficient system} on $\cC$. There is a notion of being of \emph{(split) degree} $r$ at an integer $N$ for coefficient systems on any of the categories $\cC$, defined analogously to \cref{definition:degree} by using an endofunctor $\Sigma$ on $\cC$ together with a natural transformation $\mapwoshort{\sigma}{\id}{\Sigma}$, similar to $\cC^\pi(W)$ (see \cref{remark:quillencoefficients}).  Most categories of the diagram are of the form $\langle\cN,\cG\rangle$ for a braided monoidal groupoid $\cG$ acting on a category $\cN$ and for such, $\Sigma$ and $\sigma$ are defined as in \cref{remark:quillencoefficients}. For $\cB^\pi(W)^\sharp$, the functor $\Sigma$ maps a morphism $(k,\mu)$ to $(k+1,s(\mu))$ using the stabilisation, and $\sigma$ consists of the constant paths at $A\oplus X^{\oplus n}$. For $\cB^{\pi}(W)$, we obtain $\Sigma$ and $\sigma$ by restriction from $\cB^\pi(W)^\sharp$. For $\FI^\sharp$ and $\FI_{\pi_1(E)}^\sharp$, the definition is analogous. Note that the morphisms $\sigma$ of the categories with a $\sharp$-superscript admit left-inverses, which results in all coefficient systems on them being split.

As all functors in the diagram are compatible with $\Sigma$ and $\sigma$, the property of being of (split) degree $r$ at $N$ is preserved by pulling back coefficient systems along them. In conclusion, by pulling back to $\cC^\pi(W)$, all coefficient systems of finite degree on any of the categories in the diagram induce coefficient systems for which the homology of $\pConf_n^\pi(W)$ stabilises by \cref{theorem:configurationspaces}. The degree of coefficient systems on some of the categories has been examined before, providing us with a wealth of examples.

\begin{ex}\label{examples:finitedegreesystemsconf}
\begin{enumerate}
\item In \cite{RWW}, the (split) degree of coefficient systems on \emph{prebraided monoidal categories} was introduced. This includes $\langle\pi_1(E)\wr\Sigma,\Sigma\rangle$, $\FI_{\pi_1(E)},\FI_{\pi_1(E)}^{\sharp},\FI$, and $\FI^{\sharp}$.

\item A \emph{finitely generated} coefficient system $F$ on $\FI_{\pi_1(E)}$ in the sense of \cite{SamSnowden} is of finite degree, provided that $\pi_1(E)$ is finite (see \cite[Prop.\,3.4.2]{SamSnowden}). By \cite[Rem.\,3.4.3]{SamSnowden}, this implication remains valid if $\pi_1(E)$ is virtually polycyclic (see the introduction for a definition) and even holds for arbitrary $\pi_1(E)$ if $F$ is \emph{presented in finite degree} or if $F$ extends to $\FI_{\pi_1(E)}^\sharp$.

\item More quantitatively, coefficient systems on $\FI$ that are \emph{generated in degree} $\le k$ and \emph{related in degree} $\le d$, as defined in \cite[Def.\,4.1]{ChurchEllenberg}, are of degree $k$ at $d+\text{min}(k,d)$ by \cite[Prop.\,4.18]{RWW}.

\item The degree of a coefficient system on $\cB^\pi(W)^\sharp$ has been studied by Palmer \cite{Palmer}, who also provides examples of finite degree coefficient systems on $\FI^\sharp$ (see \cite[Sect.\,4]{Palmer}). Note that the degree and the split degree of coefficient systems on these categories coincide.

\item For $W=D^2$ and $\pi=\id_{D^2}$, the category $\cC^\pi(W)$ is isomorphic to the category $\UB$ as recalled in \cref{definition:UB}. The \emph{Burau representation} gives rise to an example of a coefficient system of degree $1$ at $0$ on $\UB$ \cite[Ex.\,3.14]{RWW}. On the basis of this example, Soulié \cite{Soulie} has constructed coefficient systems on $\UB$ of arbitrary degree, using the so-called \emph{Long-Moody construction}.
\end{enumerate}
\end{ex}

\begin{rem}\label{remark:comparisonwithpalmer}Inspired by work of Betley \cite{Betley}, Palmer \cite{Palmer} proved homological stability for $\Conf^\pi_n(W)$ for trivial fibrations $\pi$ and coefficient systems of finite degree on $\cB^\pi(W)^\sharp$. His a surjectivity range agreeing with ours, but his result includes split injectivity in all degrees---a phenomenon special to configuration spaces and not captured by our general approach. In Remark 1.13, Palmer suspects stability for coefficient systems of finite degree on $\cB^\pi(W)$. \cref{theorem:configurationspaces} confirms this and extends his result to a larger class of coefficient systems and nontrivial labels.
\end{rem}

\subsection{Applications}\label{section:applicationsconfigurationspaces}
We complete the proofs of Corollaries~\ref{corollary:configurationspacesofdiscs} and~\ref{corollary:representationstability} sketched in the introduction. Unless stated otherwise, $W$ denotes a manifold satisfying the assumptions of \cref{theorem:configurationspaces}.

\subsubsection{Configuration spaces of embedded discs}\label{section:embeddeddiscs}
Recall from the introduction the configuration spaces of (un)ordered $k$-discs $\Conf^k_n(W)$ and $\FConf^k_n(W)$ of $W$, the related subgroups $\PDiff^k_{\partial,n}(W)\subseteq \Diff^k_{\partial,n}(W)\subseteq\Diff_\partial(W)$ of diffeomorphisms fixing or permuting $n$ chosen $k$-discs in $W$, respectively, and the orientation-preserving variants denoted with a (+)-superscript for $k=d$ and oriented $W$. The action of $\Diff_\partial(W)$ on $\Conf^{\pi_k}_n(W)$ extends to one on $\coprod_{n\ge0}\pConf^{\pi_k}_n(W)$ by extending diffeomorphisms of $W$ to $\widetilde{W}$ via the identity. This action commutes with the $\tE_d$-action of $\coprod_{n}\Conf_n(D^d)$, so the Borel construction $E\Diff_\partial(W)\times_{\Diff_\partial(W)}\cM$ inherits a graded $\tE_1$-module structure whose canonical resolution is highly-connected by \cref{example:groupaction}. Consequently, Theorems~\ref{theorem:constant} and~\ref{theorem:twisted} imply (twisted) stability for $E\Diff_\partial(W)\times_{\Diff_\partial(W)}\pConf^{\pi_k}_n(W)$ for $k<d$ and, as the equivalence $\Conf^{k}_n(W)\to \Conf^{\pi_k}_n(W)\subseteq \pConf^{\pi_k}_n(W)$ (see the introduction for the first map) is equivariant, also for $E\Diff_\partial(W)\times_{\Diff_\partial(W)}\Conf^{k}_n(W)$. The same argument applies to $E\Diff^{\plus}_\partial(W)\times_{\Diff^{\plus}_\partial(W)}\Conf^{d^{\plus}}_n(W)$. As announced in the introduction, we identify these homotopy quotients with classifying spaces of certain diffeomorphism groups. This proves \cref{corollary:configurationspacesofdiscs}.

\begin{lem}\label{lemma:diffeospermutingdiscs}For $k<d$, the Borel constructions $E\Diff_\partial(W)\times_{\Diff_\partial(W)}\FConf^k_n(W)$ and $E\Diff_\partial(W)\times_{\Diff_\partial(W)}\Conf^k_n(W)$ are models for the classifying spaces $\tB \PDiff^k_{\partial,n}(W)$ and $\tB \Diff^k_{\partial,n}(W)$, respectively. For $k=d$ and $W$ being oriented, the analogue identifications for the variants with (+)-superscripts hold.\end{lem}

\begin{proof}It suffices to show that $\Diff_\partial(W)$ acts transitively on $\FConf^{k}_n(W)$ and $\Conf^{k}_n(W)$, since the stabilisers of these actions are precisely the subgroups $\PDiff^k_{\partial,n}(W)$ and $\Diff^k_{\partial,n}(W)$, respectively. The required transitivity follows from the fact that the map $\Diff_\partial(W)\rightarrow \Emb(\coprod^nD^k,W\backslash\partial W)$, given by acting on $n$ fixed disjoint parametrised $k$-discs, is by \cite{Palais} a fibre bundle with path-connected base space $\Emb(\coprod^nD^k,W\backslash\partial W)\simeq \FConf_n^{\pi_k}(W)$. This same argument applies to $\PDiff^{d^{\plus}}_{\partial,n}(W)$ and $\Diff^{d^{\plus}}_{\partial,n}(W)$ by using orientation preserving diffeomorphisms and embeddings, as the fibre of the bundle $\pi_d^{\plus}$ of oriented $d$-frames is path-connected.
\end{proof}

\subsubsection{Representation stability}
\label{section:representationstability} We prove \cref{corollary:representationstability}, using the notation of the introduction.

\begin{lem}\label{lemma:multiplicityistwistedhomology}Let $W$ and $\pi$ be as in \cref{theorem:configurationspaces} and $\lambda \vdash n$ a partition. The $V_\lambda$-multiplicity in $\oH^i(\FConf^\pi_n(W);\bfQ)$ is the dimension of $\oH_i(\Conf^\pi_n(W);V_\lambda)$, where $\pi_1(\Conf^\pi_n(W))$ acts on $V_\lambda$ via the morphism $\maptwoshort{\pi_1(\Conf^\pi_n(W))}{\Sigma_n}$.
\end{lem}
\begin{proof}
Delooping the covering space 
$\Sigma_n\rightarrow\FConf^\pi_n(W)\rightarrow \Conf^\pi_n(W)$ once results in a fibration sequence with base space $\tB \Sigma_n$. We consider the induced Serre spectral sequence, twisted by the local system $V_\lambda$ on $\tB \Sigma_n$,
\[E_{p,q}^2\cong \oH_{p}\big(\tB \Sigma_n;\oH_q(\FConf^\pi_n(W);V_\lambda)\big)\implies \oH_{p+q}\big(\Conf^\pi_n(W);V_\lambda\big).\] Since the action of $\pi_1(\FConf^\pi_n(W))$ on $V_\lambda$ is trivial, we conclude \[\oH_{p}\big(\tB \Sigma_n;\oH_q(\FConf^\pi_n(W);V_\lambda)\big)\cong \oH_{p}\big(\tB \Sigma_n;\oH_q(\FConf^\pi_n(W);\bfQ)\otimes V_\lambda\big).\] These groups vanishes for $p\neq 0$ as $\Sigma_n$ has no rational cohomology in positive degree. Hence, the $E_2$-page is trivial, except for the $0$th column, which is isomorphic to the coinvariants $(\oH_q(\FConf^\pi_n(W);\bfQ)\otimes V_\lambda)_{\Sigma_n}$, which are in turn isomorphic to the invariants $(\oH_q(\FConf^\pi_n(W);\bfQ)\otimes V_\lambda)^{\Sigma_n}$. As a result of this, the spectral sequence collapses and we can identify $\oH_{q}(\Conf^\pi_n(W);V_\lambda)$ with $(\oH_q(\FConf^\pi_n(W);\bfQ)\otimes V_\lambda)^{\Sigma_n}$, whose dimension equals the $V_\lambda$-multiplicity in $\oH_q(\FConf^\pi_n(W);\bfQ)$, since $V_\mu\otimes V_\lambda$ for a partition $\mu\vdash n$ contains a trivial representation if and only if $\mu=\lambda$ and, in that case, it is $1$-dimensional (see \cite[Ex.\,4.51]{FultonHarris}). This proves the claim, because the $V_\lambda$-multiplicity in $\oH^i(\FConf^\pi_n(W);\bfQ)$ equals the one in $\oH_i(\FConf^\pi_n(W);\bfQ)$ by the universal coefficient theorem.
\end{proof}

\begin{cor}\label{corollary:representationstabilityconfig}For $W$ and $\pi$ as in \cref{theorem:configurationspaces}, the $V_{\lambda[n]}$-multiplicity in $\oH^i(\FConf^\pi_n(W);\bfQ)$ is independent of $n$ for $n$ large relative to $i$.
\end{cor}
\begin{proof}
By \cite[Prop.\,3.4.1]{ChurchEllenbergFarb}, the $\Sigma_n$-representations $V_{\lambda[n]}$ assemble into a finitely generated $\FI$-module $V(\lambda)$ with $V(\lambda)_n\cong V_{\lambda[n]}$, which pulls back along $\maptwoshort{\cC^\pi(W)}{\FI}$ of \eqref{diagram:coefficientsystems} to a coefficient system of finite degree for $\coprod_{n\ge0}\pConf_n(M)$ by \cref{examples:finitedegreesystemsconf} ii). Combining \cref{theorem:twisted} with \cref{lemma:multiplicityistwistedhomology} gives the claim.
\end{proof}

\begin{proof}[Proof of \cref{corollary:representationstability}] \cref{corollary:representationstabilityconfig} settles the statement for $\FConf^\pi_n(W)$. To derive the claim about $\FConf^k_n(W)$, observe that the equivalence $\maptwoshort{\Conf^k_n(W)}{\Conf_n^{\pi_k}(W)}$ discussed in the introduction is covered by a $\Sigma_n$-equivariant equivalence $\maptwoshort{\FConf^k_n(W)}{\FConf^{\pi_k}_n(W)}$, so we have $\oH^i(\FConf^k_n(W);\bfQ)\cong \oH^i(\FConf^{\pi_k}_n(W);\bfQ)$ as $\Sigma_n$-modules. The remaining part concerning $\tB \PDiff^k_{\partial,n}(W)$ is shown by using the model $\tB \PDiff^k_{\partial,n}(W)\simeq E\Diff_\partial(W)\times_{\Diff_\partial(W)}\FConf^k_n(W)$ provided by \cref{lemma:diffeospermutingdiscs}, and adapting the argument of Lemmas~\ref{lemma:multiplicityistwistedhomology} and~\ref{corollary:representationstabilityconfig} by replacing the covering space $\Sigma_n\rightarrow\FConf^\pi_n(W)\rightarrow \Conf^\pi_n(W)$ with \[\Sigma_n\rightarrow E\Diff_\partial(W)\times_{\Diff_\partial(W)}\FConf^k_n(W)\rightarrow E\Diff_\partial(W)\times_{\Diff_\partial(W)} \Conf^k_n(W).\] The statements about the variants $\FConf^{d^{\plus}}_n(W)$ and $\PDiff^{d^{\plus}}_n(W)$ are proved in the same way.
\end{proof}

The following ranges resulted from a discussion with Peter Patzt whom the author would like to thank.

\begin{rem}\label{remark:representationstability}
To obtain explicit ranges for \cref{corollary:representationstability}, one can show that the $\FI$-module $V(\lambda)$, used in the proof of \cref{corollary:representationstabilityconfig}, is generated in degree $|\lambda|+\lambda_1$ and related in degree $|\lambda|+\lambda_1+1$, so the corresponding coefficient system has degree $|\lambda|+\lambda_1$ at $2|\lambda|+2\lambda_1+1$ by \cref{examples:finitedegreesystemsconf} iii). Consequently, one deduces that the $V_{\lambda[n]}$-multiplicities in the cohomology groups of \cref{corollary:representationstability} are constant for $i\le\frac{n}{2}-(|\lambda|+\lambda_1+1)$. Note that our range is not uniform, i.e.~is dependent on the partition. In contrast, the range for $\oH^{i}(\FConf(W);\bfQ)$ obtained by Church \cite{Church} is $i\le\frac{n}{2}$ if the dimension is $d\ge3$ and $i\le\frac{n}{4}$ for $d=2$, at least for the manifolds $W$ to which his result applies.\end{rem}

\section{Moduli spaces of manifolds}\label{section:modulispacesofmanifolds}
Throughout the section, we fix a closed manifold $P$ of dimension $(d-1)$, together with an embedding \[P\subseteq \bfR^{d-1}\times \bfR^{\infty}\] which contains the open unit cube $(-1,1)^{d-1}\times\{0\}\subseteq\bfR^{d-1}\times \bfR^{\infty}$ and satisfies $P\subseteq \bfR^{d-1}\times [0,\infty)^{\infty}$. We consider compact manifolds $W$ with a specified identification $\partial W=P$ and denote by $\Diff_\partial(W)$ the group of diffeomorphisms fixing a neighbourhood of the boundary, equipped with the $C^\infty$-topology. To construct our preferred model of its classifying space, we choose a collar $\mapwoshort{c}{(-\infty,0]\times P}{W}$ and denote by $\Emb_\varepsilon(W,(-\infty,0]\times\bfR^d\times\bfR^{\infty})$ for $\varepsilon>0$ the space of embeddings $e$ satisfying $(e \circ c)(t,x)=(t,x)$ for $t\in(-\varepsilon,0]$, using the $C^\infty$-topology.

We define the \emph{moduli space of $W$-manifolds} $\cM(W)$ as the space of submanifolds \[W'\subseteq (-\infty,0]\times\bfR^{d-1}\times\bfR^{\infty}\] such that
\begin{enumerate}
\item there is an $\varepsilon>0$ with $W'\cap (-\varepsilon,0]\times\bfR^{d-1}\times\bfR^{\infty}=(-\varepsilon,0]\times P$ and
\item there is a diffeomorphism $\mapwoshort{\phi}{W}{W'}$ that satisfies 
$\phi\circ c|_{(-\varepsilon,0]\times P}=\text{inc}_{(-\varepsilon,0]\times P}$, 
\end{enumerate}
where $\text{inc}$ denotes the inclusion ensured by (i). 
The space $\cM(W)$ is topologised as the quotient of \[\Emb_\partial(W,(-\infty,0]\times\bfR^{d-1}\times\bfR^{\infty})=\colim_{\varepsilon \to 0}\Emb_\varepsilon(W,(-\infty,0]\times\bfR^{d-1}\times\bfR^{\infty})\] by the action of $\Diff_\partial(W)$ via precomposition. The space $\Emb_\partial(W,(-\infty,0]\times\bfR^{d-1}\times\bfR^{\infty})$ is weakly contractible by Whitney's embedding theorem, and as the action of $\Diff_\partial(W)$ is free and admits slices by \cite{BinzFischer}, the moduli space $\cM(W)$ provides a model for the classifying space $\tB\Diff_\partial(W)$. In the case of $P$ being the sphere $S^{d-1}$, we define a weakly equivalent variant $\cM^s(W)$ of $\cM(W)$, consisting of submanifolds \[W'\subseteq D^d\times\bfR^{\infty}\] such that 
\begin{enumerate}
\item the interior of $W'$ lies in $(D^d\backslash \partial D^d)\times(-\infty,0]^\infty$,
\item there exists an $\varepsilon>0$ for which $\mapwoshort{c'}{(-\varepsilon,0]\times S^{d-1}}{W'}$, mapping $(t,x)$ to $((1+t)x,0)$, is a collar, and  
\item there is a diffeomorphism $\mapwoshort{\phi}{W}{W'}$ satisfying 
$\phi\circ c|_{(-\varepsilon,0]\times P}=c'|_{(-\varepsilon,0]\times P}$.\end{enumerate}
We call $\textstyle{\cM=\coprod_{[W]}\cM(W)}$ the \emph{moduli space of manifolds with $P$-boundary}, the union taken over compact manifolds $W$ with an identification $\partial W=P$, one in each diffeomorphism class relative to $P$. Analogously, the \emph{moduli space of manifolds with sphere boundary} is $\cA=\coprod_{[N]}\cM^s(N)$ for $N$ with $\partial N=S^{d-1}$.

\begin{lem}\label{lemma:modulestructuremanifolds}
The moduli space $\cA$ of manifolds with sphere boundary forms an $\tE_d$-algebra with the moduli space $\cM$ of manifolds with $P$-boundary as an $\tE_1$-module over it.
\end{lem}

\begin{proof}The operad $\cD^\bullet(D^d)$ of little $d$-discs acts on $\coprod_{[N]}\cM^s(N)$ by gluing manifolds along their sphere boundary into a disc, instructed by little $d$-discs. Formally, define
\[\map{\theta}{\cD^k(D^d)\times (\coprod_{[N]}\cM^s(N))^k}{\coprod_{[N]}\cM^s(N)}{\big((\phi_1,\ldots,\phi_k),(N_1,\ldots,N_k)\big)}{\big((D^d\backslash\cup_{i=1}^k\im\phi_i)\times\{0\}\big)\cup\big(\cup_{i=1}^kr_iN_i+b_i\big),}\] where $r_i$ is the radius and $b_i$ the centre of $\mapwoshort{\phi_i}{D^d}{D^d}$, and $r_iN_i+b_i$ is obtained from $N_i$ by scaling by $r_i$ and translating by $b_i$, both in the $D^d$-coordinate. The conditions (i) and (ii) in the definition of $\cM^s(N)$ ensure that $\theta$ is well-defined. This action extends to an action of $\SC_d$ (see \cref{section:modules}) on $(\coprod_{[W]}\cM(W),\coprod_{[N]}\cM^s(N))$ via \[\textstyle{\mapwo{\theta}{\SC_d(\ocolour{m},\ocolour{a}^k,\ocolour{m})\times\coprod_{[W]}\cM(W)\times\big(\coprod_{[N]}\cM^s(N)\big)^k}{\coprod_{[W]}\cM(W)}},\] mapping $((s,\phi_1,\ldots,\phi_k),M,N_1,\ldots, N_k)$ to the submanifold obtained from \begin{equation}\label{equation:unionmanifold}M\cup\big(([0,s]\times P)\backslash (\cup_{i=1}^k \im\phi_i\times\{0\})\big)\cup\big(\cup_{i=1}^kr_iN_i+b_i\big)\end{equation} by translating in the first coordinate by $s$ to the left, where $r_iN_i+b_i$ is obtained from $N_i$ by scaling by the radius $r_i$ of $\mapwoshort{\phi_i}{D^d}{(0,1)\times(-1,1)^{d-1}}$ and translating by the centre $b_i$ of $\phi_i$, both in the $\bfR^d$-coordinate. Loosely speaking, we attach a cylinder to the boundary of $W$, glue in the $N_i$ via the little $d$-discs, and shift everything to the left, as in Figure~\ref{figure:modulemanifolds}. This yields indeed a smooth submanifold, since the threefold union \eqref{equation:unionmanifold} is one: our conditions on $P\subseteq \bfR^{d-1}\times[0,\infty)^{\infty}$ and on the manifolds $N_i\subseteq D^d\times(-\infty,0]^{\infty}$ in $\cM^s(N)$ ensure that $[0,s]\times P$ and $\cup_{i=1}^kr_iN_i+b_i$ intersect only in $(0,s)\times(-1,1)^{d-1}\times\{0\}$, which, together with properties (i)--(ii) of $\cM^s(N)$, implies that the second union of \eqref{equation:unionmanifold} is a smooth submanifold. This manifold intersects with $M$ only in $\{0\}\times P$, so the whole union \eqref{equation:unionmanifold} forms by property (i) of $\cM(W)$ and (ii) of $\cM^s(N)$ a smooth submanifold as well.
\end{proof}
\begin{figure}[h]
\centering{
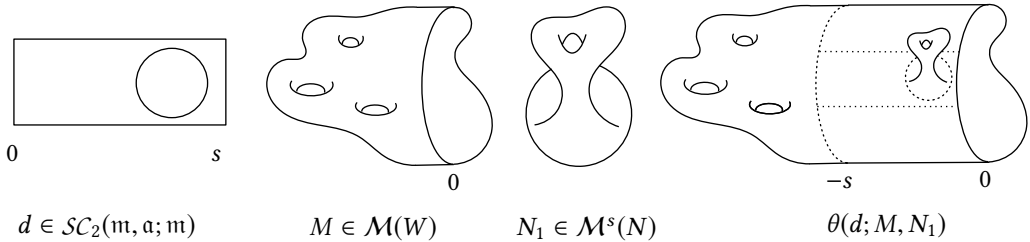}
\caption{The $\tE_1$-module structure on the moduli space of manifolds}
\label{figure:modulemanifolds}
\end{figure}

\subsection{The resolution by embeddings}\label{section:resolutionbyembeddings}
By virtue of \cref{lemma:modulestructuremanifolds}, the moduli space $\cM$ of manifolds with $P$-boundary forms in dimensions $d\ge2$ an $\tE_1$-module over the one of manifolds with sphere boundary $\cA$, considered as an $\tE_2$-algebra via the embedding $\maptwoshort{\SC_2}{\SC_d}$ of \cref{section:modules}. For $A\in\cM$ and $X\in\cA$, the stabilised manifold $A\oplus X$ is a model for the boundary connected sum $A\natural X$ of $A$ and $X$. But in contrast to the usual construction of the boundary connected sum, the manifold $A\oplus X$ contains $A$ as a canonically embedded submanifold, and the boundary of $A\oplus X$ is canonically identified with the boundary of $X$ (cf.~Figure~\ref{figure:modulemanifolds}). On components, the stabilisation takes the form $\mapwoshort{s}{\cM(A)}{\cM(A\natural X)}$, modeling the map \[\mapwo{s}{\tB\Diff_\partial(A)}{\tB\Diff_\partial(A\natural X)}\] induced by extending diffeomorphisms by the identity.

As we did for configuration spaces in \cref{section:arcresolution}, we identify the canonical resolution of $\cM$ with an augmented semi-simplicial space $R^{\stripsymbol}_\bullet(\cM)$ of geometric nature, which is a generalisation of one introduced by Galatius and Randal-Williams in \cite{GRWI}. To that end, denote by $H_X$ for $X\in\cA$ the manifold obtained from $X$ by gluing in $[-1,0]\times D^{d-1}$ along the embedding \[\mapnoname{\{-1\}\times D^{d-1}}{\partial X=S^{d-1}}{x}{(\sqrt{1-|x|},x).}\] The resulting manifolds is, after smoothing corners, diffeomorphic to $X$, but contains a canonically embedded strip $[-1,0]\times D^{d-1}\subseteq H_X$. When considering embeddings of $H_X$ into a manifold with boundary, we always implicitly require that $\{0\}\times D^d$ is sent to the boundary and the rest of $H_X$ to the interior.

\begin{dfn}Let $W$ be a $d$-manifold, equipped with an embedding $\mapwoshort{e}{(-\varepsilon,0]\times\bfR^{d-1}}{W}$ for an $\varepsilon>0$, satisfying $e^{-1}(\partial W)=\{0\}\times \bfR^{d-1}$. Define a semi-simplicial space $K^X_\bullet(W)$ with the space of $p$-simplices given by tuples $((\varphi_0,t_0),\ldots,(\varphi_p,t_p))\in(\Emb(H_X,W)\times\bfR)^{p+1}$ of embeddings with parameters, such that
\begin{enumerate}
\item the embeddings $\varphi_i$ are pairwise disjoint,
\item there exists an $\delta\in(0,\varepsilon)$ such that $\phi_i(s,p)=e(s,p+t_ie_1)$ holds for $(s,p)\in(-\delta,0]\times D^{d-1}\subseteq H_X$, where $e_1\in\bfR^{d-1}$ is the first basis vector, and
\item the parameters are ordered by $t_0<\ldots <t_p$. 
\end{enumerate}
The embedding space is topologised in the $\cC^\infty$-topology. The $i$th face map forgets $(\varphi_i,t_i)$.
\end{dfn}
For submanifolds $W\in\cM$, we use the embedding $\mapwoshort{e}{(-\varepsilon,0]\times \bfR^{d-1}}{W}$ that is obtained from the canonically embedded cube $(-\varepsilon,0]\times(-1,1)^{d-1}\subseteq (-\varepsilon,0]\times P\subseteq W$ by use of the diffeomorphism \begin{equation}\label{equation:diffeomorphismreallinecube}\mapnoname{\bfR}{(-1,1)}{x}{\frac{2}{\pi}\operatorname{arctan}(x).}\end{equation}
The group $\Diff_\partial(W)$ acts simplicially on $K^X_\bullet(W)$ by precomposition, so the levelwise Borel construction results in an augmented semi-simplicial space \begin{equation}\label{equation:partofembeddingresolution}\maptwo{\Emb_\partial(W,(-\infty,0]\times\bfR^d\times\bfR^\infty)\times_{\Diff_\partial(W)}K^X_\bullet(W)}{\cM(W)}\end{equation} in terms of which we define the \emph{resolution by embeddings} as the augmented semi-simplicial space
\[\maptwoshort{R^{\stripsymbol}_\bullet(\cM)}{\cM}\] obtained by taking coproducts of the semi-simplicial spaces \eqref{equation:partofembeddingresolution} over compact manifolds $W$ with $P$-boundary, one in each diffeomorphism class relative $P$. This is the analogue of the resolution by arcs for configuration spaces. A point in $R^{\stripsymbol}_\bullet(\cM)$ consists of a manifold $W\in\cM$ and $(p+1)$ embeddings of $H_X$ into $W$ that form an element of $K^X_p(W)$ (see the rightmost graphic of Figure~\ref{figure:manifoldresolution} for an example). Since the augmentation is by construction a levelwise fibre bundle, the resolution by embeddings is fibrant. In particular, its fibre $K^X_\bullet(A)$ at $A\in\cM$ is equivalent to the respective homotopy fibre.

\begin{thm}\label{theorem:identificationcomplexformanifolds}The canonical resolution and the resolution by embeddings are weakly equivalent as augmented $\semisimpthick$-spaces. In particular, $K^X_\bullet(A)$ for $A\in\cM$ is weakly equivalent to the space of destabilisations $W_\bullet(A)$ of $A$.\end{thm}

We closely follow the proof of \cref{theorem:highconnectivityconfigurationspaces} for configuration spaces to prove \cref{theorem:identificationcomplexformanifolds}, adopting the notation of \cref{section:arcresolution}. More specifically, we construct a zig-zag of weak equivalences
\begin{equation}\label{equation:zigzagmanifolds}R_\bullet(\cM)\xlongleftarrow{\circled{$1$}} \tB\big(U\cO(\bullet,{\smallsquare}),U\cO,B_{\smallsquare}(\cM)\big)\stackrel{\circled{$2$}}{\simeq}\tB\big(U\cO^{\stripsymbol}_{\bullet,{\smallsquare}},U\cO,B_{\smallsquare}(\cM)\big)\xlongrightarrow{\circled{$3$}} R^{\stripsymbol}_\bullet(\cM)^{\fib}\end{equation} of augmented $\semisimpthick$-spaces between the canonical resolution and the fibrant replacement of the resolution by embeddings---analogous to the one for configuration spaces, labelled by \eqref{equation:zigzag}. The first equivalence $\circled{$1$}$ of \eqref{equation:zigzag} carries over to \eqref{equation:zigzagmanifolds} verbatim. To construct $\circled{$2$}$, we replace the semi-simplicial space $R_\bullet^{\arcsymbol}(\cO^\ocolour{m})$ with an equivalent variant $R_\bullet^{\stripsymbol}(\cO^\ocolour{m})$, essentially by including a contractible choice of tubular neighbourhoods of the arcs. To this end, consider for $s>0$ the simplicial space $K^{D^d}_\bullet((0,s]\times(-1,1)^{d-1})$ for which we use the embedding $\mapwoshort{e}{(-s,0]\times\bfR^{d-1}}{(0,s]\times(-1,1)^{d-1}}$ obtained from \eqref{equation:diffeomorphismreallinecube} and the translation by $s$. Call a $0$-simplex $(\varphi,t)$ therein a \emph{little $d$-disc with thickened tether} if the restriction $\mapwoshort{\varphi|_{D^d}}{D^d}{(0,s)\times(-1,1)^{d-1}}$ is a composition of a scaling and a translation. The embedding $\mapwoshort{\varphi}{H_{D^d}}{(0,s]\times(-1,1)^{d-1}}$ induces an arc \[\mapwoshort{\varphi^{\arcsymbol}\coloneq\varphi|_{[-1,0]\times\{0\}}}{[-1,0]}{(0,s)\times(-1,1)^{d-1}},\] called the \emph{tether} of $\varphi$, which connects the little $d$-disc to the boundary. The embedding $\varphi$ furthermore induces a trivialisation of the normal bundle of the tether, which we consider as a map $\maptwoshort{[-1,0]}{V_{d-1}(\bfR^d)}$ to the space of $(d-1)$-frames in $\bfR^{d}$. We call a little $d$-disc with thickened tether $(\varphi,t)$ \emph{two-dimensional}, if
\begin{enumerate}
\item the little $d$-disc $\varphi|_{D^d}$ is the image of a little $2$-disc in $(0,s)\times(-1,1)$ under $\SC_2\rightarrow\SC_d$ (see \cref{section:modules}),
\item the induced tether $\varphi^{\arcsymbol}$ lies in the slice $(0,s)\times(-1,1)\times\{0\}^{d-2}$, and
\item the induced trivialisation $\maptwoshort{[-1,0]}{V_{d-1}(\bfR^d)}$ equals, up to scaling by a smooth function $\maptwoshort{[-1,0]}{(0,\infty)}$, the parallel transport of the frame $(e_2,\ldots,e_d)\in V_{d-1}(\bfR^d)$ at $\varphi^{\arcsymbol}(0)$ along the tether $\varphi^{\arcsymbol}$, where $e_i\in\bfR^d$ denotes the $i$th basis vector.
\end{enumerate}

\begin{dfn}Define the augmented semi-simplicial space $\maptwoshort{R^{\stripsymbol}_\bullet(\cO^\ocolour{m})}{\cO^\ocolour{m}}$ with $p$-simplices
\[R_p^{\stripsymbol}(\cO^\ocolour{m})\subseteq\cO^\ocolour{m}\times\big(\Emb(H_{D^d},(0,\infty)\times(-1,1))\times\bfR\big)^{p+1}\] consisting of
$((s,\{\phi_j\}),(\varphi_0,t_0),\ldots,(\varphi_p,t_p))$ such that $(\varphi_i,t_i)\in K^{D^d}_p((0,s]\times(-1,1)^{d-1})$ and all $(\varphi_i,t_i)$ are two-dimensional little $d$-discs with thickened tethers whose induced little $2$-disc is one of the $\phi_j$. The third graphic of Figure~\ref{figure:manifoldresolution} illustrates a $0$-simplex of this semi-simplicial space.
\end{dfn}
As a two-dimensional little $2$-disc with thickened tether is, up to a contractible choice of a thickening, determined by the associated little $2$-disc and its tether, the $\semisimpthick$-spaces $R_\bullet^{\stripsymbol}(\cO^\ocolour{m})$ and $R_\bullet^{\arcsymbol}(\cO^\ocolour{m})$ are weakly equivalent. The $(\semisimpopthick\times U\cO)$-space $U\cO^{\stripsymbol}_{\bullet,{\smallsquare}}$ in \eqref{equation:zigzagmanifolds} is defined in the same way as $U\cO^{\arcsymbol}_{\bullet,{\smallsquare}}$, but using $R^{\stripsymbol}_\bullet(\cO^\ocolour{m})$ instead of $R^{\arcsymbol}_\bullet(\cO^\ocolour{m})$. Making use of the equivalence between $R_\bullet^{\stripsymbol}(\cO^\ocolour{m})$ and $R_\bullet^{\arcsymbol}(\cO^\ocolour{m})$, the proof of \cref{lemma:frompathstoarcs} carries over to the manifold case and shows that $U\cO^{\stripsymbol}_{\bullet,{\smallsquare}}$ and $U\cO(\bullet,\smallsquare)$ weakly equivalent $(\semisimpopthick\times U\cO)$-spaces, which establishes the equivalence $\circled{$2$}$. Finally, we construct the remaining equivalence $\circled{$3$}$ via an analogue
\begin{equation}
\label{equation:simplicialactionnonfreemanifolds}\mapwo{\Phi}{\cM\times R^{\stripsymbol}_\bullet(\cO^\ocolour{m})}{R^{\stripsymbol}_\bullet(\cM)},
\end{equation}
of the simplicial map \eqref{equation:themapphi}, mapping $(A,(e,(\varphi_0,t_0),\ldots,(\varphi_p,t_p)))$ to the manifold $\theta(e;A,X^{p+1})$, equipped with the embeddings of $H_X$ obtained from the $\varphi_i$ by replacing $D^d$ by $X$ (see Figure~\ref{figure:manifoldresolution}). Using \eqref{equation:simplicialactionnonfreemanifolds} instead of \eqref{equation:themapphi} in the definition of \eqref{equation:includingarcs}, we obtain simplicial maps
$\maptwoshort{U\cO^{\stripsymbol}_{\bullet,{k}}\times B_k(\cM)}{R_\bullet^{\stripsymbol}(\cM)^{\fib}},$ which induce a morphism of the form $\tB\big(U\cO^{\stripsymbol}_{\bullet,{\smallsquare}},U\cO,B_{\smallsquare}(\cM)\big)\rightarrow R^{\stripsymbol}_\bullet(\cM)^{\fib}$, as in the case of configuration spaces. This is the last morphism $\circled{$3$}$ in the zig-zag \eqref{equation:zigzagmanifolds}, and it is a weak equivalence by the argument of the proof of \cref{lemma:thirdequivalence}, minorly modified using the following lemma, which completes the proof of \cref{theorem:identificationcomplexformanifolds}.

\begin{lem}\label{lemma:simplicialactionmanifoldsgivesweakequivalence}For all $p\ge0$ and elements of the form $(c_{p+1},(\varphi_0,t_0\ldots,\varphi_p,t_p))\in R^{\stripsymbol}_p(\cO^\ocolour{m})$, the simplicial map $\Phi$ induces a weak equivalence $\maptwoshort{\cM}{R_p^{\stripsymbol}(\cM)}$.
\end{lem}
\begin{proof}The line of argument of \cite[Lem.\,6.10]{GRWI} for $X=D^{2p}\sharp S^p\times S^p$ generalises verbatim.
\end{proof}

\begin{figure}[h]
\centering{
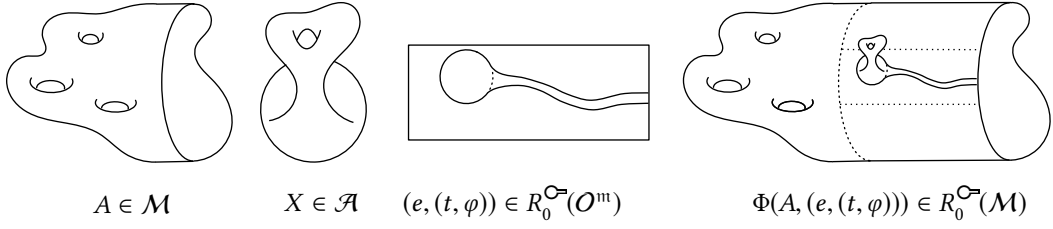}
\caption{The resolution by embeddings and the map $\Phi$}
\label{figure:manifoldresolution}
\end{figure}

Galatius and Randal-Williams \cite{GRWI} proved high-connectivity of $K_\bullet^X(A)$ if $A$ is simply connected and $X\cong D^{2p}\sharp (S^p\times S^p)$ for $p\ge3$. On the basis of this, Friedrich \cite{Friedrich} and Perlmutter \cite{Perlmutter} proved connectivity results for other choices of $A$ and $X$. To state their results, recall the stable $X$-genus $\bar{g}^X$, as introduced in \cref{section:stablegenus}, and denote by $\usr(\bfZ[G])$ for a group $G$ the \emph{unitary stable rank} \cite[Def.\,6.3]{MirzaiiVanderkallen} of its group ring $\bfZ[G]$, considered as a ring with an anti-involution.

\begin{thm}\label{theorem:spaceofdestabilisationhighlyconnectedmanifolds}The realisation of $K_\bullet^X(A)$ for a connected manifold $A\in\cM$ is
\begin{enumerate}
\item $\frac{1}{2}(\bar{g}^X(A)-4)$-connected if $X\cong D^{2p}\sharp (S^p\times S^p)$, $p\ge 3$, and $A$ is simply connected,
\item $\frac{1}{2}(\bar{g}^X(A)-\usr(\bfZ[\pi_1(A)])-3)$-connected if $X\cong D^{2p}\sharp(S^p\times S^p)$ and $p\ge 3$, and
\item $\frac{1}{2}(\bar{g}^X(A)-4-m)$-connected if $X\cong D^{p+q}\sharp(S^p\times S^q)$, $0<p<q<2p-2$, and $A$ is $(q-p+2)$-connected, where $m$ is the smallest number such that there exists an epimorphism of the form $\maptwoshort{\bfZ^m}{\pi_q(S^p)}$. 
\end{enumerate}
\end{thm}
\begin{proof}
The first two parts are \cites[Thm\,4.7]{Friedrich}[Cor.\,5.10]{GRWI}. Corollary 7.3.1 of \cite{Perlmutter} proves the third claim for the genus $g^{X}(B)$ instead of its stable variant $\bar{g}^X(B)$. However, the proof given therein goes through for $\bar{g}^X(B)$ if one replaces the relation between the genus of a manifold $B$ satisfying the assumption in (ii) and the rank of its associated Wall form (see \cite[Prop.\,6.1]{Perlmutter}) with the analogous statement relating the stable genus to the stable rank.
\end{proof}

We denote by $\bar{g}^X_A$ for a manifold $A\in\cM$ the grading of $\cM$ obtained by localising the stable $X$-genus at objects stably isomorphic to $A$ (see \cref{remark:localisation}). Combining Theorems~\ref{theorem:identificationcomplexformanifolds} and~\ref{theorem:spaceofdestabilisationhighlyconnectedmanifolds} implies the following.

\begin{cor}\label{lemma:canonicalresolutionconnectedsmanifolds}The canonical resolution $\maptwoshort{R_\bullet(\cM)}{\cM}$ is graded
\begin{enumerate}
\item $\frac{1}{2}(\bar{g}^X_A-2)$-connected for $X\cong D^{2p}\sharp (S^p\times S^p)$, $p\ge3$, and any simply-connected $A\in\cM$.
\item $\frac{1}{2}(\bar{g}^X_A-\usr(\bfZ[\pi_1(A)])-1)$-connected for $X\cong D^{2p}\sharp(S^p\times S^p)$, $p\ge3$, and any connected $A\in\cM$.
\item $\frac{1}{2}(\bar{g}^X_A-2-m)$-connected for $X\cong D^{p+q}\sharp(S^p\times S^q)$, $0<p<q<2p-2$, and any $(q-p+2)$-connected $A\in\cM$ with $m$ defined as in \cref{theorem:spaceofdestabilisationhighlyconnectedmanifolds}.
\end{enumerate}
\end{cor}

\begin{rem}\label{rem:resolutionbyembeddingssurfaces}In the case $d=2$, one can use \cite[Prop.\,5.1]{HatcherVogtmann} to show that $K_\bullet^X(A)$ is $\frac{1}{2}(\bar{g}^X(A)-3)$-connected for $X\cong D^2\sharp (S^1\times S^1)$ and $A\in\cM$ an orientable surface, which implies stability results for diffeomorphism groups of surfaces. Their homotopy discreteness \cite{EarleEells,Gramain} ensures their equivalence to their mapping class group for which stability has a longstanding history, going back to a breakthrough result by Harer \cite{Harer}, improved in manifold ways since then \cite{Boldsen,CohenMadsen,Ivanov,RWresolutions,RWW, WahlNonorientable}.
\end{rem}

By \cref{remark:improvement}, Theorems~\ref{theorem:constant} and~\ref{theorem:twisted} apply to $\cM$ when graded by $\bar{g}^X_A+2$, by $\bar{g}^X_A+\usr(\bfZ[\pi_1(A)])+1$, or by $\bar{g}^X_A+m+2$ for $X$ and $A$ as in the respective three cases of \cref{lemma:canonicalresolutionconnectedsmanifolds}. On path components, this implies \cref{theorem:manifolds}, noting that in the relevant ranges, the genus and the stable genus agree (see \cref{remark:genusisstablegenus}).

\subsection{Coefficient systems for moduli spaces of manifolds}\label{section:coefficientsystemsmanifolds}Recall from \cref{section:twistedstability} that coefficient systems for the moduli space of manifolds with $P$-boundary $\cM$ are defined in terms of the module structure of the fundamental groupoid $(\Pi(\cM),\oplus)$ over the braided monoidal category $(\Pi(\cA),\oplus, \braiding, 0)$, induced by the $\tE_1$-module structure of $\cM$ over the moduli space of manifolds with sphere boundary $\cA$. In the following, we provide an alternative description for the fundamental groupoids $\Pi(\cM)$ and $\Pi(\cA)$ that is more suitable to construct coefficient systems on $\cM$. 

Define the categories $\mcg(\cM)$ and $\mcg(\cA)$ having the same objects as $\Pi(\cM)$ and $\Pi(\cA)$, respectively, and the set of mapping classes $\pi_0(\Diff_\partial(M,N))$ as morphisms between $M$ and $N$, where $\Diff_\partial(M,N)$ is the space of diffeomorphisms that preserve a germ of the canonical collars of $M$ and $N$ ensured by condition i) in the definition of $\cM(W)$. The composition in $\mcg(\cM)$ and $\mcg(\cA)$ is the evident one.

\begin{lem}The category $\mcg(\cM)$ is canonically isomorphic to $\Pi(\cM)$, and $\mcg(\cA)$ to $\Pi(\cM)$.
\end{lem}
\begin{proof}Recall the fibre bundle from the construction of $\cM(\cA)$ in the beginning of the chapter, \[\Diff_\partial(A)\rightarrow \Emb_\partial(A,(-\infty,0]\times\bfR^d\times\bfR^{\infty})\rightarrow \cM(A).\] Lifting a path from $A$ to $B$ in $\cM$ to a path in the total space starting at the inclusion $A\subseteq(-\infty,0]\times\bfR^d\times\bfR^{\infty}$ gives a path of embeddings that ends at an embedding with image $B$ and hence provides a diffeomorphism from $A$ to $B$ by restricting to the image. This provides a functor from $\mcg(\cM)$ to $\Pi(\cM)$, whose inverse is induced by considering a diffeomorphism as an embedding, choosing a path in the contractible space $\Emb_\partial(A,(-\infty,0]\times\bfR^d\times\bfR^{\infty})$ from the inclusion $A\subseteq(-\infty,0]\times\bfR^d\times\bfR^{\infty}$ to the embedding obtained from the diffeomorphism, and mapping this path to $\cM(A)$. The argument for $\mcg(\cA)\cong \Pi(\cA)$ is analogous.
\end{proof}

The module structure of $\Pi(\cM)$ over $\Pi(\cA)$ can be transported via the identification of the preceding lemma to one of $\mcg(\cM)$ over $\mcg(\cA)$, considered as a braided monoidal category by making use of the isomorphism $\mcg(\cA)\cong\Pi(\cA)$. In concrete terms, the monoidal structure on $\mcg(\cA)$ is on objects given by the one of $\Pi(\cA)$ induced by the $\tE_2$-multiplication and on morphisms by multiplying $f\in\Diff_\partial(A,B)$ and $g\in\Diff_\partial(A',X')$ as $f\oplus g\in\Diff_\partial(A\oplus A',B\oplus B')$, defined by extending $f$ and $g$ via the identity. The description of the module structure on $\mcg(\cM)$ is analogous. Coefficient systems for $\cM$ are then given by coefficient systems for the module $\mcg(\cM)$ over $\mcg(\cA)$ in the sense of \cref{definition:coefficientsystems}.

To illustrate how this identification can be used to construct coefficient systems on $\cM$, we discuss one example in detail. Consider for $i\ge0$ the functor $\mapwoshort{\oH_i}{\mcg(\cM)}{\Ab}$ that assigns a manifold $A\in\cM$ its $i$th singular homology group $\oH_i(A)$. The inclusions $A\subseteq A\oplus X$ induce a natural transformation $\mapwoshort{\sigma^{\oH_i}}{\oH_i(-)}{\oH_i(-\oplus X)}$ that satisifies the triviality condition for coefficient systems (see \cref{definition:coefficientsystems}). To calculate the degree of $\oH_i$, we consider the commutative diagram with exact rows \begin{center}
\begin{tikzcd}[row sep=0.6cm]
0\arrow[r]& \oH_i(A)\arrow[r] \arrow[d,"{\sigma^{\oH_i}}"]&\oH_i(A\oplus X)\arrow[r]\arrow[d,"{\sigma^{\Sigma\oH_i}}"]&\widetilde{\oH}_i(\operatorname{Cone}(P)\natural X)\arrow[r]\arrow[d,"\id"]&0\\
0\arrow[r]& \oH_i(A\oplus X)\arrow[r]&\oH_i(A\oplus X\oplus X)\arrow[r]&\widetilde{\oH}_i(\operatorname{Cone}(P)\natural X)\arrow[r]&0
\end{tikzcd}
\end{center}
induced by the long exact sequence of pairs together with the equivalences $\oH_i(A\oplus X^{\oplus k},A\oplus X^{\oplus k-1})\cong\widetilde{\oH}_i(\operatorname{Cone}(P)\natural X)$ obtained by collapsing $A\oplus X^{\oplus k-1}$. The leftmost vertical map is induced by the inclusion and the second one by the inclusion followed by $A\oplus\braiding_{X,X}$. Naturality of the diagram in $A$ implies triviality of the kernel of $\oH_i$, and also that its cokernel is constant, so of degree $0$ if $\widetilde{\oH}_i(\operatorname{Cone}(P)\natural X)\neq0$ and of degree $-1$ else wise. Hence, $\oH_i$ is of degree $1$ at $0$ if $\widetilde{\oH}_i(\operatorname{Cone}(P)\natural X)\neq 0$ and of degree $0$ at $0$ if $\widetilde{\oH}_i(\operatorname{Cone}(P)\natural X)= 0$, from which \cref{corollary:homologyrepresentation} is implied by an application of \cref{theorem:manifolds}.

\subsection{Extensions}\label{section:extensionsmanifolds}
\subsubsection{Stabilisation by $(2n-1)$-connected $(4n+1)$-manifolds}
Perlmutter \cite{PerlmutterLinkingforms} established high-connectivity of the semi-simplicial spaces $K^{X}_\bullet(A)$ for $2$-connected manifolds $A$ of dimension $(4n+1)$ with $n\ge2$ and certain specific $(2n-1)$-connected stably-parallelisable manifolds $X$ with finite $\oH_{2n}(X;\bfZ)$ and trivial $\oH_{2n}(X,\bfZ/2\bfZ)$. From this, he derived homological stability with constant coefficients of \begin{equation}\label{equation:stabilisationstableparallisable}\maptwoshort{\tB\Diff_\partial(A)}{\tB\Diff_\partial(A\natural X)}\end{equation} for these specific $A$ and $X$. By using classification results of closed $(2n-1)$-connected stably parallelisable $(4n+1)$-manifolds due to Wall \cite{Wall} and De Sapio \cite{DeSapio}, he furthermore concluded that \eqref{equation:stabilisationstableparallisable} stabilises in fact for all $X$ with $X$ having the properties described above and not just the specific ones considered before. The methods of this section can be used to extend his homological stability result to abelian coefficients and coefficient systems of finite degree.

\subsubsection{Automorphisms of topological and piecewise linear manifolds}
In \cite{Kupers}, Kupers explains how one can adapt the methods of Galatius and Randal-Williams \cite{GRWI} to prove high-connectivity of the relevant semi-simplicial spaces of locally flat embeddings to prove homological stability for classifying spaces of homeomorphisms of topological manifolds and PL-automorphisms of piecewise linear manifolds. By extending the ideas of this section, our framework applies also to these examples, resulting in an extension of Kupers' stability results to coefficient systems of finite degree. 

\section{Homological stability for modules over braided monoidal categories}\label{section:modulesovercategories}
We explain the applicability of our framework to modules over braided monoidal categories, and make a comparison to the theory for braided monoidal groupoids developed by Randal-Williams--Wahl \cite{RWW}.
\subsection{$\tE_1$-modules over $\tE_2$-algebras from modules over braided monoidal categories}
\label{section:modulesoverbmcs}
Recall the categorical \emph{operad of coloured braids} $\CoB$ (see e.g.~\cite[Ch.\,5]{Fresse}). The category of $n$-operations is the groupoid $\CoB(n)$ with linear orderings of $\{1,\ldots,n\}$ as objects and braids connecting the spots as prescribed by the orderings as morphisms. The operadic composition is given by ``cabling". Algebras over $\CoB$ are exactly strict braided monoidal categories, and the topological operad obtained by taking classifying spaces is $\tE_2$ (see e.g. \cites[Thm\,5.2.12]{Fresse}[Ch.\,8]{FiedorowiczStelzerVogt}). Extending this, we construct a two-coloured operad whose algebras are modules over braided monoidal categories and whose classifying space is $\tE_{1,2}$ (see \cref{section:modules}).
\begin{dfn}
Define a categorical operad $\CoBM$ with colours $\ocolour{m}$ and $\ocolour{a}$ whose operations $\CoBM(\ocolour{m}^k,\ocolour{a}^l;\ocolour{m})$ are empty for $k\neq 1$ and equal $\CoB(l)$ otherwise. The operations $\CoBM(\ocolour{m}^k,\ocolour{a}^l;\ocolour{a})$ are empty for $k\neq 0$ and equal $\CoB(l)$ elsewise. Restricted to the $\ocolour{a}$-colour, $\CoBM$ is defined as $\CoB$. Requiring commutativity of 
\begin{equation*}
\begin{tikzcd}[row sep=0.5cm]
\CoBM(\ocolour{m},\ocolour{a}^l;\ocolour{m})\times\CoBM(\ocolour{m},\ocolour{a}^k;\ocolour{m})\times\big(\bigtimes\limits_{j=1}^l\CoBM(\ocolour{a}^{i_j};\ocolour{a})\big)\arrow[rr,"\gamma_{\CoBM}"]\arrow[d,"\tau",swap]&[-1.8cm]&[-0.6cm]
\CoBM(\ocolour{m},\ocolour{a}^{k+\sum_j i_j};\ocolour{m})\\
\CoB(k)\times\CoB(l)\times\big(\bigtimes\limits_{j=1}^l\CoB(i_j)\big)\arrow[r,"\id\times\gamma_{\CoB}"]&\CoB(k)\times\CoB(\sum_j i_j)\arrow[r,"\oplus"]&\CoB(k+\sum_j i_j)\arrow[u,equal]
\end{tikzcd}
\end{equation*}
defines the remaining composition $\gamma_\CoBM$, where $\tau$ interchanges the first two factors, $\gamma_{\CoB}$ is the composition of $\CoB$, and $\oplus$ is $\gamma_{\CoB}(\id_{\{1<2\}};-,-)$, i.e.~puts braids next to each other (see Figure~\ref{figure:braidoperad} for an example).

\end{dfn}
\begin{figure}[h]
\begin{subfigure}{.19\linewidth}
\centering
\begin{tikzpicture}[scale=0.7,fat_node/.style={circle,fill=.,draw,inner sep=0pt,minimum size=5pt]}]
\node[fat_node] (1) at (0,0){};
\node[fat_node] (2) at (0.5,0){};
\node[fat_node] (3) at (0,-2){};
\node[fat_node] (4) at (0.5,-2){};
  \draw[-] (2) to [out=-90,in=90]  (3);
   \draw[preaction={draw, line width=3pt, white}][-] (1) to [out=-90,in=90]  (4);
\end{tikzpicture}
\caption*{$d\in \CoBM(\ocolour{m},\ocolour{a}^2;\ocolour{m})$}
\end{subfigure}
\begin{subfigure}{.19\linewidth}
\centering
\begin{tikzpicture}[scale=0.7,fat_node/.style={circle,fill=.,draw,inner sep=0pt,minimum size=5pt]}]
\node[fat_node] (1) at (0,0){};
\node[fat_node] (2) at (0.5,0){};
\node[fat_node] (3) at (1,0){};
\node[fat_node] (4) at (0,-2){};
\node[fat_node] (5) at (0.5,-2){};
\node[fat_node] (6) at (1,-2){};
\draw[-] (1) to [out=-90,in=90]  (5);
\draw[-] (2) to [out=-90,in=90]  (6);
 \draw[preaction={draw, line width=3pt, white}][-] (3) to [out=-90,in=90]  (4);
\end{tikzpicture}
\caption*{$e\in \CoBM(\ocolour{m},\ocolour{a}^3;\ocolour{m})$}
\end{subfigure}
\begin{subfigure}{.16\linewidth}
\centering
\begin{tikzpicture}[scale=0.7,fat_node/.style={circle,fill=.,draw,inner sep=0pt,minimum size=5pt]}]
\node[fat_node] (1) at (0,0){};
\node[fat_node] (2) at (0.5,0){};
\node[fat_node] (3) at (0,-2){};
\node[fat_node] (4) at (0.5,-2){};
 \draw[-] (0.5,-1) to [out=-90,in=90]  (3);
 \draw[-] (2) to [out=-90,in=90]  (0,-1);
 \draw[preaction={draw, line width=3pt, white}][-] (1) to [out=-90,in=90]  (0.5,-1);
 \draw[preaction={draw, line width=3pt, white}][-] (0,-1) to [out=-90,in=90]  (4);
\end{tikzpicture}
\caption*{$f\in \CoBM(\ocolour{a}^2;\ocolour{a})$}
\end{subfigure}
\begin{subfigure}{.16\linewidth}
\centering
\begin{tikzpicture}[scale=0.7,fat_node/.style={circle,fill=.,draw,inner sep=0pt,minimum size=5pt]}]
\node[fat_node] (1) at (0,0){};
\node[fat_node] (2) at (0,-2){};
 \draw[-] (1) to [out=-90,in=90]  (2);
\end{tikzpicture}
\caption*{$g\in \CoBM(\ocolour{a}^1;\ocolour{a})$}
\end{subfigure}
\begin{subfigure}{.27\linewidth}
\centering
\begin{tikzpicture}[scale=0.7,fat_node/.style={circle,fill=.,draw,inner sep=0pt,minimum size=5pt]}]
\node[fat_node] (1) at (0,0){};
\node[fat_node] (2) at (0.5,0){};
\node[fat_node] (3) at (1,0){};
\node[fat_node] (4) at (0,-2){};
\node[fat_node] (5) at (0.5,-2){};
\node[fat_node] (6) at (1,-2){};
\draw[-] (1) to [out=-90,in=90]  (5);
\draw[-] (2) to [out=-90,in=90]  (6);
 \draw[preaction={draw, line width=3pt, white}][-] (3) to [out=-90,in=90]  (4);
 \node[fat_node] (7) at (1.5,0){};
\node[fat_node] (8) at (2,0){};
\node[fat_node] (9) at (2.5,0){};
\node[fat_node] (10) at (1.5,-2){};
\node[fat_node] (11) at (2,-2){};
\node[fat_node] (12) at (2.5,-2){};
\draw[-] (9) to [out=-90,in=90]  (10);
\draw[preaction={draw, line width=3pt, white}][-] (2.25,-1) to [out=-90,in=90]  (11);
\draw[preaction={draw, line width=3pt, white}][-] (8) to [out=-90,in=90]  (1.75,-1);
\draw[preaction={draw, line width=3pt, white}][-] (7) to [out=-90,in=90]  (2.25,-1);
\draw[preaction={draw, line width=3pt, white}][-] (1.75,-1) to [out=-90,in=90]  (12);
\end{tikzpicture}
\caption*{$\gamma(d;e,f,g)\in\CoBM(\ocolour{m},\ocolour{a}^6;\ocolour{m})$}
\end{subfigure}
\caption{The operadic composition in $\CoBM$}
\label{figure:braidoperad}
\end{figure}
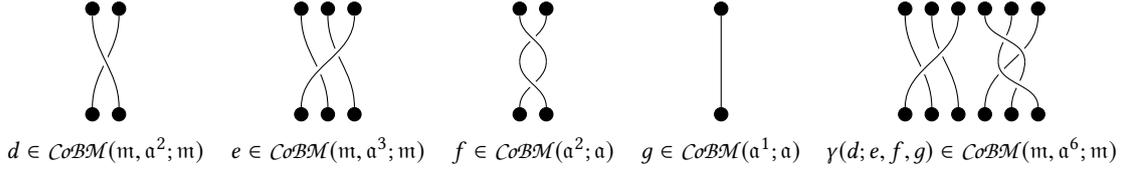
Recall the notion of a right-module $(\cM,\oplus)$ over a monoidal category $(\cA,\oplus,0)$: a category $\cM$ with a functor $\mapwoshort{\oplus}{\cM\times\cA}{\cM}$ that is unital and associative up to coherent isomorphism (see \cref{section:gradings}).

\begin{lem}\label{lemma:classifyingspacesofmodules}The structure of a (graded) $\CoBM$-algebra on a pair of categories $(\cM,\cA)$ is equivalent to a strict (graded) braided monoidal structure on $\cA$ and a strict (graded) right-module structure on $\cM$ over it. Furthermore, the topological operad obtained from $\CoBM$ by taking levelwise classifying spaces is $\tE_{1,2}$.
\end{lem}

\begin{proof} The proof of the corresponding result for $\CoB$ in \cite[Ch.\,5]{Fresse} carries over mutatis mutandis. 
\end{proof}

As a consequence of the previous lemma, the classifying space of a graded module over a braided monoidal category carries the structure of a graded $\tE_1$-module over an $\tE_2$-algebra.

\begin{rem}The operad of \emph{parenthesised coloured braids} encodes non-strict braided monoidal categories, and its classifying space operad is $\tE_2$ as well \cite[Ch.\,6]{Fresse}. By considering a parenthesised version of $\CoBM$, this extends in a similar fashion to non-strict right-modules over non-strict braided monoidal categories, whose classifying spaces hence also give $\tE_1$-modules over $\tE_2$-algebras.
\end{rem}

\subsection{Homological stability for groups and monoids}
Let $(\cM,\oplus)$ be a graded right-module over a braided monoidal category $(\cA,\oplus,\braiding,0)$ with a stabilising object $X$, i.e.~an object of $\cA$ of degree $1$. Taking classifying spaces results by \cref{lemma:classifyingspacesofmodules} in a graded $\tE_1$-module $\tB\cM$ over the $\tE_2$-algebra $\tB \cA$ with stabilising object $X\in\tB \cA$, hence provides a suitable input for Theorems~\ref{theorem:constant} and~\ref{theorem:twisted}. In the following, we introduce a condition on $\cM$ that ensures a simplification of the canonical resolution of $\tB\cM$.

\begin{dfn}The module $(\cM,\oplus)$ is called \emph{injective} at an object $A$ of $\cM$ if the stabilisation \[\mapwoshort{(-\oplus X^{\oplus p+1})}{\aut(B)}{\aut(B\oplus X^{\oplus p})}\] is injective for all objects $B$ for which $B\oplus X^{\oplus p}$ is isomorphic to $A$ for a $p\ge 0$.
\end{dfn}

\begin{dfn}\label{definition:RWWsimplicialsets}Define for an object $A$ of $\cM$ a semi-simplicial set $W_\bullet^{\RW}(A)$ with $p$-simplices given as equivalence classes of pairs $(B,f)$ of an object $B$ of $\cM$ and a morphism $f\in\cM(B\oplus X^{\oplus p+1},A)$, where $(B,f)$ and $(B',f')$ are equivalent if there is an isomorphism $g\in\cM(B,B')$ satisfying $f'\circ (g\oplus X^{\oplus p+1})=f$. The $i$th face of a $p$-simplex $[B,f]$ is defined as $[B\oplus X,f\circ (B\oplus\braiding_{X^{\oplus i},X}^{-1}\oplus X^{\oplus p-i})]$.
\end{dfn}

Recall the spaces of destabilisations $W_\bullet(A)$, i.e.~the fibres of the canonical resolution (see \cref{definition:canonicalresolution}).

\begin{lem}\label{lemma:equivalenttosimplicialsets}If $\cM$ is a groupoid, then the semi-simplicial set of path components $\pi_0(W_\bullet(A))$ for an object $A$ of $\cM$ is isomorphic to $W_\bullet^{\RW}(A)$. Moreover, $W_\bullet(A)$ is homotopy discrete if and only if $\cM$ is injective at $A$.
\end{lem}

\begin{proof}
The inclusion of the $0$-simplices $\ob\cM\subseteq \tB\cM$, together with the natural map $\maptwoshort{\mor\cM}{\Path\cM}$, induces a preferred bijection $\maptwoshort{W_p^{\RW}(A)}{\pi_0(W_p(A))}$ for all $p\ge0$, since every path in $\tB\cM$ between $0$-simplices is homotopic relative to its endpoints to a one simplex, i.e.~to a path in the image of $\maptwoshort{\mor\cM}{\Path\cM}$. By the definition of the respective face maps, these bijections assemble to an isomorphism of simplicial sets, which proves the first claim. The homotopy fibre $W_p(A)$ of the map $\mapwoshort{\tB(-\oplus X^{\oplus p+1})}{\tB\cM}{\tB\cM}$ at $A$ is homotopy discrete if and only if the induced morphisms on $\pi_1$ based at all objects $B$ with $B\oplus X^{\oplus p+1}\cong A$ for $p\ge0$ are injective, which is clearly equivalent to $\cM$ being locally injective at $A$.
\end{proof}

\begin{rem}
\label{remark:coefficientsgroupoid}If $\cM$ and $\cA$ are groupoids, then $\cM\simeq\Pi(\tB \cM)$ holds naturally as a module over $\cA\simeq\Pi(\tB\cA)$, so coefficient systems for $\tB \cM$ (see \cref{definition:coefficientsystemsE1}) are coefficient systems for $\cM$ as in \cref{section:coefficientsystems}.
\end{rem}

\begin{rem}\label{remark:discreteframework}Since the connectivity of the canonical resolution can be tested on the spaces of destabilisations $W_\bullet(A)$ (see \cref{remark:connectivityresolutionbyspacesofdestabilisation}), Lemmas~\ref{lemma:equivalenttosimplicialsets} and~\ref{remark:coefficientsgroupoid} imply a version of Theorems~\ref{theorem:constant} and~\ref{theorem:twisted} that is phrased entirely in terms of (discrete) categories and semi-simplicial sets. This provides a simplified toolkit for proving homological stability for graded modules over braided monoidal categories with a stabilising object $X$ for which the multiplication $\mapwoshort{(-\oplus X)}{\aut(B)}{\aut(B\oplus X)}$ is injective for all objects $B$ of finite degree.
\end{rem}

\subsection{Comparison with the work of Randal-Williams and Wahl} \label{section:comparisonRWW}Let $(\cG,\oplus,\braiding,0)$ be a braided monoidal groupoid. In \cite{RWW}, it is shown that, for objects $A$ and $X$ in $\cG$, the maps  \begin{equation}\label{equation:groupoidstabilisation}\mapwo{\tB(-\oplus X)}{\tB\aut_\cG(A\oplus X^{\oplus n})}{\tB\aut_\cG(A\oplus X^{\oplus n+1}})\end{equation} satisfy homological stability with constant, abelian, and a class of coefficient systems if a certain family of associated semi-simplicial sets $W_n(A,X)_\bullet$ (see \cite[Def.\,2.1]{RWW}) is sufficiently connected and $\cG$ satisfies
\begin{enumerate}
\item injectivity of the stabilisation map $\mapwoshort{(-\oplus X)}{\aut_\cG(A\oplus X^{\oplus n})}{\aut_\cG(A\oplus X^{\oplus n+1})}$ for all $n\ge 0$,
\item \emph{local cancellation at $(A,X)$}, i.e.~$Y\oplus X^{\oplus m}\cong A\oplus X^{\oplus n}$ for $Y\in\cG$ and $1\le m\le n$ implies $Y\cong A\oplus X^{\oplus m-n}$,
\item no zero-divisors, i.e.~$U\oplus V\cong 0$ implies $U\cong 0$, and
\item the unit $0$ has no nontrivial automorphisms.
\end{enumerate}
As indicated by our choice of notation, if we consider $\cG$ as a module over itself, the simplicial set $W_n(A,X)_\bullet$ of \cite{RWW} equals $W^{\RW}_\bullet(A\oplus X^{\oplus n})$ as specified in \cref{definition:RWWsimplicialsets}. To compare \cite{RWW} with our work, define the module $\cG_{A,X}=\coprod_{n\ge0}\aut_\cG(A\oplus X^{\oplus n})$ over the braided monoidal category $\cG_{X}=\coprod_{n\ge0}\aut_\cG(X^{\oplus n})$, both graded in the evident way. By Theorems~\ref{theorem:constant} and~\ref{theorem:twisted}, the maps \eqref{equation:groupoidstabilisation} stabilise homologically---without assumptions on $\cG$---if the canonical resolution of $\tB \cG_{A,X}$ is sufficiently connected, or equivalently, if the spaces of destabilisations $W_\bullet(A\oplus X^{\oplus n})$ associated to $\tB \cG_{A,X}$ are (see \cref{remark:connectivityresolutionbyspacesofdestabilisation}).

The semi-simplicial sets $W_n(A,X)_\bullet$ of \cite{RWW} are equivalent to the spaces of destabilisations $W_\bullet(A\oplus X^{\oplus n})$ of $\tB \cG_{A,X}$ if conditions (i) and (ii) hold. Indeed, assumption (ii) implies that $W_n(A,X)_\bullet$ agrees with the semi-simplicial set $W^{\RW}_\bullet(A\oplus X^{\oplus n})$ associated to $\cG_{A,X}$ and hence also with $\pi_0(W_\bullet(A\oplus X^{\oplus n}))$ by \cref{lemma:equivalenttosimplicialsets}. The first condition imposes injectivity of $\cG_{A,X}$ at all objects $A\oplus X^{\oplus n}$, which is by \cref{lemma:equivalenttosimplicialsets} equivalent to the homotopy discreteness of the space of destabilisations $W_\bullet(A\oplus X^{\oplus n})$ of $\tB\cG_{A,X}$.

Hence, if one prefers to work in a discrete setting as in \cite{RWW}, i.e.~using semi-simplicial sets, condition (i) is necessary. Condition (ii) ensures that the semi-simplicial sets of \cite{RWW} agree with our spaces of destabilisations $W_\bullet(A\oplus X^{\oplus n})$, whose high-connectivity always imply stability by Theorems~\ref{theorem:constant} and~\ref{theorem:twisted}. The last two conditions are redundant, i.e.~imposing (i) and (ii) already implies (twisted) homological stability of \eqref{equation:groupoidstabilisation} under the connectivity assumptions of \cite{RWW}. The presence of these additional assumptions in \cite{RWW} is due to their usage of Quillen's construction $\langle\cG,\cG\rangle$ since the conditions (iii) and (iv) guarantee that the automorphism groups $\aut_{\langle\cG,\cG\rangle}(A\oplus X^{\oplus n})$ and $\aut_\cG(A\oplus X^{\oplus n})$ coincide. If (i)--(iii) are satisfied and the $W_n(A,X)$ are highly-conected, then \cite{RWW} implies stability for $\aut_{\langle\cG,\cG\rangle}(A\oplus X^{\oplus n})$. Hence, in this case, high-connectivity of $W_n(A,X)$ shows stability for both $\aut_\cG(A\oplus X^{\oplus n})$ and $\aut_{\langle\cG,\cG\rangle}(A\oplus X^{\oplus n})$. The reason for this is that, although these automorphism groups might differ, their quotients $\aut(A\oplus X^n)/\aut(A\oplus X^{\oplus n-p-1})\cong W^{\RW}_p(A\oplus X^{\oplus n})$, forming the corresponding semi-simplicial sets, agree.

\begin{rem}\label{remark:improvecoefficientsystemsRWW}
The coefficient systems \cite{RWW} deals with are functors of finite degree on the subcategory $\cC_{A,X}\subseteq\langle\cG,\cG\rangle$ generated by the objects $A\oplus X^{\oplus n}$. In contrast, \cref{theorem:twisted} is applicable to functors of finite degree on $\langle\cG_{A,X},\cB\rangle$ (see Remarks~\ref{remark:quillencoefficients} and~\ref{remark:coefficientsgroupoid}). 
As the canonical functors $\maptwoshort{\cG_{A,X}}{\cG}$ and $\maptwoshort{\cB}{\cG}$ induce $\maptwoshort{\langle\cG_{A,X},\cB\rangle}{\cC_{A,X}}$, every coefficient system of \cite{RWW} gives one in ours (cf.\,Remarks~\ref{remark:RWWcoefficients} and~\ref{remark:quillencoefficients}). 
\end{rem}

\begin{rem}\label{remark:improverangesRWW}
The ranges for coefficient systems of finite degree provided by \cref{theorem:twisted} agree with the ones of \cite{RWW} in the situations in which their work is applicable.
The ranges for abelian coefficients of \cref{theorem:constant} improve the ones of \cite{RWW} marginally, and so does the surjectivity range for constant coefficients in the case $k>2$. Note that, by \cref{remark:improvement}, these ranges can in some cases be further improved.
\end{rem}

\printbibliography

\end{document}

%% file: Figures/module_configurations_newest.pdf_tex
\begingroup%
  \makeatletter%
  \providecommand\color[2][]{%
    \errmessage{(Inkscape) Color is used for the text in Inkscape, but the package 'color.sty' is not loaded}%
    \renewcommand\color[2][]{}%
  }%
  \providecommand\transparent[1]{%
    \errmessage{(Inkscape) Transparency is used (non-zero) for the text in Inkscape, but the package 'transparent.sty' is not loaded}%
    \renewcommand\transparent[1]{}%
  }%
  \providecommand\rotatebox[2]{#2}%
  \ifx\svgwidth\undefined%
    \setlength{\unitlength}{402.29321097bp}%
    \ifx\svgscale\undefined%
      \relax%
    \else%
      \setlength{\unitlength}{\unitlength * \real{\svgscale}}%
    \fi%
  \else%
    \setlength{\unitlength}{\svgwidth}%
  \fi%
  \global\let\svgwidth\undefined%
  \global\let\svgscale\undefined%
  \makeatother%
  \begin{picture}(1,0.22388508)%
    \put(0,0){\includegraphics[width=\unitlength,page=1]{Figures/module_configurations_newest.pdf}}%
    \put(-0.01557575,0.09435723){\color[rgb]{0,0,0}\makebox(0,0)[lb]{\smash{}}}%
    \put(0.20303202,0.09435723){\color[rgb]{0,0,0}\makebox(0,0)[lb]{\smash{}}}%
    \put(-0.00208897,0.07454608){\color[rgb]{0,0,0}\makebox(0,0)[lb]{\smash{$0$}}}%
    \put(0.21468081,0.07454608){\color[rgb]{0,0,0}\makebox(0,0)[lb]{\smash{$s$}}}%
    \put(0.36155876,0.04923612){\color[rgb]{0,0,0}\makebox(0,0)[lb]{\smash{$0$}}}%
    \put(0.44049309,0.05004667){\color[rgb]{0,0,0}\makebox(0,0)[lb]{\smash{$s'$}}}%
    \put(0.73593941,0.04272345){\color[rgb]{0,0,0}\makebox(0,0)[lb]{\smash{$0$}}}%
    \put(0.81676107,0.04272345){\color[rgb]{0,0,0}\makebox(0,0)[lb]{\smash{$s'$}}}%
    \put(0.9388529,0.04312875){\color[rgb]{0,0,0}\makebox(0,0)[lb]{\smash{$s+s'$}}}%
    \put(0.03215822,0.00607357){\color[rgb]{0,0,0}\makebox(0,0)[lb]{\smash{$d\in\SC_2(\ocolour{m},\ocolour{a};\ocolour{m})$}}}%
    \put(0.30472616,0.00607357){\color[rgb]{0,0,0}\makebox(0,0)[lb]{\smash{$C\in\pConf_4^\pi(W)$}}}%
    \put(0.48943589,0.0046899){\color[rgb]{0,0,0}\makebox(0,0)[lb]{\smash{$D\in\Conf_3(D^d)$}}}%
    \put(0.58075307,0.01160788){\color[rgb]{0,0,0}\makebox(0,0)[lb]{\smash{}}}%
    \put(0.72533852,0.00676528){\color[rgb]{0,0,0}\makebox(0,0)[lb]{\smash{$\theta(d;C,D)\in\pConf^\pi_7(W)$}}}%
  \end{picture}%
\endgroup%

%% file: Figures/configurations_action_newest.pdf_tex
\begingroup%
  \makeatletter%
  \providecommand\color[2][]{%
    \errmessage{(Inkscape) Color is used for the text in Inkscape, but the package 'color.sty' is not loaded}%
    \renewcommand\color[2][]{}%
  }%
  \providecommand\transparent[1]{%
    \errmessage{(Inkscape) Transparency is used (non-zero) for the text in Inkscape, but the package 'transparent.sty' is not loaded}%
    \renewcommand\transparent[1]{}%
  }%
  \providecommand\rotatebox[2]{#2}%
  \ifx\svgwidth\undefined%
    \setlength{\unitlength}{401.39838752bp}%
    \ifx\svgscale\undefined%
      \relax%
    \else%
      \setlength{\unitlength}{\unitlength * \real{\svgscale}}%
    \fi%
  \else%
    \setlength{\unitlength}{\svgwidth}%
  \fi%
  \global\let\svgwidth\undefined%
  \global\let\svgscale\undefined%
  \makeatother%
  \begin{picture}(1,0.18326368)%
    \put(0.00555938,0.07327618){\color[rgb]{0,0,0}\makebox(0,0)[lb]{\smash{}}}%
    \put(0.22465448,0.07327618){\color[rgb]{0,0,0}\makebox(0,0)[lb]{\smash{}}}%
    \put(0.07654036,0.00146461){\color[rgb]{0,0,0}\makebox(0,0)[lb]{\smash{$A\in\cM$}}}%
    \put(0.2538176,0.00278166){\color[rgb]{0,0,0}\makebox(0,0)[lb]{\smash{$X\in\cA$}}}%
    \put(0.60321757,-0.00965764){\color[rgb]{0,0,0}\makebox(0,0)[lb]{\smash{}}}%
    \put(0.6887108,0.00281633){\color[rgb]{0,0,0}\makebox(0,0)[lb]{\smash{$\Phi((e,\varphi),A)\in R^{\arcsymbol}_0(\cM)$}}}%
    \put(0.84750035,0.24925733){\color[rgb]{0,0,0}\makebox(0,0)[lb]{\smash{}}}%
    \put(0,0){\includegraphics[width=\unitlength,page=1]{Figures/configurations_action_newest.pdf}}%
    \put(0.39180785,0.00377464){\color[rgb]{0,0,0}\makebox(0,0)[lb]{\smash{$(e,\varphi)\in R^{\arcsymbol}_0(\cO^{\ocolour{m}})$}}}%
    \put(0.68008337,-0.04133596){\color[rgb]{0,0,0}\makebox(0,0)[lb]{\smash{}}}%
    \put(0,0){\includegraphics[width=\unitlength,page=2]{Figures/configurations_action_newest.pdf}}%
    \put(0.09675482,0.03521222){\color[rgb]{0,0,0}\makebox(0,0)[lb]{\smash{$0$}}}%
    \put(0.15870522,0.03589892){\color[rgb]{0,0,0}\makebox(0,0)[lb]{\smash{$s_A$}}}%
    \put(0,0){\includegraphics[width=\unitlength,page=3]{Figures/configurations_action_newest.pdf}}%
    \put(0.71979319,0.03450584){\color[rgb]{0,0,0}\makebox(0,0)[lb]{\smash{$0$}}}%
    \put(0.78679601,0.03450584){\color[rgb]{0,0,0}\makebox(0,0)[lb]{\smash{$s_A$}}}%
    \put(0.88801296,0.03484889){\color[rgb]{0,0,0}\makebox(0,0)[lb]{\smash{$s_A+s_e$}}}%
    \put(0,0){\includegraphics[width=\unitlength,page=4]{Figures/configurations_action_newest.pdf}}%
    \put(0.35061649,0.04759589){\color[rgb]{0,0,0}\makebox(0,0)[lb]{\smash{$0$}}}%
    \put(0.57277216,0.04857646){\color[rgb]{0,0,0}\makebox(0,0)[lb]{\smash{$s_e$}}}%
    \put(0.91267411,0.03120025){\color[rgb]{0,0,0}\makebox(0,0)[lb]{\smash{}}}%
    \put(0.36746966,0.36469628){\color[rgb]{0,0,0}\makebox(0,0)[lb]{\smash{}}}%
    \put(0.55051109,0.2426686){\color[rgb]{0,0,0}\makebox(0,0)[lb]{\smash{}}}%
  \end{picture}%
\endgroup%

%% file: Figures/module_manifolds_newest.pdf_tex
\begingroup%
  \makeatletter%
  \providecommand\color[2][]{%
    \errmessage{(Inkscape) Color is used for the text in Inkscape, but the package 'color.sty' is not loaded}%
    \renewcommand\color[2][]{}%
  }%
  \providecommand\transparent[1]{%
    \errmessage{(Inkscape) Transparency is used (non-zero) for the text in Inkscape, but the package 'transparent.sty' is not loaded}%
    \renewcommand\transparent[1]{}%
  }%
  \providecommand\rotatebox[2]{#2}%
  \ifx\svgwidth\undefined%
    \setlength{\unitlength}{384.83385942bp}%
    \ifx\svgscale\undefined%
      \relax%
    \else%
      \setlength{\unitlength}{\unitlength * \real{\svgscale}}%
    \fi%
  \else%
    \setlength{\unitlength}{\svgwidth}%
  \fi%
  \global\let\svgwidth\undefined%
  \global\let\svgscale\undefined%
  \makeatother%
  \begin{picture}(1,0.22909557)%
    \put(0,0){\includegraphics[width=\unitlength,page=1]{Figures/module_manifolds_newest.pdf}}%
    \put(0.01409127,0.09574521){\color[rgb]{0,0,0}\makebox(0,0)[lb]{\smash{}}}%
    \put(0.24261695,0.09574521){\color[rgb]{0,0,0}\makebox(0,0)[lb]{\smash{}}}%
    \put(0.00649443,0.07503525){\color[rgb]{0,0,0}\makebox(0,0)[lb]{\smash{$0$}}}%
    \put(0.2055701,0.07545888){\color[rgb]{0,0,0}\makebox(0,0)[lb]{\smash{$s$}}}%
    \put(0.9494004,0.04972385){\color[rgb]{0,0,0}\makebox(0,0)[lb]{\smash{$0$}}}%
    \put(0.8009212,0.04930022){\color[rgb]{0,0,0}\makebox(0,0)[lb]{\smash{$-s$}}}%
    \put(0.01817338,0.00490269){\color[rgb]{0,0,0}\makebox(0,0)[lb]{\smash{$d\in\SC_2(\ocolour{m},\ocolour{a};\ocolour{m})$}}}%
    \put(0.30085631,0.00345625){\color[rgb]{0,0,0}\makebox(0,0)[lb]{\smash{$M\in\cM(W)$}}}%
    \put(0.50110519,0.00278441){\color[rgb]{0,0,0}\makebox(0,0)[lb]{\smash{$N_1\in\cM^s(N)$}}}%
    \put(0.63747466,0.00924164){\color[rgb]{0,0,0}\makebox(0,0)[lb]{\smash{}}}%
    \put(0.43421554,0.04742141){\color[rgb]{0,0,0}\makebox(0,0)[lb]{\smash{$0$}}}%
    \put(0,0){\includegraphics[width=\unitlength,page=2]{Figures/module_manifolds_newest.pdf}}%
    \put(0.80194911,0.00567704){\color[rgb]{0,0,0}\makebox(0,0)[lb]{\smash{$\theta(d;M,N_1)$}}}%
  \end{picture}%
\endgroup%

%% file: Figures/action_manifolds_new.pdf_tex
\begingroup%
  \makeatletter%
  \providecommand\color[2][]{%
    \errmessage{(Inkscape) Color is used for the text in Inkscape, but the package 'color.sty' is not loaded}%
    \renewcommand\color[2][]{}%
  }%
  \providecommand\transparent[1]{%
    \errmessage{(Inkscape) Transparency is used (non-zero) for the text in Inkscape, but the package 'transparent.sty' is not loaded}%
    \renewcommand\transparent[1]{}%
  }%
  \providecommand\rotatebox[2]{#2}%
  \ifx\svgwidth\undefined%
    \setlength{\unitlength}{389.65586264bp}%
    \ifx\svgscale\undefined%
      \relax%
    \else%
      \setlength{\unitlength}{\unitlength * \real{\svgscale}}%
    \fi%
  \else%
    \setlength{\unitlength}{\svgwidth}%
  \fi%
  \global\let\svgwidth\undefined%
  \global\let\svgscale\undefined%
  \makeatother%
  \begin{picture}(1,0.20480593)%
    \put(0,0){\includegraphics[width=\unitlength,page=1]{Figures/action_manifolds_new.pdf}}%
    \put(0.01119344,0.0754844){\color[rgb]{0,0,0}\makebox(0,0)[lb]{\smash{}}}%
    \put(0.23689111,0.0754844){\color[rgb]{0,0,0}\makebox(0,0)[lb]{\smash{}}}%
    \put(0.08431347,0.00150874){\color[rgb]{0,0,0}\makebox(0,0)[lb]{\smash{$A\in\cM$}}}%
    \put(0.26693308,0.00286549){\color[rgb]{0,0,0}\makebox(0,0)[lb]{\smash{$X\in\cA$}}}%
    \put(0.62686244,-0.00994868){\color[rgb]{0,0,0}\makebox(0,0)[lb]{\smash{}}}%
    \put(0,0){\includegraphics[width=\unitlength,page=2]{Figures/action_manifolds_new.pdf}}%
    \put(0.71493206,0.0029012){\color[rgb]{0,0,0}\makebox(0,0)[lb]{\smash{$\Phi(A,(e,(t,\varphi)))\in R^{\stripsymbol}_0(\cM)$}}}%
    \put(0,0){\includegraphics[width=\unitlength,page=3]{Figures/action_manifolds_new.pdf}}%
    \put(0.87850683,0.25676886){\color[rgb]{0,0,0}\makebox(0,0)[lb]{\smash{}}}%
    \put(0,0){\includegraphics[width=\unitlength,page=4]{Figures/action_manifolds_new.pdf}}%
    \put(0.37978945,0.00287827){\color[rgb]{0,0,0}\makebox(0,0)[lb]{\smash{$(e,(t,\varphi))\in R^{\stripsymbol}_0(\cO^\ocolour{m})$}}}%
    \put(0.70604464,-0.04258164){\color[rgb]{0,0,0}\makebox(0,0)[lb]{\smash{}}}%
  \end{picture}%
\endgroup%